\documentclass{siamart220329}
\usepackage{indentfirst}
\usepackage{graphicx}
\usepackage{epstopdf}
\usepackage{epsfig}
\usepackage{overpic}
\usepackage{array}
\usepackage{color}
\usepackage{multirow}
\usepackage{float}
\usepackage{subfigure}
\usepackage{amssymb}
\usepackage[mathscr]{euscript}
\usepackage{latexsym,bm}
\usepackage{booktabs}
\usepackage{bbm}
\usepackage{diagbox}
 \usepackage[all]{xy}
 \usepackage{color}
\usepackage{diagbox}
\usepackage{algorithm}  
\usepackage{algorithmicx} 
\usepackage{algpseudocode} 
\usepackage{mathtools} 

\newcommand\comment[1]{}

{\theoremstyle{remark} \newtheorem{remark}{\textup{\textbf{Remark}}}}
\def\bx{\bm{x}}
\def\by{\bm{x}}

\def\by{\bm{y}}

\def\D{\mathrm{d}}

\newcommand{\be}{\begin{equation}}
\newcommand{\ee}{\end{equation}}
\newcommand{\ba}{\begin{aligned}}
\newcommand{\ea}{\end{aligned}}
\newcommand{\nexp}{N_{\textup{exp}}}

\makeatletter
\@addtoreset{equation}{section}
\makeatother
\renewcommand{\theequation}{\arabic{section}.\arabic{equation}}

\begin{document}
\bibliographystyle{unsrt}

\title{Compressing the memory variables in constant-Q viscoelastic wave propagation via an improved sum-of-exponentials approximation}
\author{
Xu Guo\footnotemark[1]
\and
Shidong Jiang\footnotemark[2]
\and
Yunfeng Xiong\footnotemark[3] \footnotemark[5]
\and
Jiwei Zhang\footnotemark[4]
}
\renewcommand{\thefootnote}{\fnsymbol{footnote}}
\footnotetext[1]{Geotechnical and Structural Engineering Center, Shandong University, Jinan, Shandong 250061, China
(\email{guoxu@sdu.edu.cn})}
\footnotetext[2]{Center for Computational Mathematics, Flatiron Institute, Simons Foundation,
  New York, New York 10010 (\email{sjiang@flatironinstitute.org}).}
  \footnotetext[3]{School of Mathematical Sciences, Beijing Normal University, Beijing 100091, China
(\email{yfxiong@bnu.edu.cn}).}
  \footnotetext[4]{School of Mathematics and Statistics, and Hubei Key Laboratory of Computational Science, Wuhan University, Wuhan 430072, China.
(\email{jiweizhang@whu.edu.cn}).}
\footnotetext[5]{To whom correspondence should be addressed.}
\date{\today}
\maketitle

\begin{abstract}
  Earth introduces strong attenuation and dispersion to propagating waves. 
  The time-fractional wave equation with very small fractional exponent, based on Kjartansson's
  constant-Q theory, is widely recognized in the field of geophysics as a reliable model for
  frequency-independent Q anelastic behavior.
  Nonetheless, the numerical resolution of this equation poses considerable challenges
  due to the requirement of storing a complete time history of wavefields.
  To address this computational challenge, we present a novel approach: a nearly optimal
  sum-of-exponentials (SOE) approximation to the Caputo fractional derivative with very small
  fractional exponent, utilizing the machinery of generalized Gaussian quadrature.
  This method minimizes the number of memory variables needed to approximate the power
  attenuation law within a specified error tolerance. We establish a mathematical equivalence
  between this SOE approximation and the continuous fractional stress-strain relationship,
  relating it to the generalized Maxwell body model. Furthermore, we prove an improved SOE
  approximation error bound to thoroughly assess the ability of rheological models to
  replicate the power attenuation law. Numerical simulations on constant-Q viscoacoustic equation in 3D homogeneous media
  and variable-order P- and S- viscoelastic wave equations in 3D inhomogeneous media
  are performed. These simulations demonstrate that our proposed technique accurately captures
  changes in amplitude and phase resulting from material anelasticity. 
  This advancement provides a significant step towards the practical
  usage of the time-fractional wave equation in seismic inversion.


\end{abstract}

\newsavebox{\tablebox}
\setcounter{tocdepth}{2}

\section{Introduction}
Seismic wave propagation has anelastic characteristics in real earth materials due to the loss of energy from the geometrical effect of the enlargement of the wavefront and the intrinsic absorption of the earth~\cite{LiuAndersonKanamori1976,Kjartansson1979,CarcioneCavalliniMainardiHanyga2002}.
The attenuation of seismic waves causes a decrease in the resolution of seismic images with depth, and transmission losses lead to variations in amplitude with offset~\cite{UrsinToverud2002}.
Therefore, accurate attenuation compensation is crucial for enhancing the reliability of seismic data interpretation~\cite{EmmerichKorn1987,YangBrossierMetivierVirieux2017,Tromp2020,MirzanejadTranWang2022,WangHarrierBaiSaadYangChen2022}. 

Empirically, the attenuation in an elastic solid is often described by formulations using memory variables, where the elastic moduli act as time convolution operators instead of constants \cite{UrsinToverud2002,bk:MarquesGreus2012}.
Energy dissipation is characterized by the quality factor $Q$, which is loosely defined as the number of wavelengths through which a wave can propagate in a medium before its amplitude decreases by
$e^{-\pi}$ \cite{BlanchRobertssonSymes1995}.
The most common methods are rheological models composed of Hookean elements (springs) and dashpots with different connecting styles, including the standard linear solid \cite{LiuAndersonKanamori1976}, the generalized Maxwell body (GMB), and the generalized Zener body (GZB)
\cite{BlanchRobertssonSymes1995,DayMinster1984,ParkSchapery1999,MoczoKristek2005,ZhuCarcioneHarris2013,CaoYin2015,BlancKomatitschChaljubLombardXie2016}.
These methods are based on an approximation of the viscoacoustic/viscoelastic modulus by a low-order rational function, or equivalently, replacing the relaxation spectrum with a discrete one, so that the frequency-independent Q behavior is mimicked by a superposition of mechanical elements
\cite{GrobyTsogka2006,WangXingZhuZhouShi2022}.
Despite their sound physical interpretation and convenience in the time-domain finite-difference implementation \cite{GrobyTsogka2006,RobertssonBlanchSymes1994}, there remains an open question regarding the overall accuracy of such approaches, especially when propagating waves over a large number of wavelengths \cite{BlancKomatitschChaljubLombardXie2016}.
Additionally, the quality factor $Q$ is implicitly parameterized by a set of parameters obtained via either linear or nonlinear optimization \cite{BlanchRobertssonSymes1995,DayMinster1984,BlancKomatitschChaljubLombardXie2016}, meaning that parameter crosstalk presents a significant challenge to the nonlinear inversion of the $Q$ value \cite{YangBrossierMetivierVirieux2017,WangLiuJiChenLiu2019}.

By contrast, the power attenuation law, which belongs to a family of fractional-derivative viscoelastic models, has garnered significant attention from geophysicists due to its succinct capability to describe the frequency-independent attenuation behavior consistently with the requirements of causality and dissipativity \cite{Kjartansson1979,CarcioneKosloffKosloff1988,HanyaSerdynska2003,bk:Mainardi2010}.
The time-fractional viscoelastic wave equation \cite{CarcioneKosloffKosloff1988,Carcione2009,bk:Mainardi2010,Zhu2017}, based on Kjartansson’s constant-Q theory \cite{Kjartansson1979}, is entirely specified by two parameters: the phase velocity at the reference frequency and the quality factor $Q$ \cite{Carcione2009}.
Consequently, it provides a much more economical parameterization across a broad range of frequencies and might potentially reduce the parameter crosstalk in the Q-inversion problem \cite{Zhu2017}.
However, the forward and inverse wavefield modeling via the time-fractional wave equation remains challenging because of the need to store a complete time history \cite{SchmidtGaul2006,Diethelm2008,HuangLiLiZengGuo2022}.
Since the loss caused by scattering and absorption is relatively minor in most situations, the attenuation of seismic energy is only treated as a minor perturbation to the propagation \cite{Kjartansson1979,SunGaoLiu2019}.
For instance, $Q$ values typically range from $10$ to $100$ in real viscoelastic media,
and the fractional exponent $2\gamma = \frac{2}{\pi} \arctan \frac{1}{Q}$  ranges from $10^{-3}$ to $10^{-1}$  \cite{CarcioneCavalliniMainardiHanyga2002}.
Due to the Abel kernel's flatness, the classical finite difference methods for the temporal fractional stress-strain relation struggle with the extensive memory requirement. As indicated in
\cite{Zhu2017,LiuHuGuoLuo2018},
this requires the storage of $400$-$500$ levels of the wavefield history in real imaging applications, resulting in substantial computational efforts to address the nonlocality, which hinders the broader application of Kjartansson’s theory.

However, such long-memory dependence seems to contradict our intuition that stress should rely more on the immediate history of strain than on its distant history \cite{bk:Christensen2012}. Since the fractional Caputo derivative with a small exponent is very close to an identity operator, a more appropriate discretized approximation should adhere to the short-memory principle \cite{YuanAgrawal2002,BlancChiavassaLombard2014,XiongGuo2022}.
To achieve this, an effective strategy is to convert the flat Abel kernel into a localized power function through an integral representation of the power creep, yielding the sum-of-exponentials (SOE) approximation to the fractional stress-strain relation. The challenge lies in accurately mimicking a given attenuation law using a minimal set of memory variables \cite{BlancKomatitschChaljubLombardXie2016}, that is, in determining how to reduce the number of exponentials, $\nexp$, for a given error tolerance.

The purpose of our paper is two-fold. First, we aim to explore the limit of compressing the memory variables for a practically small fractional exponent. While there is a vast body of literature devoted to constructing efficient SOE approximations 
\cite{jiang2001thesis,JiangGreengard2004,BeylkinMonzon2010,JiangZhangZhangZhang2017,JRLi2010,HuangLiLiZengGuo2022},
existing methods might still suffer from stagnation. This means they might require more exponentials than necessary due to the low-exponent breakdown in the initial integral representation. To address this challenge, we begin with a new integral representation of the power function and employ the generalized Gaussian quadrature (GGQ)
machinery \cite{BremerGimbutasRokhlin2010,ggq2,ggq3}
to construct a (nearly) optimal SOE approximation, which requires less than half the nodes
compared to existing methods \cite{JiangZhangZhangZhang2017}. 
To validate the accuracy of our new SOE approximation, we simulate the constant-Q viscoacoustic wave equation in 3-D homogeneous media and compare the results with the analytical solution
\cite{Hanyga2002}. Numerical results show that the new SOE approximation can accurately capture the changes in both amplitude and phase induced by material anelasticity \cite{ZhuCarcioneHarris2013}.
For the strong attenuation case ($Q=10$), it can achieve a relative error less than $10^{-3}$ with
$\nexp < 10$. Additionally, it demands fewer memory variables for larger $Q$ values, and
$\nexp$ decreases as the fractional exponent approaches zero, ensuring that stagnation is
effectively circumvented.

Second, we establish a rigorous mathematical equivalence between the SOE approximation of the power creep model and GMB. The latter has been proven to be equivalent to other rheological models such as GZB \cite{MoczoKristek2005,CaoYin2015}.
It is observed that the discretized constitutive equation \eqref{fractional_stress_strain} can be
transformed into a differential equation form \cite{YuanAgrawal2002,XiongGuo2022}, leading to the relaxed dynamics of a collection of parallel springs and dashpots \cite{ParkSchapery1999,SchmidtGaul2006}.
In some sense, the SOE approximation provides a nearly optimal finite spring-dashpot representation to approximate the continuous fractional stress-strain relation through a uniformly accurate curve fitting to the power function $(t/t_0)^{-2\gamma}$  \cite{ParkSchapery1999,ZhuCarcioneHarris2013}.
An improved error bound for the new SOE approximation is presented, ensuring numerical accuracy within the specified wave propagation time and range of wavelengths, and firmly establishing the ability of rheological models to approximate the frequency-independent Q behavior. 
Compared with other approaches, such as the Yuan-Agrawal method \cite{YuanAgrawal2002,LuHanyga2005} or the diffusive approximation \cite{Diethelm2008,BlancChiavassaLombard2014},
the new SOE approximation avoids the augmentation of projection errors \cite{XiongGuo2022}.
This is because it possesses greater flexibility in stiffnesses and viscosities to adapt to the viscoelastic behaviors of the given material data, thereby addressing some of the criticisms noted in
\cite{SchmidtGaul2006}. This represents a significant advancement towards the practical use of the time-fractional wave equation in real 3D geological applications.

The rest of this paper is organized as follows. \Cref{sec.backgroud} provides a brief introduction to the constant-Q viscoelastic and viscoacoustic wave equations.
In \cref{sec.scheme}, a new SOE approximation to the time-fractional constitutive relation is developed, along with an exponential operator splitting scheme. An improved SOE error bound is proved to rigorously establish the ability of rheological models to mimic the power attenuation law.
\Cref{sec.numerical} showcases numerical experiments on the viscoacoustic wave equation in 3D homogeneous media and the variable-order P- and S- viscoelastic wave equations in 3D inhomogeneous media to validate the convergence and accuracy of the new SOE approximation. Finally, conclusions and discussions are presented in \cref{sec.conclusion}.

\section{The constant-Q wave equation}
\label{sec.backgroud}

In this section, we briefly review the theory of constant-Q viscoelastic wave equation, which is completely characterized by the group velocities $c_P$, $c_S$
and the quality factors $Q_P$, $Q_S$ for longitudinal and shear waves (P- and S-waves), respectively. 
The dynamics of constant-Q wave propagation is governed by three sets of equations.  

First, the conservation of linear momentum leads to 
\begin{equation}\label{conservation_momentum}
  \rho(\bx) \frac{\partial^2}{\partial t^2} u_i(\bx, t) = \sum_{j=1}^3 \frac{\partial}{\partial x_j} \sigma_{ij}(\bx, t) + f_i(\bx, t),
  \quad i = 1, 2, 3, \; \bx \in \mathbb{R}^3, 
\end{equation}
where $\sigma_{ij}$ are the components of the stress tensor, $u_i$ are the components of the displacement vector, $\rho$ is the mass density and $f_i$ are components of the body forces per unit (source term). 

Second, the strain tensor $\varepsilon_{ij}$ is related to the displacement vector via the formula
\begin{equation}\label{defintion_strain}
\varepsilon_{ij}(\bx, t) = \frac{1}{2} \left(  \frac{\partial}{\partial x_i} u_j(\bx, t) +  \frac{\partial}{\partial x_j} u_i(\bx, t) \right), \quad i, j = 1, \dots, 3.
\end{equation}

Third, the constitutive equation for the viscoelastic media \cite{Carcione2009} is 
\begin{equation}\label{stress_strain_relation}
  \ba
  \sigma_{ij}(\bx, t) &= \frac{M_P(\bx)}{\omega_0^{2\gamma_P(\bx)}} { _{C}D_t^{2\gamma_P(\bx)}}\left[ \sum_{k=1}^3\varepsilon_{kk}(\bx, t) \delta_{ij}\right]\\
  &+   \frac{2 M_S(\bx)}{\omega_0^{2\gamma_S(\bx)}}  {_{C}D_t^{2\gamma_S(\bx)}} \left[ \varepsilon_{ij}(\bx, t) - \sum_{k=1}^3\varepsilon_{kk}(\bx, t) \delta_{ij} \right], \quad i, j=1,\dots,3.
\ea
\end{equation} 
Here $\omega_0 =t_0^{-1}$ is the reference frequency, $\gamma_P(\bx) = \pi^{-1}\arctan Q_P^{-1}(\bx)$, $\gamma_S(\bx) = \pi^{-1}\arctan Q_S^{-1}(\bx)$. $M_P(\bx)  = \rho(\bx) c_P^2(\bx) \cos^2(\pi \gamma_P(\bx)/2)$, $M_S(\bx)  = \rho(\bx) c_S^2(\bx) \cos^2(\pi \gamma_S(\bx)/2)$ are the bulk moduli for P- and S-wave, respectively. $_{C}D_t^{\beta} $ is the Caputo fractional derivative operator defined as 
\begin{equation}\label{def.extended_Caputo_time_fractional_operator}
{_C}D_t^{\beta} \varepsilon(\bx, t) = 
\frac{1}{\Gamma(1 - \beta)} \int_0^t (t - \tau)^{-\beta} \left[\frac{\partial}{\partial \tau} \varepsilon(\bx, \tau)\right] \D \tau,  \quad 0 < \beta < 1. \end{equation}
The variable-order Caputo fractional derivative operators $ {_C}D_t^{2\gamma_P(\bx)}$ and $ {_C}D_t^{2\gamma_S(\bx)}$ characterize the anisotropic attenuation in inhomogeneous materials, e.g., the multi-layer viscoelastic media.

When $c_P=c_S$ and $\gamma_P=\gamma_S$, \cref{conservation_momentum} --\cref{stress_strain_relation}
can be reduced via the Helmholtz decomposition
and the P-wave propagation can be described by a scalar viscoacoustic wave equation \cite{ChapmanHobroRobertsson2014}
\begin{equation}
\rho(\bx) \frac{\partial^2}{\partial t^2} u(\bx, t) = \nabla \cdot  (\rho(\bx) C_P(\bx) {_C D_t^{2\gamma_P(\bx)}} \nabla u(\bx, t)) + f(\bx, t)
\label{pwaveequation}
\end{equation}
with $C_P(\bx) = c_P^2(\bx)\cos^2(\pi \gamma_P(\bx)/2) \omega_0^{-2\gamma_P(\bx)}$.
In the case of $ \rho(\bx)  \equiv \rho_0$, $C_P(\bx) \equiv C_P$, $\gamma_P(\bx) \equiv \gamma_P$ and $f  = 0$,  \cref{pwaveequation}
becomes
\begin{equation}
  _C D_t^{2-2\gamma_P} u(\bx, t) = C_P \Delta u(\bx, t),
  \label{constantpwave}
\end{equation}
where the Caputo derivative (with order $1<\beta <2$) is defined as 
\begin{equation}\label{caputodef2}
{_C}D_t^{\beta} u(\bx, t) = \frac{1}{\Gamma(2 - \beta)} \int_0^t (t - \tau)^{1-\beta} \left[\frac{\partial^2}{\partial \tau^2} u(\bx, \tau)\right] \D \tau,  \quad 1 < \beta < 2.
\end{equation}

When \cref{constantpwave} is combined with the initial data $u(\bx, 0^+) = u_0(\bx)$ and $\frac{\partial }{\partial t} u(\bx, 0^+) = 0$, the solution to this initial
value problem (IVP) can be expressed as follows:
\begin{equation}\label{exact_solution}
u(\bx, t) =  \int_{\mathbb{R}^3} G^{(3)}(\by, t; 1-\gamma_P) u_0(\bx - \by) \D \by,
\end{equation}
where $G^{(3)}(\bx, t; 1-\gamma_P)$ is the Green's function of the given IVP (see \cref{Green} for its explicit expression in terms of the derivatives of
Mainardi functions).
The Green's function $G^{(3)}$ with different quality factors $Q_P$ (or equivalently, $\gamma_P$) is visualized in \cref{Green_function}. We observe that Green's functions for time-fractional wave equation are oscillatory, bounded from below and above, and approach the derivative of the Heaviside function as $\gamma_P \to 0$ ($Q_P \to \infty$). Moreover, the power function relaxation in the fractional stress-strain relation not only leads to amplitude loss but also contributes to phase dispersion.
It may even change the first arrival time of the seismic signals
\cite{CarcioneKosloffKosloff1988}. The smoothness of the wavefield at the wave fronts is an important consequence of singular memory, implying that the signals build up with some delay after the passage of the wave front. Therefore, it necessities a correction in the travel time in view of the dependence of the signal delay on the propagation distance \cite{HanyaSerdynska2003}.

\begin{figure}[!ht]
    \centering
    \subfigure[Derivatives of the Mainardi functions.]{{\includegraphics[width=0.49\textwidth,height=0.26\textwidth]{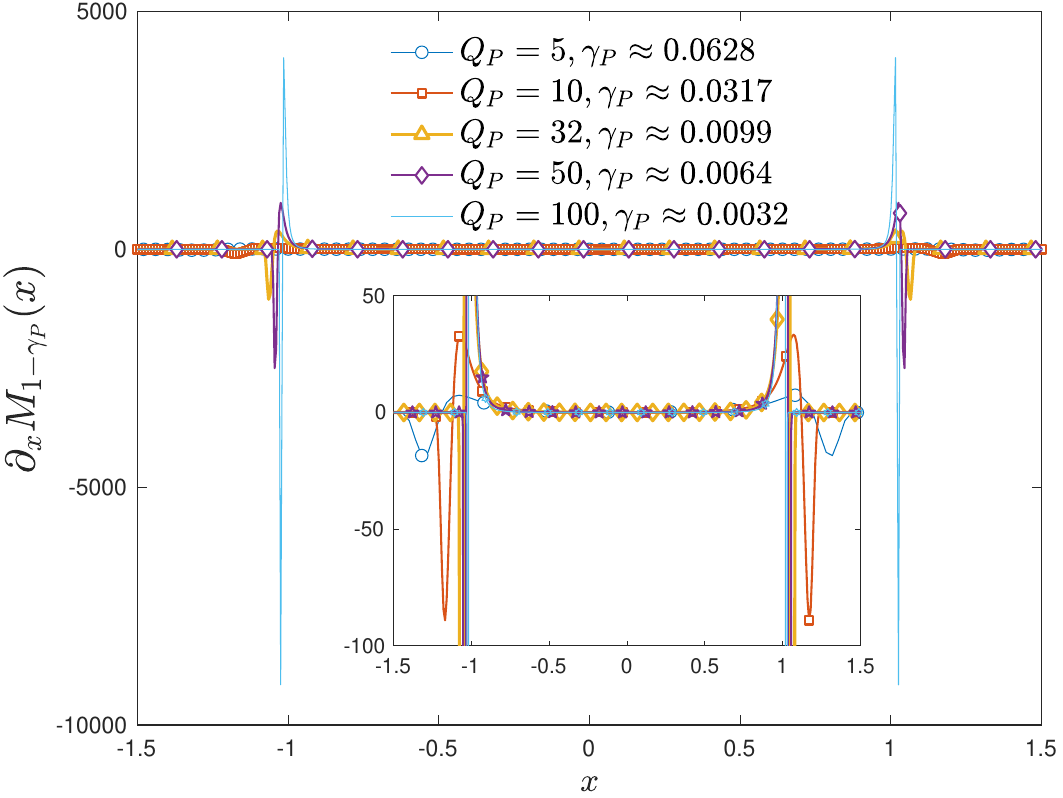}}}
    \subfigure[Green's functions.]{{\includegraphics[width=0.49\textwidth,height=0.26\textwidth]{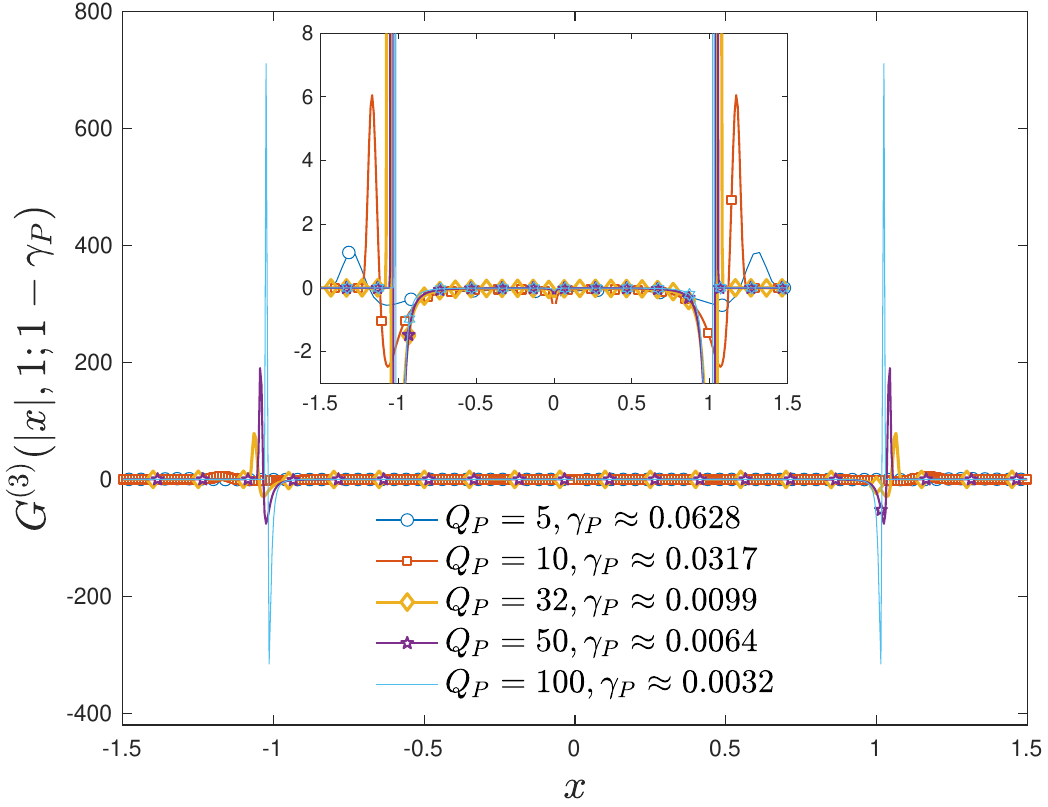}}}
    \caption{\small The derivatives of the Mainardi functions and Green's functions for 3-D constant-Q wave equation under different quality factors $Q_P$.
      The power function relaxation in the fractional constitutive relation changes both amplitude loss and phase dispersion.
    \label{Green_function}}
\end{figure} 

\section{Short-memory scheme for constant-Q wave equation}
\label{sec.scheme}

The major difficulty to numerically solve the time-fractional viscoelastic wave equation, as pointed out in \cite{Carcione2009,Zhu2017}, lies in the necessity of storing the whole history.  A potential way to alleviate such problem is to utilize an alternative form of the fractional Caputo derivative and the reciprocal relation $t^{-2\gamma} \to s^{2\gamma-1}$, then truncate the time-reciprocal variable $s$ instead. For $\gamma \ll 1$, SOE approximations can dramatically reduce the number of memory variables as $s^{2\gamma-1}$ is strongly localized. In the following, we propose a new SOE approximation and demonstrate its equivalence to GMB.  

\subsection{Exponential splitting scheme}
\label{sec.differential_form}
To simplify the notation, we will use the wave equation in one dimension to illustrate the key ideas.          
Introduce the velocity $v(x, t) = \frac{\partial }{\partial t} u(x, t)$. In one dimension, \cref{conservation_momentum}--\cref{stress_strain_relation}
are reduced to
\begin{equation}\label{1d_wave}
\left\{
\begin{split}
&\frac{\partial }{\partial t} v(x, t) = \frac{1}{\rho(x)} \frac{\partial }{\partial x} \sigma(x, t) +  \frac{1}{\rho(x)} f(x), \\
&\frac{\partial }{\partial t} \sigma(x, t) = \rho(x) C_P(x)  {_C}D_t^{2\gamma_P(x)} \frac{\partial }{\partial x} v(x, t),
\end{split}
\right.
\end{equation}
where we have used the following relationship between the strain and the velocity
\begin{equation}\label{strainvelocity}
\frac{\partial}{\partial t} \varepsilon(x, t) = \frac{\partial }{\partial x} v(x, t).
\end{equation}

Note that the power function in the Caputo fractional derivative has the integral representation
\begin{equation}\label{powerfunctionrep}
  \frac{1}{t^\beta} = \frac{1}{\Gamma(\beta)}\int_0^\infty e^{-t s}s^{\beta-1}\D s.
\end{equation}

Combining \cref{caputodef2} and \cref{1d_wave}--\cref{powerfunctionrep}, one has the stress-strain relation
\begin{equation}\label{extension_stress_strain}
\begin{aligned}
  \sigma(x, t) &= \frac{ \rho(x)  C_P(x)}{\Gamma(1-2\gamma_P(x))\Gamma(2\gamma_P(x))}  \int_0^{\infty} s^{2\gamma_P(x) - 1}  \left\{\int_0^t e^{-s (t - \tau)} \left[\frac{\partial}{\partial \tau} \varepsilon(x, \tau)\right] \D \tau\right\} \D s\\
  &= \frac{ \rho(x)  C_P(x)}{\Gamma(1-2\gamma_P(x))\Gamma(2\gamma_P(x))}  \int_0^{\infty} s^{2\gamma_P(x) - 1} \Phi(x,s,t)\D s,
\end{aligned}
\end{equation}
where the auxiliary function $\Phi$ is defined by the formula
\begin{equation}
\begin{split}
&\Phi(x, s, t) =  \int_{0}^{t} e^{- s(t - \tau) } \left[\frac{\partial}{\partial \tau} \varepsilon(x, \tau)\right] \D \tau =  \int_{0}^{t} e^{- s(t - \tau) } \left[\frac{\partial}{\partial x} v(x, \tau)\right] \D \tau.
\end{split}
\end{equation}
It is easy to see that the auxiliary function $\Phi$ satisfies the differential equation 
\begin{equation}\label{stiff_equation}
\frac{\partial }{\partial t} \Phi(x, s, t) = -s \Phi(x, s, t) + \frac{\partial }{\partial x} v(x, t)
\end{equation}
with the initial condition
\begin{equation}\label{initialcondition}
\Phi(x, s, t=0^+) = \frac{\Gamma(1-2\gamma_P(x))}{ \rho(x) C_P(x)} e^{-s} \sigma(x, t=0^+).
\end{equation}
Here, we would like to remark that the initial condition \cref{initialcondition} matches the asymptotic behaviors of the power function
(see \cite{XiongGuo2022} for details). 

Suppose now that the power function admits an efficient SOE approximation
\begin{equation}\label{powersoeappr}
  \frac{1}{t^{2\gamma_P(x)}} \approx \sum_{j=1}^{N_{\rm exp}}w_j e^{-s_jt}.
\end{equation}

\begin{remark}
We would like to remark that the nodes $s_j$ and weights $w_j$ in the SOE approximation \cref{powersoeappr}
may depend on the spatial variable $x$ since $\gamma_P$ is a function of $x$. Analyzing the case where $\gamma_P(x)$
is an arbitrary function of $x$ can be challenging. Fortunately, in many geophysical models, $\gamma_P(x)$ is a piecewise
constant function, which corresponds to layered media. See \cref{sec.3d.layeredmedia}
for a numerical example of this case.
\end{remark}

Combining  \cref{powerfunctionrep}, \cref{extension_stress_strain}, and \cref{powersoeappr}, we obtain 
\begin{equation}\label{fractional_stress_strain}
\begin{aligned}  
  \sigma(x, t)&\approx  \frac{\rho(x)  C_P(x)}{\Gamma(1-2\gamma_P(x))}
  \sum_{j=1}^{N_{\textup{exp}}} w_j \underbrace{\int_0^t e^{-s_j \tau}\frac{\partial}{\partial t} \varepsilon(\bx, t - \tau) \D \tau}_{\textup{memory variables}}\\
  &=  \frac{\rho(x)  C_P(x)}{\Gamma(1-2\gamma_P(x))}
  \sum_{j=1}^{N_{\textup{exp}}} w_j   \Phi(x, s_j, t).
\end{aligned}  
\end{equation}
And the fractional derivative is removed from \cref{1d_wave}, which is reduced to the following system of differential equations
\begin{equation}\label{1d_differential_form}
\left\{
\begin{split}
&\frac{\partial }{\partial t} v(x, t) = \frac{1}{\rho(x)} \frac{\partial }{\partial x} \sigma(x, t) +  \frac{1}{\rho(x)} f(x), \\
  & \sigma(x, t)  =  \frac{\rho(x)  C_P(x)}{\Gamma(1-2\gamma_P(x))}
  \sum_{j=1}^{N_{\textup{exp}}} w_j   \Phi(x, s_j, t),  \\
& \frac{\partial }{\partial t} \Phi(x, s_j, t) = -s_j \Phi(x, s_j, t) + \frac{\partial }{\partial x} v(x, t), \quad j = 1, \dots, \nexp.
\end{split}
\right.
\end{equation}

It seems that the stiff terms (i.e., the terms with large values of $s_j$) in \cref{1d_differential_form} may impose a severe stability constraint
on the time step size $ \Delta t$ \cite{Diethelm2008}. However, they can be integrated stably by an exponential operator splitting scheme,
where the velocity and stress components are updated alternatingly \cite{XiongGuo2022}. 
Indeed, for a small interval $[t_n, t_{n+1}]$ with $t_{n+1} - t_n = \Delta t$,  the stress components can be assumed to be invariant when updating the velocity components. In the meantime,  the velocity components can be assumed to be invariant when updating the stress components, so that the exact solutions of the dynamics of rigid bodies \eqref{dynamics_phi_trace} can be utilized. In this way, the constraint on the time step is removed as the exact flows of rigid dynamics are exploited.
In practice, we use the following second-order exponential Strang splitting scheme.
\begin{equation}\label{1d_exponential_splitting}
\left\{
\begin{split}
&v(x, t_{n+\frac{1}{2}}) = v(x, t_{n}) + \frac{\Delta t}{2} \frac{1}{\rho(x)} \frac{\partial}{\partial x} \sigma(x, t_n) +  \int_{t_n}^{t_n +\frac{\Delta t}{2}} \frac{f(x, s)}{\rho(x)} \D s.  \\
&\Phi(x, s_j , t_{n+1}) = e^{- s_j \Delta t} \Phi(x, s_j , t_{n}) + \frac{1}{s_j}(1 - e^{-s_j \Delta t})  \frac{\partial}{\partial x}v(x, t_{n+\frac{1}{2}}), \quad j = 1, \dots, \nexp,\\
& \sigma(x, t_{n+1}) = \frac{\rho(x)  C_P(x)}{\Gamma(1-2\gamma_P(x))} \sum_{j=1}^{\nexp} w_j   \Phi(x, s_j, t_{n+1}), \\
&v(x, t_{n+1}) = v(x, t_{n+\frac{1}{2}}) + \frac{\Delta t}{2}\frac{1}{ \rho(x)} \frac{\partial}{\partial x} \sigma(x, t_{n+1}) +  \int_{t_{n+\frac{1}{2}}}^{t_{n+\frac{1}{2}} +\frac{\Delta t}{2}} \frac{f(x, s)}{\rho(x)} \D s.
\end{split}
\right.
\end{equation}

\subsection{Nearly optimal SOE approximations for the power function with small exponent}
\label{sec.setting_new_SOE}

\label{sec.SOE}
We now construct an efficient SOE approximation of the power function $t^{-\beta}$ when the exponent $\beta$ is small (and possibly close to zero).
That is,
we try to find nodes $s_j$ and weights $w_j$ ($j=1,\ldots,\nexp$) such that
\be\label{soeappr}
  \left|{t^{-\beta}}-\sum_{j=1}^{\nexp} w_j e^{-s_j t}\right| \leq
  \varepsilon {t^{-\beta}}, \quad t\in[\delta, T],
\ee
for exponent $0<\beta\le 1$ within a prescribed (relative) error tolerance $\varepsilon$, a small gap $\delta$ and final time $T$ \cite{BeylkinMonzon2010,JiangZhangZhangZhang2017,JRLi2010}. In general, we can set $\delta = \Delta t$.

When the exponent $\beta$ is very small, the power function becomes flatter and approaches the
constant function away from the origin. Intuitively, one expects that the SOE approximation would need fewer number of exponentials as the exponent decreases. However, existing methods for constructing the SOE approximation of the power function suffer from stagnation when $\beta$ approaches zero because the starting integral representation of the existing methods has a low-exponent breakdown. The reason is that many practitioners \emph{incorrectly} conclude that $\nexp$ will grow as $\beta\rightarrow 0$, which is not in agreement with the observation that the power function $t^{-\beta}$ becomes closer and closer to a constant function as $\beta\rightarrow 0$. Partly, this incorrect conclusion stems from the $1/\beta$ and $\log \beta$ factors
in Theorem 5 of \cite{BeylkinMonzon2010}, which shows that $\nexp$ will grow rapidly as $\beta\rightarrow 0$.

In~\cite{JiangZhangZhangZhang2017,JRLi2010}, the construction of the SOE approximations for the
general power function starts from the integral representation \eqref{powerfunctionrep}. 
It then truncates the infinite integral in \eqref{powerfunctionrep} according to $\delta$ and $\varepsilon$,
divides the resulting finite interval into dyadic subintervals towards the origin, and
approximates the integral on each subinterval by the Gauss-Legendre quadrature of a proper order.
In \cite{BeylkinMonzon2010}, a further
change of variable is used to obtain another integral representation of the power function
\be\label{intrep2}
\frac{1}{t^\beta} = \frac{1}{\Gamma(\beta)}\int_{-\infty}^\infty e^{-te^x+\beta x}\D x.
\ee
The infinite integral in \eqref{intrep2} is then discretized via the truncated trapezoidal rule.

The construction in~\cite{JiangZhangZhangZhang2017} tries to bound the absolute error of the SOE
approximations in order to prove the stability of the overall numerical scheme for solving
fractional PDEs. In practice, one observes that SOE approximations with controlled relative
error are sufficient, even though the bound on relative error is actually weaker
than the absolute error for $t\in [\delta, 1]$. The analysis in \cite{BeylkinMonzon2010} 
provides an estimate on $\nexp$ based on relative error. However, because of additional
change of variable introduced in \eqref{intrep2}, the truncated trapezoidal rule
used in \cite{BeylkinMonzon2010} leads to many small exponentials. And the resulting estimate
on $\nexp$ blows up as $\beta \rightarrow 0$. We would like to remark here that in
\cite{BeylkinMonzon2010} a modified Prony's method is applied to reduce the number of
small exponentials. Thus, the algorithms in \cite{BeylkinMonzon2010} can be used to construct
the SOE approximations for the power function with very small $\beta$. The number of exponentials
does not blow up as $\beta\rightarrow 0$, though it is slightly inefficient.

When the exponent $\beta$ is very close to zero, the integrand in \eqref{powerfunctionrep} is nearly
singular due to the factor $s^{\beta-1}$. Simultaneously, $\Gamma(\beta)$ tends to infinity as
$\beta\rightarrow 0$. In fact, the estimate on $\nexp$ in \cite[Eq. (22)]{BeylkinMonzon2010}
contains the factor ${1}/{\beta}$, which blows up as $\beta\rightarrow 0$. Nevertherless,
the algorithm in \cite{BeylkinMonzon2010} can still be used to obtain efficient SOE approximations
for the power function in this case. The reason is that Prony's method is able to reduce the number
of small exponentials significantly, effectively removing the factor ${1}/{\beta}$ in the
complexity estimate of $\nexp$.

\subsubsection{Improved bound on $\nexp$ for $\beta > 0$}

We now refine the analysis in \cite{JiangZhangZhangZhang2017}
to give an improved estimate on $\nexp$ when one is concerned with relative error. 
By simple scaling $\tilde \delta = \delta/T$, we may consider the same approximation problem as in \cite{BeylkinMonzon2010}:
\be\label{soeappr2}
  \Big |\frac{1}{t^{\beta}}-\sum_{j=1}^{\nexp} w_j e^{-s_j t}\Big | \leq
  \varepsilon\frac{1}{t^{\beta}}, \quad t\in[\tilde \delta, 1].
\ee
The original approximation problem~ \eqref{soeappr} is identical to \eqref{soeappr2} with
$\tilde \delta=\delta/T$.

Using the property $\Gamma(1+x) = x\Gamma(x)$, we have 
\be\label{intrep4}
\frac{1}{t^\beta} = \frac{1}{\Gamma(\beta)}\int_0^\infty e^{-t s}s^{\beta-1}\D s
=\frac{\beta}{\Gamma(1+\beta)}\int_0^\infty e^{-t s}s^{\beta-1}\D s, 
\ee
We now decompose the integral on the right side of \cref{intrep4} into
three parts:
\begin{equation*}
\int_0^\infty e^{-t s}s^{\beta-1}\D s =  \underbrace{\int_{0}^{1} e^{-ts}s^{\beta-1} \D s}_{\textup{weakly singular}}  +  \underbrace{\sum_{j = 0}^{N-1}\int_{2^j}^{2^{j+1}} e^{-ts}s^{\beta-1} \D s}_{\textup{non-singular}} + \underbrace{\int_{2^N}^{\infty} e^{-t s} s^{\beta-1}\D s}_{\textup{truncation}}.
 \end{equation*}
The weakly singular part can be approximated by the Gauss-Jacobi quadrature, and the non-singular part on $[2^j,2^{j+1}]$ ($j=0,\ldots,N-1$) can be approximated by the Gauss-Legendre
quadrature. The truncation term can be ignored since it becomes sufficiently small when $N$ is large.

\begin{lemma}[truncation part]\label{soelem1}
  Suppose that $p_0 \tilde \delta >1$. Then
  for $t\geq \tilde \delta >0$,
  \be
  \frac{t^{\beta}}{\Gamma(\beta)}\int_{p_0}^\infty e^{-ts}s^{\beta-1} \D s
  \leq \frac{\beta}{\Gamma(1+\beta)}e^{- \tilde \delta p_0}.
  \ee
\end{lemma}
\begin{proof}
It is readily verified that
\begin{equation}
\begin{aligned}
\frac{t^\beta}{\Gamma(\beta)}\int_{p_0}^\infty e^{-ts}s^{\beta-1} \D s
&= \frac{1}{\Gamma(\beta)}
e^{-tp_0}\int_0^\infty e^{-x}(x+tp_0)^{\beta-1} \D x \leq e^{-\tilde \delta p_0}\frac{\beta}{\Gamma(1+\beta)}.
\end{aligned}
\end{equation}
\end{proof}

\begin{lemma}[non-singular part]\label{soelem2}
  For a dyadic interval $[a, b] = [2^j, 2^{j+1}]$, let $s_1,\dots, s_n$ and
  $w_1,\dots,w_n$ be the nodes and weights for $n$-point Gauss-Legendre
  quadrature on the interval. Then for $\beta > 0$ and $n\ge 8$,
  \begin{equation}
    \frac{t^\beta}{\Gamma(\beta)}
    \left|\int_a^b e^{-ts}s^{\beta-1}ds-\sum_{k=1}^n w_ks_k^{\beta-1}e^{-s_k t}\right|
    <\frac{\beta}{\Gamma(1+\beta)} 2^{ \frac{5}{2}}\pi  (2n)^\beta \left(\frac{1}{2}\right)^{2n}.
 \end{equation}

\end{lemma}

\begin{proof}
  By Eq. (2.18) in \cite{JiangZhangZhangZhang2017}, we have
 \begin{equation*}
 \frac{t^\beta}{\Gamma(\beta)}
 \left|\int_a^b e^{-ts}s^{\beta-1}ds-\sum_{k=1}^n w_ks_k^{\beta-1}e^{-s_k t}\right|
 \leq \frac{1}{\Gamma(\beta)}
 2^{\frac{5}{2}}\pi (at)^\beta e^{-at}\left(\frac{eat}{8n}+\frac{1}{4}\right)^{2n}.
\end{equation*}
Consider the function
 \begin{equation*}
f(x)=x^\beta e^{-x}\left(\frac{ex}{8n}+\frac{1}{4}\right)^{2n}.
\end{equation*}
Straightforward calculation shows that for $x>0$, $f$ achieves its maximum
at
 \begin{equation*}
x_\ast=\frac{\beta}{2}+\left(1-\frac{1}{e}\right)n
+\sqrt{\left(\frac{\beta}{2}+\left(1-\frac{1}{e}\right)n\right)^2+\frac{2n}{e}\beta}.
\end{equation*}
It is easy to see that $x_\ast>2\left(1-\frac{1}{e}\right)n$ and
 \begin{equation*}
\ba
x_\ast&<\beta+2\left(1-\frac{1}{e}\right)n+\frac{1}{2}\frac{2n}{e}\beta
\frac{1}{\frac{\beta}{2}+\left(1-\frac{1}{e}\right)n}\\
&<\beta+2\left(1-\frac{1}{e}\right)n+\frac{\beta}{e-1} <4+2\left(1-\frac{1}{e}\right)n.
\ea
\end{equation*}
Thus it yields that
 \begin{equation*}
\ba
\max_{x>0} f(x) &= f(x_\ast)=x_\ast^\beta e^{-x_\ast}\left(\frac{ex_\ast}{8n}+\frac{1}{4}\right)^{2n}\\
&<\left(4+2\left(1-\frac{1}{e}\right)n\right)^\beta
e^{-2\left(1-\frac{1}{e}\right)n}
\left(\frac{e}{8n}\left(4+2\left(1-\frac{1}{e}\right)n\right)
+\frac{1}{4}\right)^{2n}\\
&=\left(4+2\left(1-\frac{1}{e}\right)n\right)^\beta
\left(\frac{e^{1/e}}{2n}
+\frac{e^{1/e}}{4}\right)^{2n} <(2n)^\beta\left(\frac{1}{2}\right)^{2n},
\ea
\end{equation*}
where the last inequality follows from
$4+2\left(1-\frac{1}{e}\right)n<2n$ and 
$\frac{e^{1/e}}{2n}+\frac{e^{1/e}}{4}<\frac{1}{2}$ for $n\ge 8$. And the lemma follows.
\end{proof}

For the weakly singular part, the following bound slightly improves the estimate in \cite{JiangZhangZhangZhang2017}.
\begin{lemma}[weakly singular part]\label{soelem3}
  Let $s_1,\dots, s_n$ and
  $w_1,\dots,w_n$ ($n\geq 2$) be the nodes and weights for $n$-point Gauss-Jacobi
  quadrature with the weight function $s^{\beta-1}$
  on the interval. Then for $t\in [\tilde \delta,1]$, $\beta > 0$ and $n>1$,
    \begin{equation}\label{error_singular}
\frac{t^\beta}{\Gamma(\beta)}\left|\int_0^1 e^{-ts}s^{\beta-1}ds-\sum_{k=1}^n w_ke^{-s_k t}\right|
    < \frac{\beta}{\Gamma(1+\beta)}  \frac{2n+1}{2n+\beta} \sqrt{\frac{\pi}{n}}\left(\frac{e}{8n}\right)^{2n}.
 \end{equation}
\end{lemma}

\begin{proof}
It starts from the error estimate of the Gauss-Jacobi quadrature,
\begin{equation*}
\left|\int_0^a e^{-s t} s^{\beta-1}  \D s - \sum_{k=1}^{n} w_k e^{-s_k t} \right| \le \frac{a^{2n+\beta}}{2n+\beta} c_{n, \beta}\max_{s\in(0, a)}  |D_s^{2n} (e^{-s t})|, \quad c_{n, \beta} =  \frac{(n!)^2}{(2n)!} \left[\frac{\Gamma(n+\beta)}{\Gamma(2n+\beta)}\right]^2,
\end{equation*}
where $\max_{s\in[0, 1]}  |D_s^{2n} (e^{-s t})| = \max_{s\in[0, 1]} |t^{2n} e^{-s t}| \le 1$.

Now we introduce a strictly monotonically increasing sequence $\{d_n\}$, $d_n \le d_{n+1}$,
\begin{equation*}\label{def_dn}
d_n= \frac{2^{4n+1}(n!)^4}{(2n)!(2n+1)!}, \quad \lim_{n \to \infty} d_n = \pi.
\end{equation*}
From \cite{Kambo1970}, we have $2.66 < d_n \le \pi$ ($n > 1$). 
Thus,
\begin{equation*}
c_{n, \beta} < \frac{(n!)^2}{(2n)!} \left[ \frac{\Gamma(n)}{\Gamma(2n)}\right]^2 = \frac{[n!(n-1)!]^2}{(2n)! \left[(2n-1)!\right]^2} = \frac{d_n (4n^2 + 2n)}{ 2^{4n+1} n^2 (2n-1)!} \le \frac{\pi (4n+2)}{ 2^{4n}(2n)!},
\end{equation*}
where the first inequality utilizes the fact that 
\begin{equation*}
\frac{\Gamma(n + \beta)}{\Gamma(2n+\beta)} = \frac{(n -1 + \beta) \dots (1+\beta) \Gamma(1 + \beta)}{(2n-1+\beta) \dots (1+ \beta) \Gamma(1 + \beta)} = \frac{1}{(2n-1+\beta)\dots(n+\beta)} < \frac{\Gamma(n)}{\Gamma(2n)}.
\end{equation*}
The application of Stirling's approximation $(2n)! > \sqrt{2\pi}{(2n)^{2n+\frac{1}{2}}}e^{-2n}$ leads to
\begin{equation*}
\left|\int_0^1 e^{-s t} s^{\beta-1}  \D s - \sum_{j=1}^{n} w_k e^{-s_j t} \right| < \frac{\pi(4n+2)}{2n + \beta} \frac{1}{(2n)!} \left(\frac{1}{4}\right)^{2n} < \sqrt{\frac{\pi}{n}} \frac{2n+1}{2n+\beta}\left(\frac{e}{8}\right)^{2n} \left(\frac{1}{n}\right)^{2n},
\end{equation*}
which implies the estimate \eqref{error_singular}.
\end{proof}

We are now in a position to combine the above three lemmas to give an
efficient SOE approximation for $t^{-\beta}$
on $[\tilde \delta, 1]$ for $\beta > 0$ as follows.
\begin{theorem}\label{soethm}
Let $0<\tilde \delta\leq t \leq 1$,
let $\varepsilon > 0$ be the desired
precision, let $n_{o}=O(\log\frac{\beta}{\varepsilon}) $,
and let $N=O(\log\log\frac{\beta}{\varepsilon}+\log\frac{1}{\tilde \delta})$.
Furthermore, let ${s_{o,1}, \dots, s_{o,n_{o}}}$ and
${w_{o,1}, \dots, w_{o,n_{o}}}$
be the nodes and weights for the $n_{o}$-point Gauss-Jacobi quadrature
on the interval $[0,1]$,
and let ${s_{j,1}, \dots, s_{j,n_l}}$ and
${w_{j,1}, \dots, w_{j,n_l}}$
be the nodes and weights for $n_l$-point Gauss-Legendre quadrature
on intervals $[2^j,2^{j+1}]$, $j=0,\dots,N-1$, where
$n_l= O\left(\log\frac{\beta}{\varepsilon}+\log\log\frac{1}{\tilde \delta}\right)$.
Then for $t\in [\tilde \delta,1]$ and $\beta > 0$,
\begin{equation}
  \left|\frac{1}{t^\beta} -
  \frac{\beta}{\Gamma(1+\beta)}\left(\sum_{k=1}^{n_{o}} e^{-s_{o,k}t}w_{o,k}
  +\sum_{j=0}^{N-1}\sum_{k=1}^{n_l} e^{-s_{j,k}t}s_{j,k}^{\beta-1}w_{j,k}
  \right)
  \right|
\leq \varepsilon\frac{1}{t^\beta}.
\end{equation}
\end{theorem}
\begin{proof}
  By \cref{soelem1}, we may choose 
\begin{equation*}
 p_0 =O\left(\frac{1}{\tilde \delta}\log\frac{\beta}{\varepsilon}\right), \quad N=\lceil \log_2 p_0\rceil = O(\log\log\frac{\beta}{\varepsilon}+ \log\frac{1}{\tilde \delta})
\end{equation*}
  so that
\begin{equation*}
  \ba
 \left|\frac{1}{t^\beta} -\frac{\beta}{\Gamma(1+\beta)}\int_0^{2^N} e^{-ts}s^{\beta-1}\D s\right| 
  & =\frac{\beta}{\Gamma(1+\beta)}\int_{2^N}^\infty e^{-ts}s^{\beta-1}\D s \le   \varepsilon\frac{1}{t^{\beta}}, \quad t\in[\tilde \delta, 1],
  \ea
\end{equation*}
where we have used the fact that $\Gamma(x)>0.88$ for any $x>0$.
  We now split the integration range $[0,2^N]$ into $[0,1]$ and $[2^j,2^{j+1}]$, $j=0,\cdots,N-1$,
  and approximate the integral on $[0,1]$ by $n_o$-point Gauss-Jaboci quadrature and the other
  $N$ integrals on dyadic intervals by $n_l$-point Gauss-Legendre quadrature. This leads to
\begin{equation*}
  \begin{aligned}
  & \left| \frac{\beta}{\Gamma(1+\beta)}
  \left(\sum_{k=1}^{n_{o}} e^{-s_{o,k}t}w_{o,k}
  +\sum_{j=0}^{N-1}\sum_{k=1}^{n_l} e^{-s_{j,k}t}s_{j,k}^{\beta-1}w_{j,k}\right)
  -\frac{\beta}{\Gamma(1+\beta)}\int_0^{2^N} e^{-ts}s^{\beta-1}\D s\right|\\
  & \leq \frac{1}{t^\beta}\frac{\beta}{\Gamma(1+\beta)}
  \left( N 2^{\frac{5}{2}} \pi (2n_l)^{\beta} \left(\frac{1}{2}\right)^{2n_l}
  +
 \sqrt{\frac{\pi}{n_{o}}} \frac{2n_o+1}{2n_{o}+\beta}
  \left(\frac{e}{8n_o}\right)^{2n_o} \right).
  \end{aligned}
\end{equation*}
By choosing $n_l = O(\log \frac{N\beta}{\varepsilon})$ and $n_{o}=O(\log\frac{\beta}{\varepsilon})$, the error is bounded by $\varepsilon/t^{\beta}$ and
\begin{equation}
  \ba
  \nexp &= N n_l + n_o \lesssim \left(\log \frac{\beta}{\varepsilon} + \log N \right) N + n_o \\
  &=O\left( \left(\log\frac{\beta}{\varepsilon}+\log\log\frac{1}{\tilde \delta}\right)
  \left(\log\log\frac{\beta}{\varepsilon} +  \log\frac{1}{\tilde \delta}\right)\right), \qquad \beta>\varepsilon.
\ea\label{neestimate}
\end{equation}
\end{proof}
\begin{remark}\label{remark1}
For $\beta \le \varepsilon$, it is easy to see that the integration range in \cref{intrep4} can
be truncated to $[0,1]$ and the Gauss-Jacobi quadrature with $O(1)$ nodes is sufficient to approximate the integral
on $[0,1]$ with relative precision $\varepsilon$. In this case, $\nexp = O(1)$.
\end{remark}
\begin{remark}
We would like to emphasize that \cref{neestimate} provides a new complexity bound
on the total number of exponentials $\nexp$ needed to
achieve {\sl relative precision} $\varepsilon$, which is a key parameter
to determine the memory length in our scheme for solving the constant-Q wave equations.

Compared with our previous estimate  in \cite{JiangZhangZhangZhang2017} when bounding the {\sl absolute error}
\begin{equation}\label{nexpbound}
\nexp= O\left(\log\frac{1}{\varepsilon}\left(
    \log\log\frac{1}{\varepsilon}+\log\frac{T}{\delta}\right)
   +\log\frac{1}{\delta}\left(
     \log\log\frac{1}{\varepsilon}+\log\frac{1}{\delta}\right)
     \right),
\end{equation}
the $\log^2 ({1}/{\delta})$ factor in \cref{nexpbound} is gone. 
\end{remark}
\begin{remark}  
The dependence on $\beta$ is explicit in \cref{neestimate}. 
It is easy to see that $\nexp$ decreases as $\beta\rightarrow 0$
when other parameters are fixed (see \eqref{neestimate} and \cref{remark1}).
This is 
in agreement with the observation that the power function $1/t^\beta$
becomes closer and closer to a constant function as $\beta\rightarrow 0$.
Indeed, our nearly optimal SOE approximations for very small $\beta$
have a nearly static mode (i.e., an exponential node very close to zero).
On the other hand, the estimate on $\nexp$  in Theorem 5 of \cite{BeylkinMonzon2010}
contains $1/\beta$ and $\log \beta$ factors, which grows rapidly
as $\beta\rightarrow 0$.
\end{remark}

\subsubsection{Generalized Gaussian quadrature and the new integral representation}

Recently, GGQ (see, for example, \cite{ggq2}) has been applied to construct efficient SOE
approximations for $t^{-1}$ in \cite{GimbutasMarshallRokhlin2020}. As pointed out in \cite{ggq3}, the first step of GGQ is to use
singular value decomposition or pivoted QR decomposition to reduce the number of basis functions that need to be integrated.
For SOE approximation problems, it appears that there are infinitely many linearly independent functions since $t$ is the continuous
parameter on the interval $[\delta,T]$ in these functions (c.f., \eqref{powerfunctionrep}). 
However, to any {\it finite} precision $\varepsilon$,
the number of {\it numerically} linearly independent basis functions $n_f$ is surprisingly low for almost all problems encountered
in practice, even when some or all of these functions are singular or nearly singular.
This simple yet profound observation is critical to the usefulness of GGQ.
The second step of GGQ is to find an $n_f$-point quadrature to integrate these $n_f$ basis functions exactly, which
is easily achieved by preselecting $n_f$ quadrature nodes, say, at the shifted and scaled Gauss-Legendre nodes, then solving a square
linear system to obtain the associated quadrature weights. Finally, GGQ uses a nonlinear optimization procedure to reduce the number of
quadrature nodes by one in each step while maintaining the prescribed precision, until the final number of quadrature nodes is reduced by a factor of $2$ (or very close to $2$) (see, for example, \cite{BremerGimbutasRokhlin2010,serkh2016thesis}). 

In order to apply GGQ machinery to construct efficient SOE approximations for the power
function with small exponent, we introduce another change of variable $s=u^{1/\beta}$, leading to
\be\label{intrep3}
\frac{1}{t^\beta} = \frac{1}{\Gamma(\beta+1)}\int_0^\infty e^{-t u^{1/\beta}}\D u.
\ee
As compared with the integral representation \eqref{powerfunctionrep}, the advantage of the new integral
representation \eqref{intrep3} is obvious, especially for very small $\beta$. First, the factor
$\Gamma(\beta+1)$ approaches $1$ as $\beta\rightarrow 0$. Second, the integrand behaves much nicer
both at the origin and at the infinity. At the origin, the nearly singular term $s^{\beta-1}$ is gone.
And the integral can be truncated at a much smaller number
\be
p=O\left(\left(\frac{1}{\delta}\log\frac{1}{\varepsilon}\right)^\beta\right),
\quad t>\delta.
\ee

Algorithm \ref{alg.fat} contains a short summary on the construction of nearly optimal SOE approximations of the power function via the application
of GGQ to the integral representation \eqref{intrep3}.

\begin{algorithm}[!ht] 
\caption{(Nearly) optimal Sum-of-Exponential approximation via GGQ \label{alg.fat}}
\vspace{2mm}

\begin{itemize}

\item[Step 1.] Start from the integral representation \eqref{intrep3} of $t^{-\beta}$. Use SVD or pivoted QR decomposition to reduce the number of basis functions that need to be integrated \cite{ggq3,GimbutasMarshallRokhlin2020}. This step selects $n_f$ basis functions.

\item[Step 2.]  Preselect $n_f$ shifted and scaled Gauss-Legendre nodes and then solve a square linear system to obtain the associated quadrature weights. This step returns an $n_f$-point quadrature to integrate these $n_f$ basis functions exactly.

\item[Step 3.]  Use a nonlinear optimization procedure \cite{BremerGimbutasRokhlin2010,serkh2016thesis} to reduce the number of
quadrature nodes by one in each step while maintaining the prescribed precision.

\end{itemize}

\vspace{2mm}
\end{algorithm}

Here we provide some numerical examinations on the new SOE approximation. From \cref{curve_fitting}, it is confirmed that the new SOE approximation provides a good curve fitting for the power creep $t^{-\beta}$, even with a very small number of exponentials $\nexp \le 5$ (the precision is set to be $\varepsilon = 1\times10^{-2}$). Indeed, as presented in \cref{curve_error}, it achieves a uniform accuracy within the prescribed precision $\varepsilon$.

For typical qualify factors, we provide the number of exponentials in \cref{Nexp_SOE} to attain the given precision tolerance. It can be observed that GGQ is able to compress the memory variables to a large extent. Additionally, it requires fewer memory variables as the exponent becomes smaller. 
When the exponent is small, say, less than $0.1$, the number of exponentials $\nexp$ needed in the SOE approximation is reduced by a factor of two
or more as compared with published results~\cite{BeylkinMonzon2010,JiangZhangZhangZhang2017}.
For long-time simulations, as seen in \cref{Nexp_SOE}, $\nexp$ has to increase to maintain the prescribed precision due to $\log \frac{T}{\delta}$ factor in \eqref{neestimate}. 
Fortunately, the increase in $\nexp$ is fairly mild as the final instant $T$ increases from $10$ to $1000$ (i.e., the total number of time steps increases from $2000$ to $200000$). 

\begin{figure}[!ht]
    \centering
    \subfigure[Curve fitting for the power creep.\label{curve_fitting}]{{\includegraphics[width=0.49\textwidth,height=0.26\textwidth]{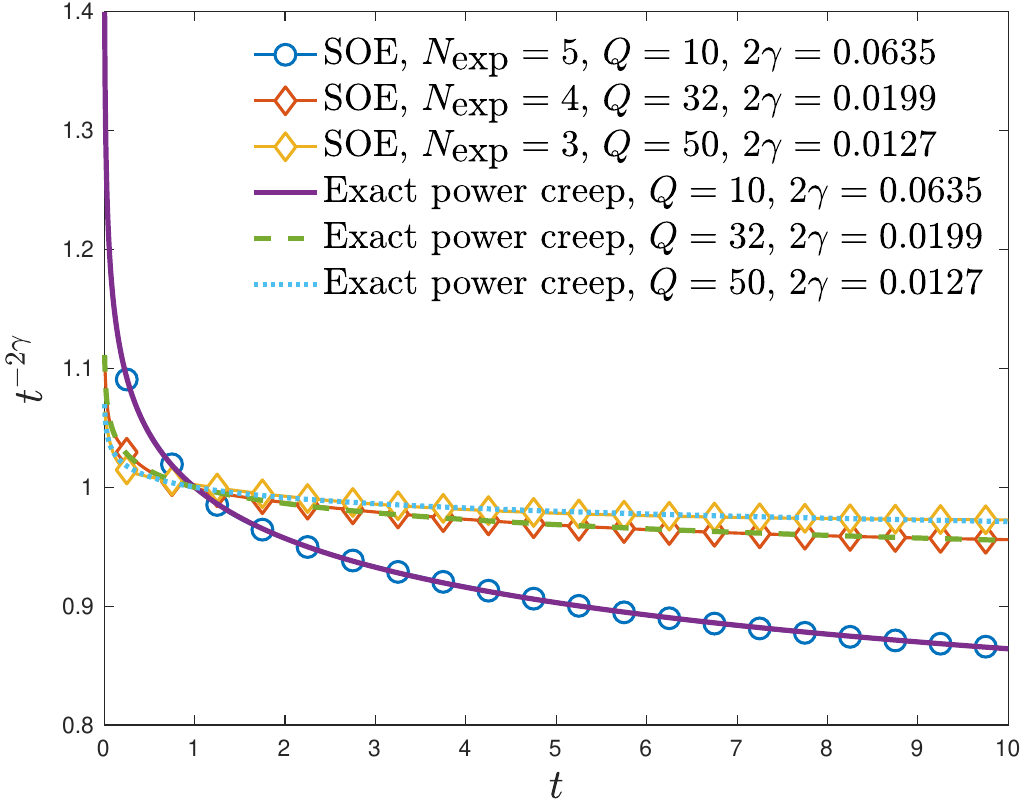}}}
    \subfigure[Absolute errors of SOE approximation.\label{curve_error}]{{\includegraphics[width=0.49\textwidth,height=0.26\textwidth]{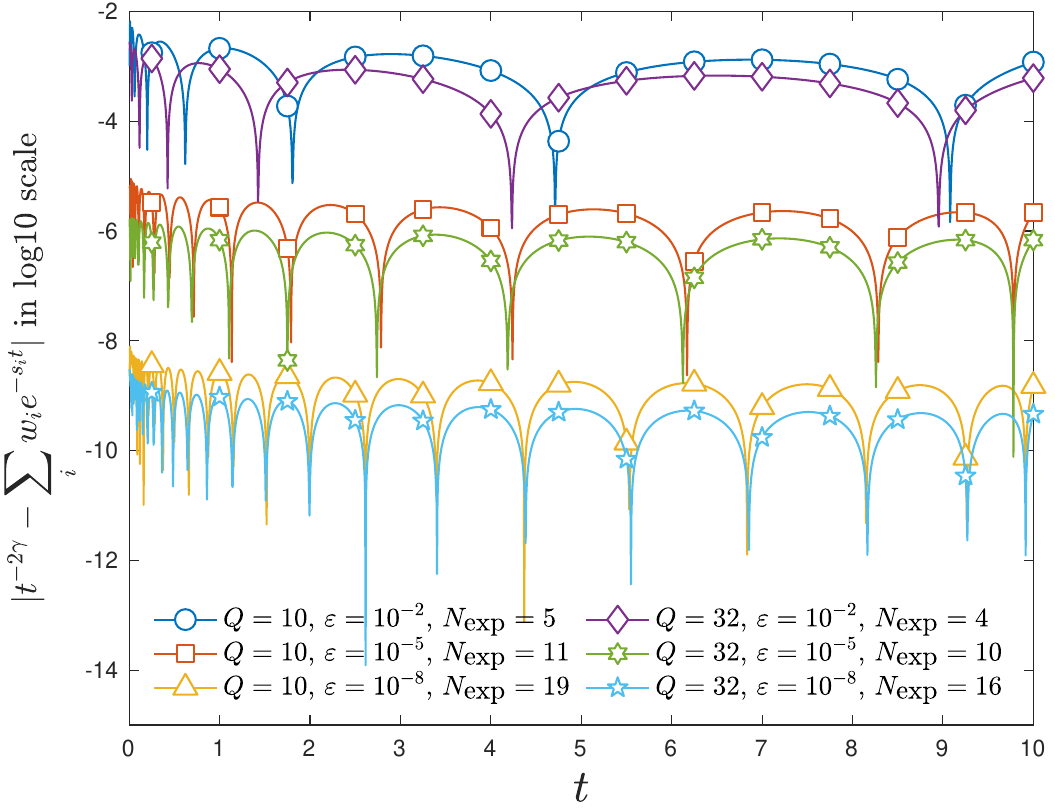}}}
    \caption{\small GGQ seeks an optimal curve fitting for the power creep $t^{-\beta}$. The number of exponentials $\nexp$ is adaptive to the fractional exponent $2\gamma$ and the error tolerance $\varepsilon$, and smaller $\nexp$ is allowed as the fractional exponent $\gamma$ decreases (for larger qualify factor $Q$).    \label{SOE_approx}}
\end{figure} 

\begin{table}[!ht]
  \centering
  \caption{\small Number of exponentials needed to approximate $t^{-\beta}$ with various error tolerance $\varepsilon$ and fractional exponents $\beta = 2\pi^{-1}\arctan Q^{-1}$.  Here, the gap (time step) is $\delta = 0.005$ and the final instants are $T=10, 100, 1000$, which are equivalent to
    $N_T=T/\delta= 2000, 20000, 200000$, respectively.
    Note that $\nexp$ decreases as the fractional exponent $\beta$ becomes smaller (for larger quality factor $Q$).
    \label{Nexp_SOE}}
\label{notation}
 \begin{lrbox}{\tablebox}
  \begin{tabular}{c|ccc|ccc|ccc|cccc}
  \hline\hline
  $T$   &  10 & $10^2$ & $10^3$	 &  10 & $10^2$ & $10^3$ &  10 & $10^2$ & $10^3$ &  10 & $10^2$ & $10^3$\\
  \hline
\diagbox[width=10\tabcolsep]{$\varepsilon$}{$\nexp$}{$\beta$}&
    \multicolumn{3}{c|}{$Q=10,\beta \approx 0.0635$} &  \multicolumn{3}{c|}{$Q=32,\beta \approx 0.0199$} &  \multicolumn{3}{c|}{$Q=50,\beta \approx 0.0127$}  &  \multicolumn{3}{c}{$Q=100,\beta \approx 0.0064$}   \\
\hline		
$10^{-2}$   &  5	 & 5 & 6 & 4 & 4 & 5 & 3 & 4 & 5 & 2 & 4 & 4\\
$10^{-3}$   &  7 & 7 & 9 & 6 & 7 &8 & 5 & 6 & 7 & 5 & 6 & 7\\
$10^{-4}$   & 9 & 10 & 12 & 8 & 9 &11 & 7 & 9 & 10 & 6 & 8 & 9 \\
$10^{-5}$   &  11 & 13 & 15 & 10 &12 &14 & 9 & 11& 13 & 8 &11 & 12\\
$10^{-6}$   &  13 & 15 & 18 & 12 &14 & 17 & 11& 14 & 16& 10 & 13 & 15\\
$10^{-7}$   & 15 & 17 & 21 & 14 & 17 &20 & 13 & 16 & 19 & 12 & 16 & 18\\
$10^{-8}$   & 17 & 20 & 24 & 16 & 19 &23 & 15 & 19 & 22 & 14 & 18 & 21\\
\hline\hline
 \end{tabular}
\end{lrbox}
\scalebox{0.95}{\usebox{\tablebox}}
\end{table}

\subsubsection{Equivalence between SOE approximations to the power creep and a rheological model}

In fact, the error bound \cref{soeappr} offers a rigorous mathematical justification for the capacity of rheological models to approximate the frequency-independent Q behavior. The fractional stress-strain relation \cref{fractional_stress_strain} is equivalent to 
 \begin{equation*}\label{SOE_stress_strain}
\begin{split}
\sigma(x, t) &= \frac{\rho(x) C_P(x)}{\Gamma(1 - 2\gamma(x))} \left[\frac{1}{t^{2\gamma(x)}} H(t) \ast \frac{\partial}{\partial t} \varepsilon(x, t)\right] \\
&\approx E_R(x) \sum_{j=1}^{\nexp} w_j  \int_{0}^{t} e^{- s_j(t - \tau) } \left[\frac{\partial}{\partial \tau} \varepsilon(x, \tau)\right] \D \tau = \underbrace{E_R(x) \sum_{j=1}^{\nexp} \frac{\tau_{\varepsilon_j}}{\tau_{\sigma_j}} e^{-\frac{t}{\tau_{\sigma_j}}} H(t)}_{\textup{modulus of GMB}} \ast  \frac{\partial}{\partial t} \varepsilon(x, t),
\end{split}
\end{equation*}
where $E_R(x)  = \rho(x) C_P(x)/\Gamma(1-2\gamma(x))$ is the equilibrium response of the viscoelastic material \cite{CaoYin2015}, $\tau_{\sigma_j} = s_j^{-1}$ and $\tau_{\varepsilon_j} = w_j \tau_{\sigma_j}$ are regarded as the relaxation times, $H(t)$ is the Heaviside step function and $\ast$ denotes the convolution in time. In practice, the first node $s_1$ is very close to $0$. As a result, the finite-node system \eqref{SOE_stress_strain} can be directly interpreted as a GMB \cite{ParkSchapery1999,SchmidtGaul2006,CaoYin2015}, which consists of $\nexp$ Maxwell chains in parallel, as depicted in \cref{viscoelastic_model}. The force acting in $j$-th branch of spring with relaxation time $\tau_{\sigma_j} = s_j^{-1}$ produces the response $\Phi_j = \Phi(x, s_j, t) $, and the stress $\sigma$ is recovered by summing over $\Phi_j$ weighted by $w_j  $ and relaxation modulus $E_R(x)$.  
\begin{figure}[!ht]
\centering
\includegraphics[width=0.8\textwidth,height=0.37\textwidth]{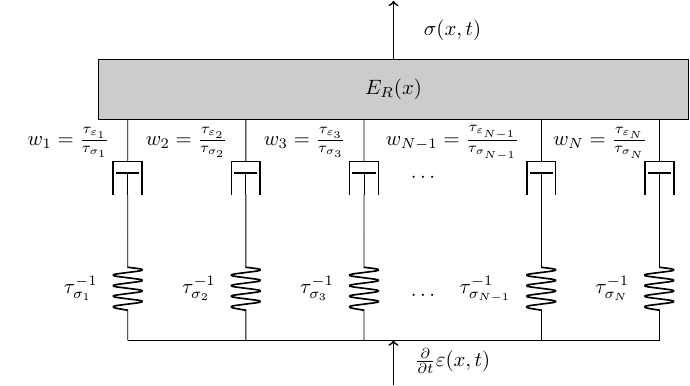}
    \caption{\small The connection between the SOE approximation of the fractional stress-strain relation and a classical rheological model. The coefficients of the SOE approximation are determined by a uniformly accurately curve fitting to the power creep function, which is adaptive to the constant-Q viscoelastic behavior. \label{viscoelastic_model}}
\end{figure} 

From \cref{curve_error}, it is evident that the new SOE approximation achieves a uniformly accurate approximation to the Caputo fractional derivative operator up to the final time $T$.
This contrasts with the behavior of projection errors in the Gauss-Laguerre quadrature rule, as adopted in the Yuan-Agrawal method
\cite{YuanAgrawal2002,BlancChiavassaLombard2014,XiongGuo2022},
which may amplify over time. Therefore, it partially addresses the criticism made in \cite{SchmidtGaul2006}. The accuracy of the SOE approximation, as well as the short-memory numerical scheme, can be ensured by allowing more flexibility in the parameters of the SOE approximation to adapt to the viscoelastic behaviors of the provided material data.

\subsection{Application to P- and S-wave propagation in 3D media}
We outline the changes in our scheme for solving the constant-Q wave equation
\cref{conservation_momentum} --\cref{stress_strain_relation}.  
The whole scheme can be easily applied to solve the P- and S-wave propagation, where the vector form of fractional stress-strain relation involves two Caputo fractional derivative operators. Starting with the SOE approximations for two different power creeps,  
\begin{equation}
  t^{-2\gamma_P}\approx \sum_{l=1}^{M_P} w_{l}^{(P)} e^{-s_l^{(P)} t},
  \quad t^{-2\gamma_S}\approx \sum_{l=1}^{M_S} w_{l}^{(S)} e^{-s_{l}^{(S)} t},
\end{equation}
we define $M_P$ auxiliary functions for the trace of strain,
\begin{equation*}
  \Phi_{123}(\bx, s_l^{(P)}, t) = \int_{0}^{t} e^{- s_l^{(P)}(t - \tau) } \left[\frac{\partial}{\partial \tau} (\varepsilon_{11}(\bx, \tau)+\varepsilon_{22}(\bx, \tau)+\varepsilon_{33}(\bx, \tau))\right] \D \tau,
  \quad l=1,\ldots, M_P,
\end{equation*}
and $6M_S$ auxiliary functions for the other six strain components,
\begin{equation*}
\Phi_{ij}(\bx, s_l^{(S)}, t) = \int_{0}^{t} e^{- s_l^{(S)}(t - \tau) } \left[\frac{\partial}{\partial \tau} \varepsilon_{ij}(\bx, \tau)\right] \D \tau, \quad i\ne j,~ 1\le i, j \le 3,~l=1,\ldots, M_S.
\end{equation*}
Thus, it requires storing $M_P + 6M_S$ memory variables in total.

Similar to \cref{stiff_equation}, $\Phi_{123}(\bx, s_l^{(P)}, t)$ and $\Phi_{ij}(\bx, s_l^{(S)}, t) $
satisfy the differential equations (without fractional derivatives)
\begin{align} 
  &\frac{\partial \Phi_{123}(\bx, s_l^{(P)}, t)}{\partial t}  = - s_l^{(P)}\Phi_{123}(\bx, s_l^{(P)}, t) +\frac{\partial v_{1}(\bx, t)}{\partial x_1}  + \frac{\partial v_{2}(\bx, t)}{\partial x_2}  + \frac{\partial v_{3}(\bx, t)}{\partial x_3} , \label{dynamics_phi_trace}\\
&  \frac{\partial \Phi_{ij}(\bx, s_l^{(S)}, t) }{\partial t} = - s_l^{(S)}\Phi_{ij}(\bx, s_l^{(S)}, t) + \frac{1}{2}\left(\frac{\partial v_{j}(\bx, t)}{\partial x_i}  + \frac{\partial v_{i}(\bx, t)}{\partial x_j} \right), \; i\ne j,\; 1 \le i, j \le 3,\label{dynamics_phi}
\end{align}
and the initial conditions
\begin{align}
&\Phi_{123}(\bx, s_l^{(P)}, t=0^+) = e^{-s_l^{(P)}}\frac{\sigma_{11}(\bx, t=0^+) + \sigma_{22}(\bx, t=0^+) + \sigma_{33}(\bx, t=0^+)}{(\Gamma(1-2\gamma_P(\bx)))^{-1}  \rho(\bx) C_P(\bx) }, \\
&\Phi_{ij}(\bx, s_l^{(S)}, t=0^+) = e^{-s_l^{(S)}}\frac{ \sigma_{ij}(\bx, t=0^+)}{2(\Gamma(1-2\gamma_S(\bx)))^{-1}\rho(\bx) C_S(\bx) }, \quad i\ne j, \quad 1 \le i, j \le 3.
\end{align}
The short-memory operator splitting scheme for 3D viscoelastic wave equation is summarized in Algorithm \ref{SMOS}, where $v^{n}$, $\sigma^n$, $\Phi_{123}^n$ and $\Phi_{ij}^n$ denote the numerical solutions at $t = t_n$. Note that the first node $s_1$ is usually very close to $0$ (see the tables of SOE formulae in our supplementary material). To avoid the numerical instability, we use the approximations $e^{-s_1 \Delta t} \approx 1$ and  ${(1 - e^{-s_1 \Delta t})}/{s_1} \approx \Delta t$. 

\begin{algorithm}[!ht] 
\caption{Short-memory operator splitting scheme for 3D viscoelastic wave equation\label{SMOS}}

\vspace{2mm}
\begin{itemize}
\item[1.] Half-step update of the velocity components:
\begin{equation*}
v_i^{n+\frac{1}{2}}(\bx) = v_i^n(\bx) + \frac{1}{\rho(\bx)} \left(\frac{\Delta t}{2}  \sum_{j=1}^3\frac{\partial}{\partial x_j} \sigma^n_{ij}(\bx) + \int_{t_n}^{t_{n}+\frac{\Delta t}{2}} f_i(\bx, \tau) \D \tau \right), \quad  i =1, 2, 3.
\end{equation*}

\item[2.] Full-step update of auxiliary functions for $l_1=1, 2\dots, M_P$, $ l_2=1, 2\dots, M_S$, $i, j = 1, 2, 3$:
\begin{equation*}
\begin{split}
&\Phi_{123}^{n+1}(\bx, s_{l_1}^{(P)}) = e^{-s_{l_1}^{(P)} \Delta t}\Phi_{123}^n(\bx, s_{l_1}^{(P)}) +  \frac{1 - e^{-s_{l_1}^{(P)} \Delta t}}{s_{l_1}^{(P)} }\sum_{j=1}^3\frac{\partial}{\partial x_j}  v_j^{n+\frac{1}{2}}(\bx),\\
&\Phi_{ij}^{n+1}(\bx, s_{l_2}^{(S)}) = e^{-s_{l_2}^{(S)} \Delta t}\Phi_{ij}^n(\bx, s_{l_2}^{(S)}) +  \frac{1 - e^{-s_{l_2}^{(S)} \Delta t}}{2s_{l_2}^{(S)}}\left(\frac{\partial}{\partial x_i}  v_j^{n+\frac{1}{2}}(\bx) + \frac{\partial}{\partial x_j}  v_i^{n+\frac{1}{2}}(\bx)\right).
\end{split}
\end{equation*}

\item[3.] Update of the stress components:
\begin{equation*}
\begin{split}
&\sigma^{n+1}_{11}(\bx) = E_R^{(P)}(\bx)\sum_{l=1}^{M_P} w_l^{(P)} \Phi_{123}^{n+1}(\bx, s_l^{(P)}) - 2 E_R^{(S)}(\bx) \sum_{l=1}^{M_S} w_l^{(S)} (\Phi^{n+1}_{22}+\Phi^{n+1}_{33})(\bx, s_l^{(S)}), \\
&\sigma^{n+1}_{22}(\bx) = E_R^{(P)}(\bx) \sum_{l=1}^{M_P} w_l^{(P)} \Phi_{123}^{n+1}(\bx, s_l^{(P)}) - 2 E_R^{(S)}(\bx) \sum_{l=1}^{M_S} w_l^{(S)} (\Phi^{n+1}_{11}+\Phi^{n+1}_{33})(\bx, s_l^{(S)}), \\
&\sigma^{n+1}_{33}(\bx) = E_R^{(P)}(\bx) \sum_{l=1}^{M_P} w_l^{(P)} \Phi_{123}^{n+1}(\bx, s_l^{(P)}) - 2 E_R^{(S)}(\bx) \sum_{l=1}^{M_S} w_l^{(S)} (\Phi^{n+1}_{11}+\Phi^{n+1}_{22})(\bx, s_l^{(S)}), \\
& \sigma^{n+1}_{ij}(\bx) = 2  E_R^{(S)}(\bx) \sum_{l=1}^{M_S} w_l^{(S)} \Phi_{ij}^{n+1}(\bx, s_l^{(S)}), \quad 1 \le i < j \le 3,
\end{split}
\end{equation*}
where $E_R^{(P)}(\bx) = \frac{\rho(\bx) C_P(\bx)}{\Gamma(1-2\gamma_P(\bx))}$ and $E_R^{(S)}(\bx) = \frac{\rho(\bx) C_S(\bx)}{\Gamma(1-2\gamma_S(\bx))}$.

\item[4.] Half-step update of the velocity components:
\begin{equation*}
v_i^{n+1}(\bx) = v_i^{n+\frac{1}{2}}(\bx) + \frac{1}{\rho(\bx)} \left(\frac{\Delta t}{2}  \sum_{j=1}^3\frac{\partial}{\partial x_j} \sigma^{n+1}_{ij}(\bx) + \int_{t_n + \frac{\Delta t}{2}}^{t_{n+1}} f_i(\bx, \tau) \D \tau \right), \quad  i =1, 2, 3.
\end{equation*}

\end{itemize}
\end{algorithm}

\section{Numerical experiments}
\label{sec.numerical}
We have implemented the scheme for solving the constant-Q wave equation
\cref{conservation_momentum}--\cref{stress_strain_relation} and its special
case - scalar viscoacoustic P-wave equation \cref{pwaveequation} in Fortran with OpenMP
shared-memory parallelization. All numerical results are obtained on a machine 
with AMD Ryzen 5950X (3.40GHz, 72MB Cache, 16 Cores, 32 Threads) and 128GB Memory@3600Mhz.

The Fourier spectral method is used for spatial discretization, while the exponential Strang operator splitting scheme is used for time evolution. To avoid the artificial wave reflection, we choose a sufficiently large computational domain such that the wavepacket does not reach the boundary during the entire simulation. For readers' convenience, we collect the nodes and weights of SOE approximations in our supplementary material with
the quality factor $Q=10, 32, 50, 100$, relative precision $\varepsilon = 10^{-2},\ldots, 10^{-7}$,
$T = 10$ and $\Delta t = 0.005$. 
To measure the numerical error, we use the relative $L^2$-error $\varepsilon_2[v](t)$ and the relative maximum error $\varepsilon_{\infty}[v](t)$,
\begin{equation*}
\begin{split}
\varepsilon_2[v](t) = \frac{\left(\int_{\Omega} |v^{\textup{num}}(\bx, t) - v^{\textup{ref}}(\bx, t)|^2\D \bx \right)^{1/2}}{\max_{\bx}|v^{\textup{ref}}(\bx, t)|}, \quad \varepsilon_\infty[v](t) = \frac{\max_{\bx \in \Omega} |v^{\textup{num}}(\bx, t) - v^{\textup{ref}}(\bx, t)|}{\max_{\bx}|v^{\textup{ref}}(\bx, t)|},
\end{split}
\end{equation*}
where $v^{\textup{num}}$ and $v^{\textup{ref}}$ denote the numerical and reference velocity wavefields, respectively.

\subsection{3D viscoacoustic wave equation}
We first study the scalar viscoacoustic P-wave equation \cref{pwaveequation}.
The initial condition is chosen to be $v_0(\bx) = e^{-|\bx|^2}$, and the analytical
solution (see \cref{exact_solution}) is given by the formula
\begin{equation}\label{analytical}
v(\bx, t) = -\frac{1}{4\pi C_P t^{2-2\gamma_P}} \int_{\mathbb{R}^3} \frac{e^{-|\bx - \by|^2}}{|\by|} \frac{\partial  M_{1 - \gamma_P}(z)}{\partial z} \Big |_{z = \frac{|\by|}{\sqrt{C_P} t^{1-\gamma_P}}}   \D \by, 
\end{equation}
where the Mainardi function $M_{l}(z)$ is a special Wright function of the second kind \cite{bk:Mainardi2010} (see Appendix \ref{Green}). Here we focus on the solution on the line $\bx = (0, 0, x_3)$,
\begin{equation*}
\begin{split}
v(\bx, t) & =  -\frac{e^{-|x_3|^2}}{4\pi C t^{2-2\gamma_P}} \int_{0}^{+\infty} \D r\int_{0}^{\pi} \D \theta \int_0^{2\pi} \D \phi ~e^{-2x_3 r \cos\theta} r e^{-r^2} \sin \theta \left(\frac{\partial  M_{1 - \gamma_P}(z)}{\partial z} \Big |_{z = \frac{r}{\sqrt{C_P} t^{1-\gamma_P}}}\right) \\
& = -\frac{1}{4 C_P t^{2-2\gamma_P}} \int_{0}^{+\infty} \frac{1 - e^{-4x_3 r} }{x_3} e^{-(x_3-r)^2} \left(\frac{\partial  M_{1 - \gamma_P}(z)}{\partial z} \Big |_{z = \frac{r}{\sqrt{C_P} t^{1-\gamma_P}}}\right)  \D r.
\end{split}
\end{equation*}

To seek an accurate approximation to the analytical solution, we first truncate the $r$-domain by utilizing the rapid decay of $\frac{\partial  M_{1 - \gamma}(z)}{\partial z}$, then apply the Gauss-Legendre quadrature on the finite interval to obtain
\begin{equation*}\label{convolution_approx}
\begin{split}
v(\bx, t)&\approx
\left\{
\begin{split}
&\frac{1}{C_P t^{2-2\gamma_P}} \sum_{j = 1}^{N_{H}} \omega_j^{H} \frac{1 - e^{-4x_3 r_j}}{4x} e^{-(x_3-r_j)^2}\left(-\frac{\partial  M_{1 - \gamma_P}(z)}{\partial z} \Big |_{z = \frac{r_j}{\sqrt{C_P} t^{1-\gamma_P}}} \right), \quad x_3 > 0, \\
&\frac{1}{C_P t^{2-2\gamma_P}}  \sum_{j = 1}^{N_{H}} \omega_j^{H} r_j e^{-(x_3-r_j)^2} \left(-\frac{\partial  M_{1 - \gamma_P}(z)}{\partial z} \Big |_{z = \frac{r_j}{C_P t^{1-\gamma_P}}} \right), \quad x_3 = 0,
\end{split}
\right.
\end{split}
\end{equation*}
where $(r_j, \omega_j^H)_{j=1}^{N_G}$ are collocation points and weights, respectively. The Mainardi functions are calculated using the algorithm in \cite{XiongGuo2022}, where the saddle point approximation \eqref{saddle_point} is adopted for large $|z|$.
This can achieve a reference solution with a precision of about $10^{-8}$,
though the accuracy is still limited by the truncation error in the asymptotic expansion.
Therefore, we also use the numerical results with the SOE approximation with $10$ digits of accuracy to check the numerical accuracy.
\begin{figure}[!ht]
    \centering
    \subfigure[The relative $L^2$-error (left) and the relative maximum error (right) at $t = 8$. \label{convergence_Nexp}]{
    {\includegraphics[width=0.49\textwidth,height=0.26\textwidth]{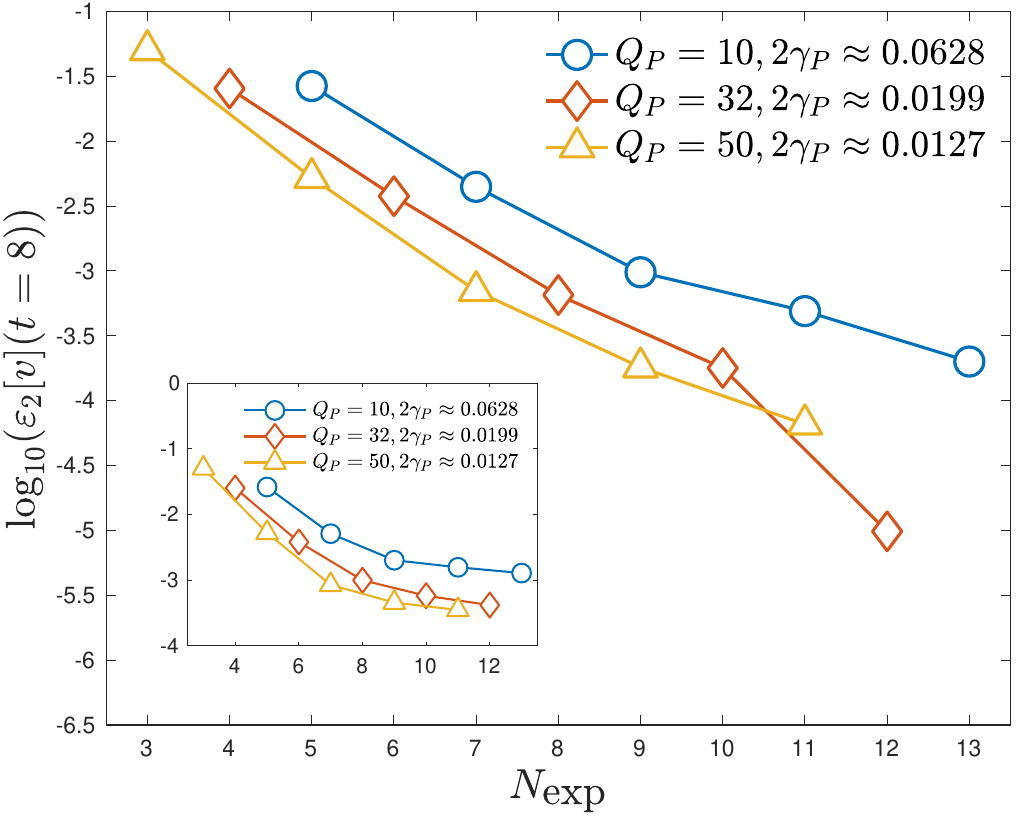}}
    {\includegraphics[width=0.49\textwidth,height=0.26\textwidth]{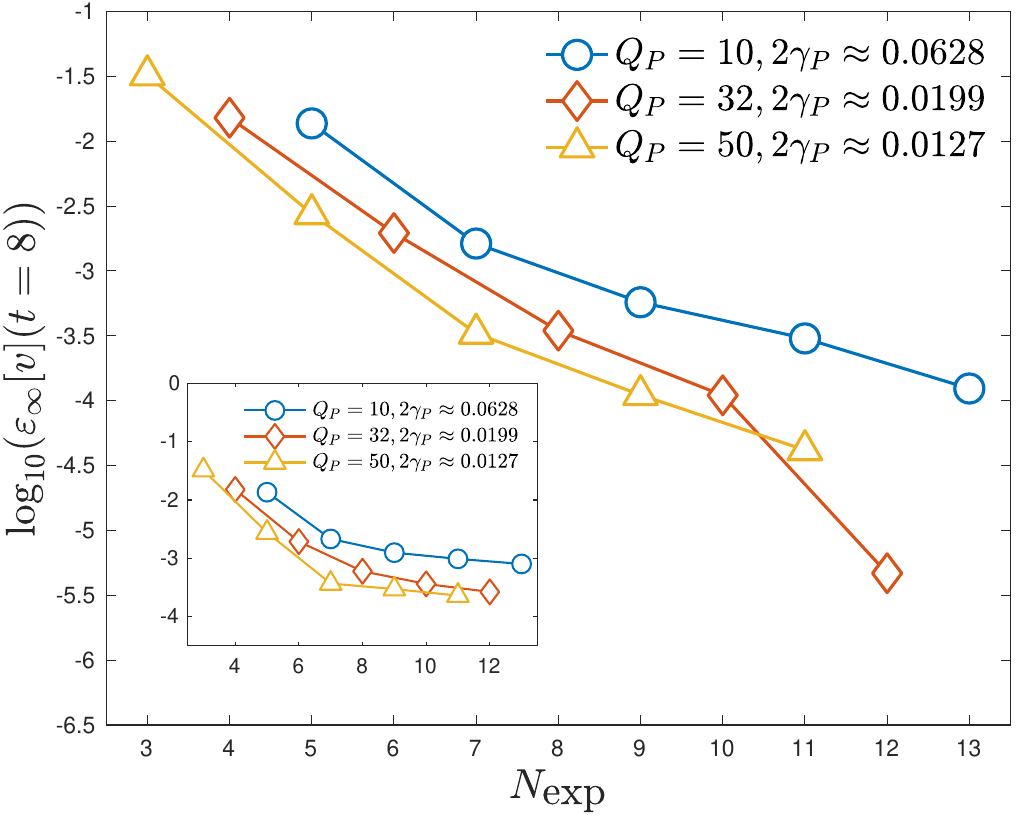}}}
    \\
    \centering
    \subfigure[The wavefield $v(0, 0, x_3)$ (left) and the relative error (right) when $Q_P=10$. \label{acoustic_error_QP10}]{{\includegraphics[width=0.49\textwidth,height=0.26\textwidth]{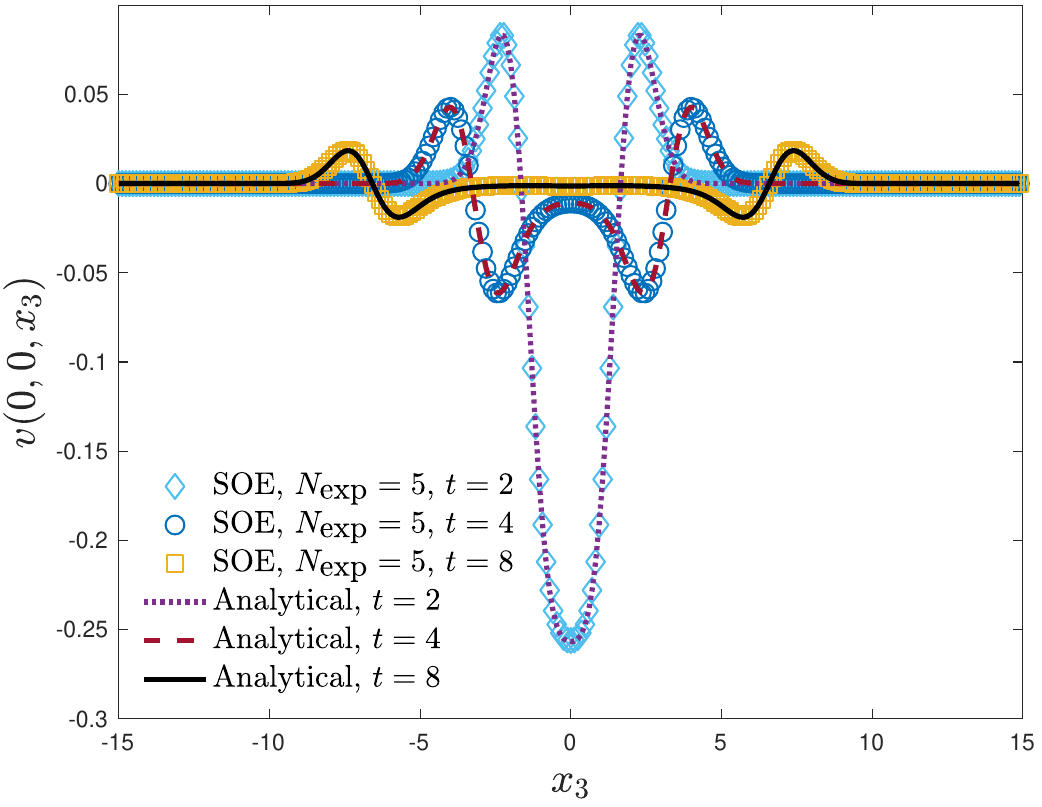}}
     {\includegraphics[width=0.49\textwidth,height=0.26\textwidth]{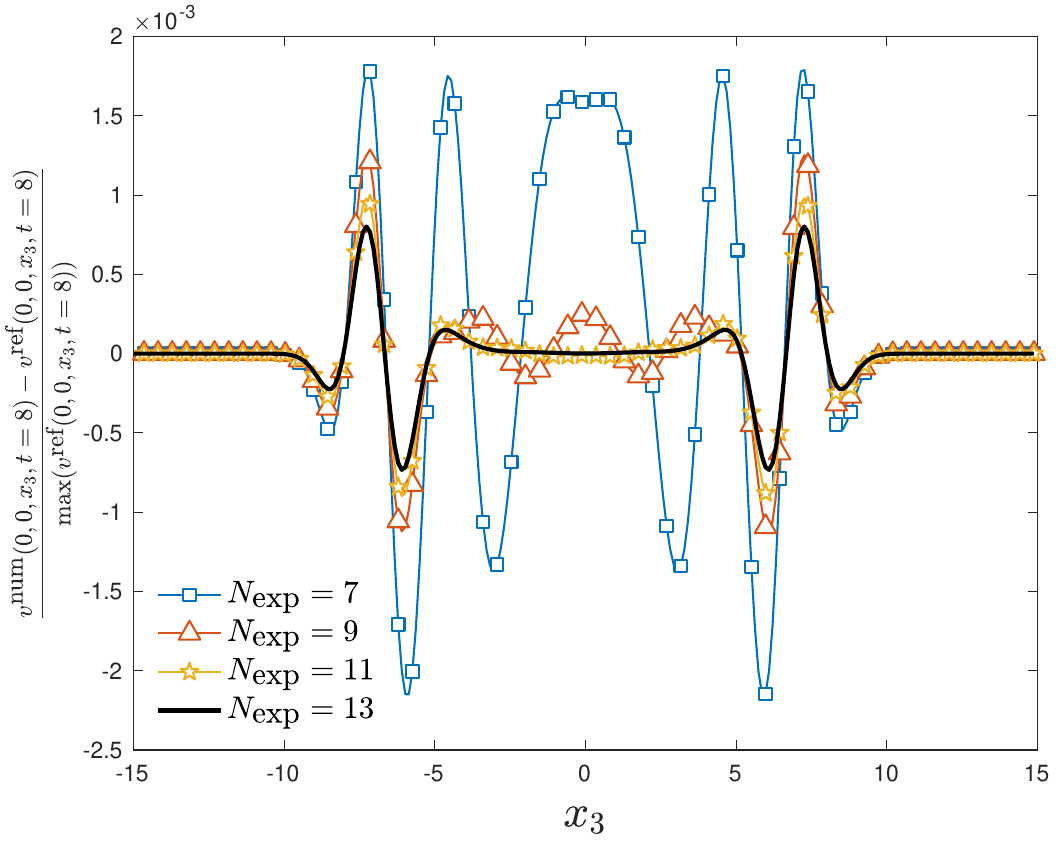}}}
    \\
    \subfigure[The wavefield $v(0, 0, x_3)$ (left) and the relative error (right) when $Q_P=32$. \label{acoustic_error_QP32}]{{\includegraphics[width=0.49\textwidth,height=0.26\textwidth]{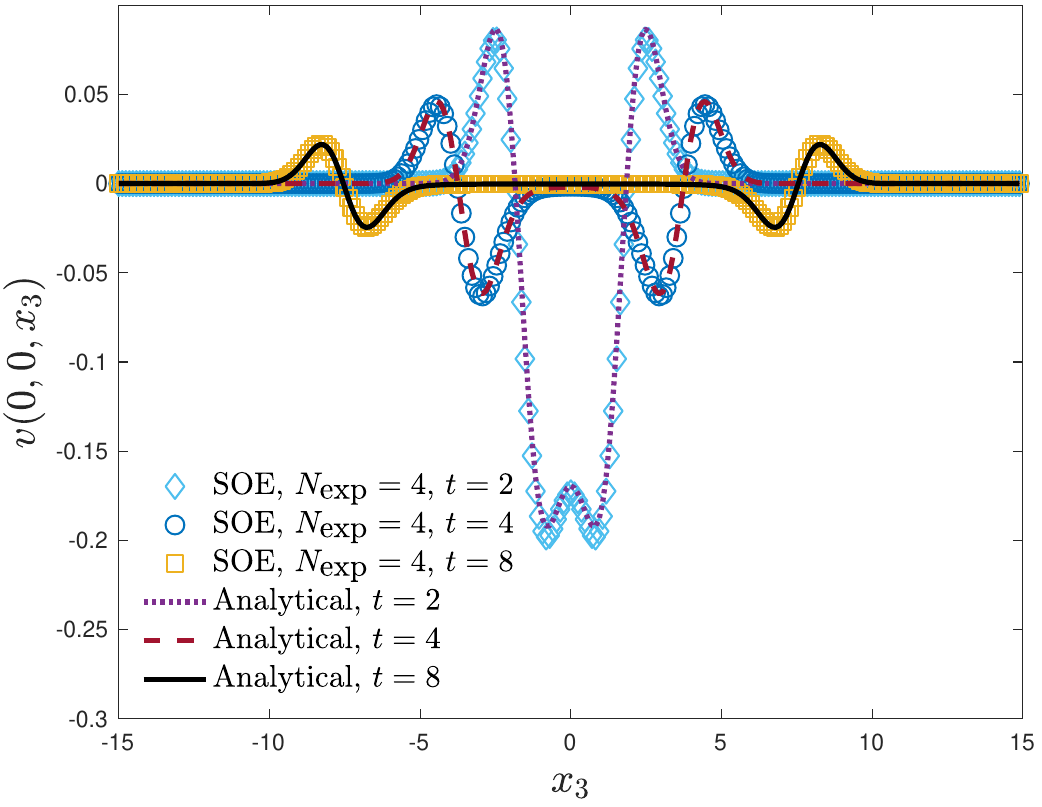}}
    {\includegraphics[width=0.49\textwidth,height=0.26\textwidth]{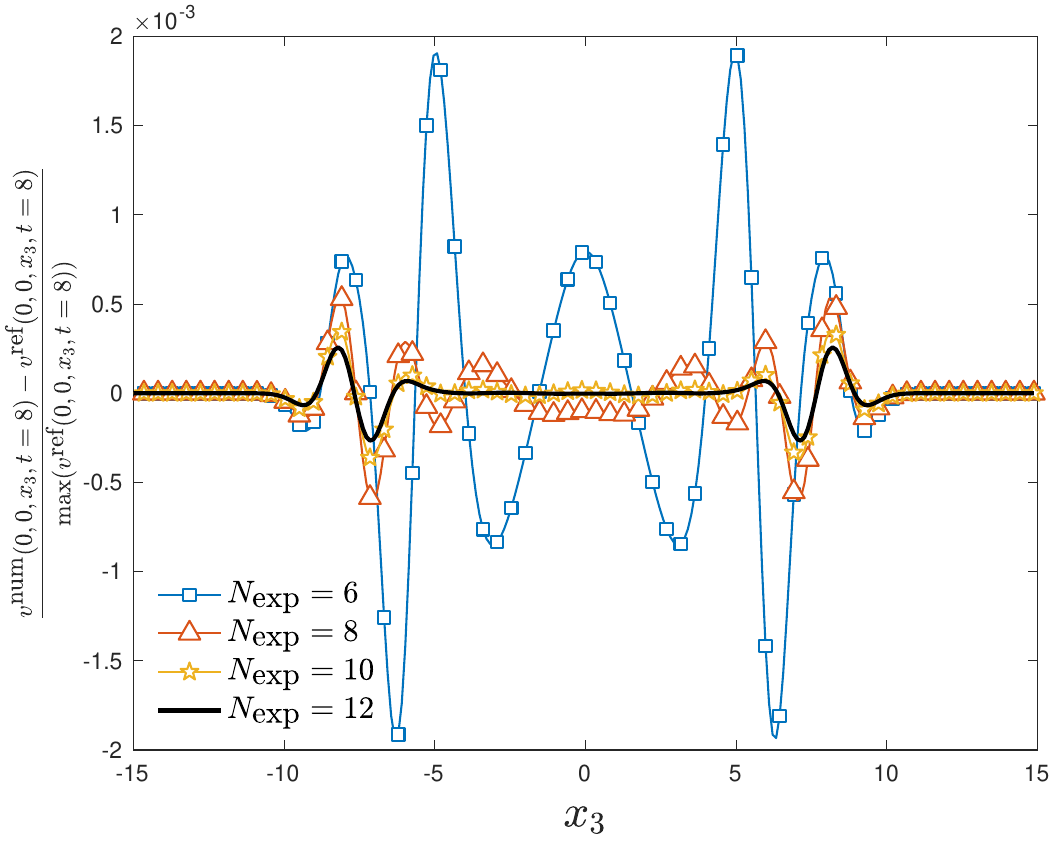}}}
        \\
    \subfigure[The wavefield $v(0, 0, x_3)$ (left) and the relative error (right) when $Q_P=50$. \label{acoustic_error_QP50}]{{\includegraphics[width=0.49\textwidth,height=0.26\textwidth]{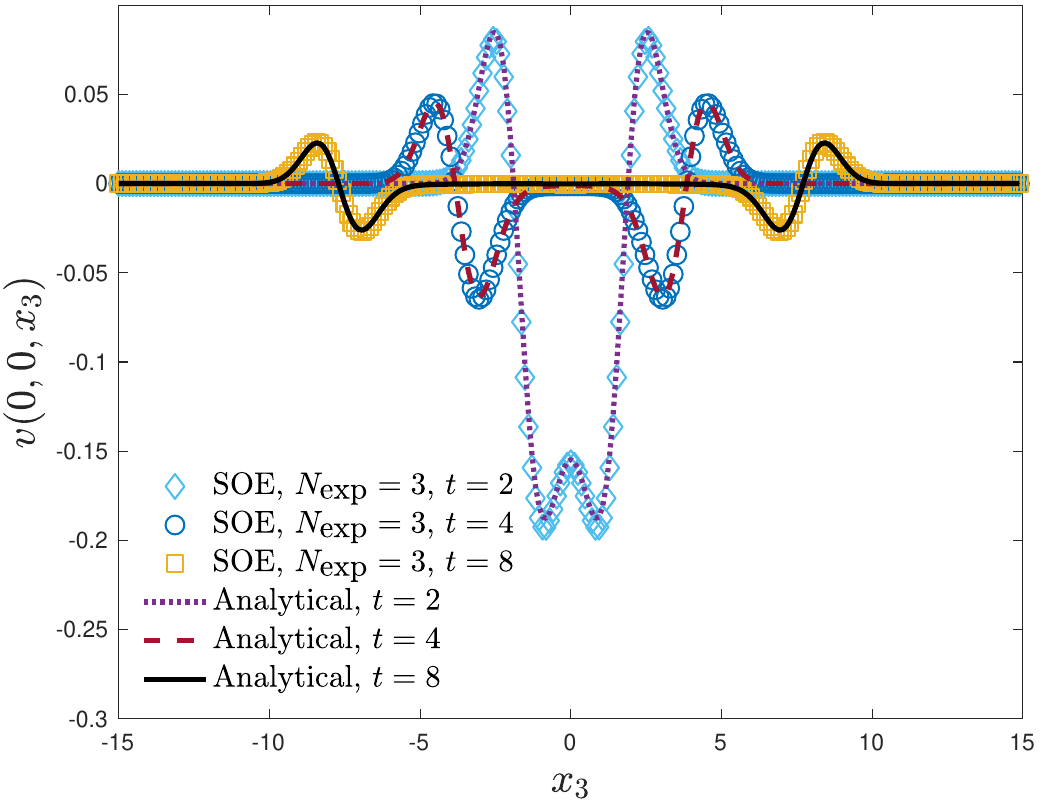}}
    {\includegraphics[width=0.49\textwidth,height=0.26\textwidth]{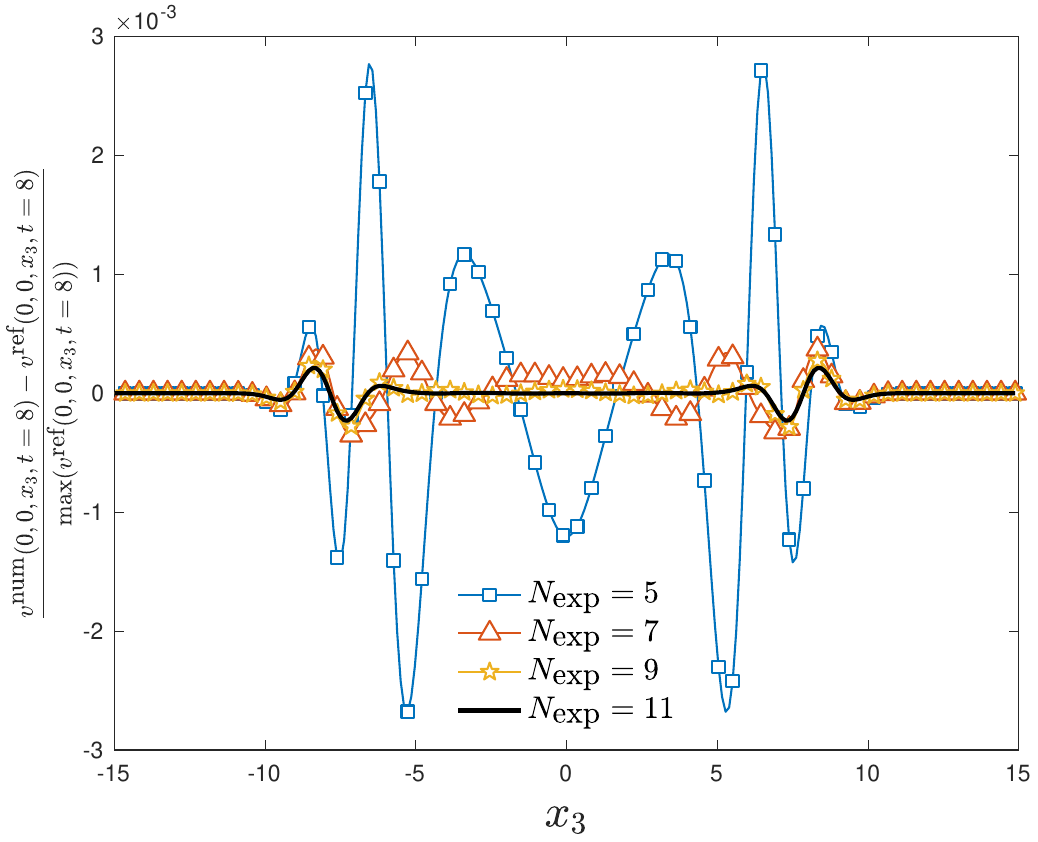}}}
    \caption{\small 3-D viscoacoustic wave equation: The rapid convergence with respect to the memory length $\nexp$ is verified, where either numerical solutions produced under $N^3 = 256^3$ and $\Delta t = 0.001$ or the analytical solutions (in panel) are adopted as the reference. For $Q_P \ge 10$, the new SOE can achieve a uniform relative error less than $10^{-3}$ with $N_{\exp} < 10$.    
     \label{SOE_error_Q}}
\end{figure} 

\begin{figure}[!ht]
    \centering
    \subfigure[Second-order convergence in Strang splitting.]{{\includegraphics[width=0.49\textwidth,height=0.26\textwidth]{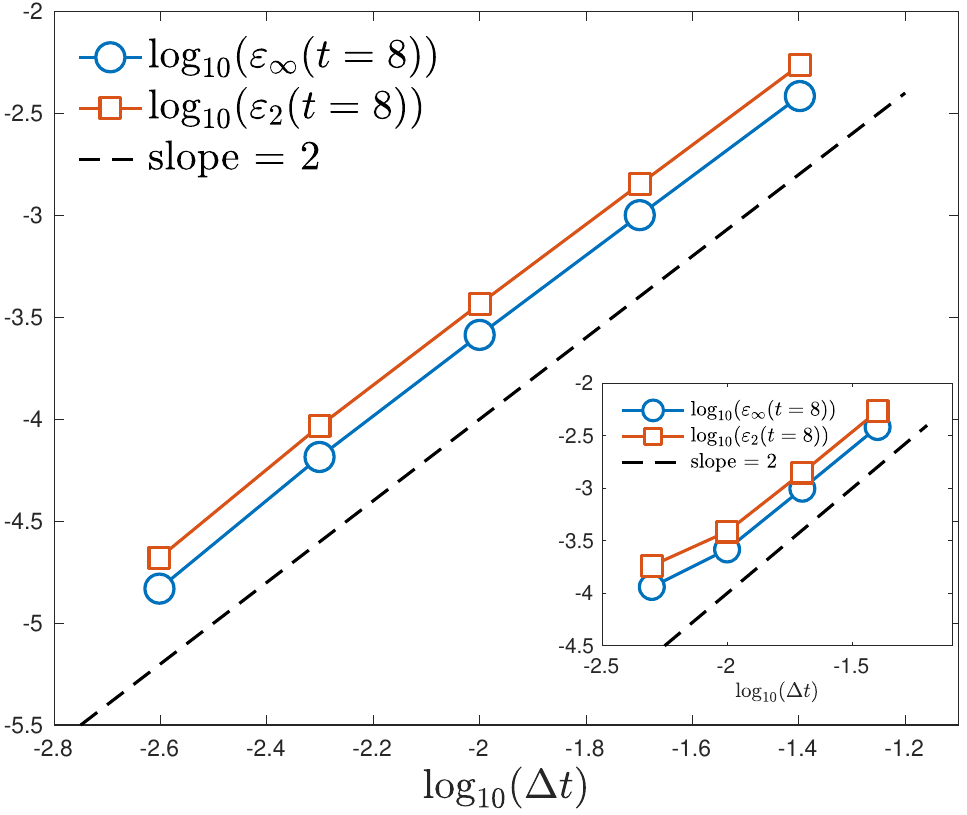}}}
   \subfigure[Spectral convergence in spatial discretization.]{{\includegraphics[width=0.49\textwidth,height=0.26\textwidth]{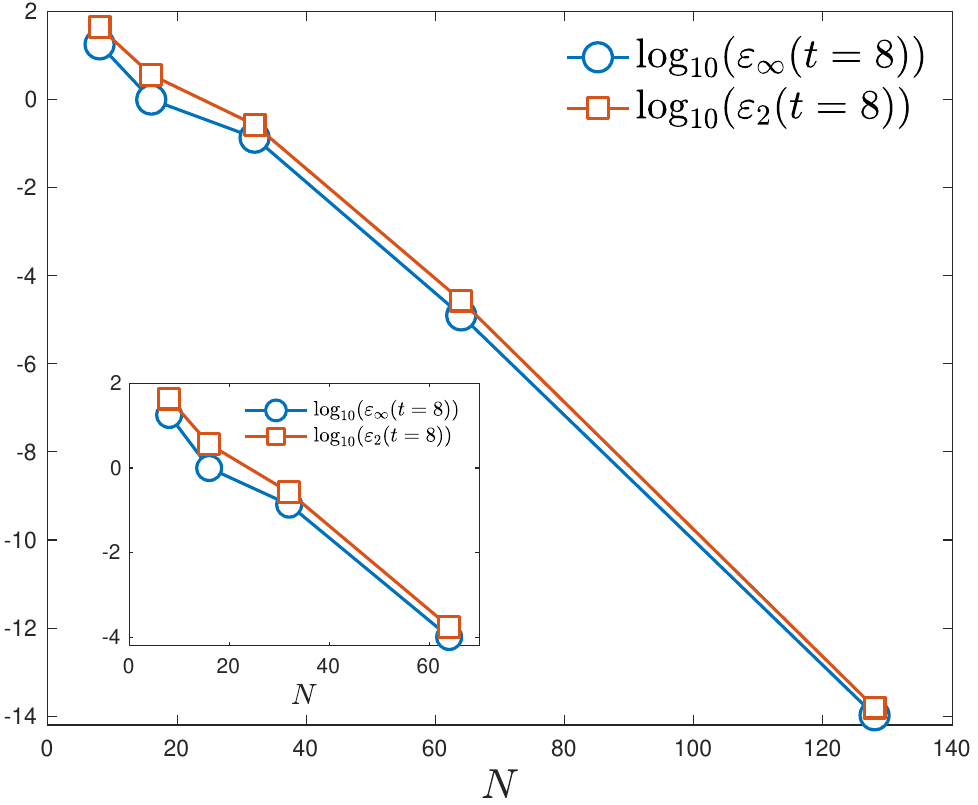}}}
    \caption{\small 3-D viscoacoustic wave equation: Convergence of the Strang splitting and the Fourier spectral method is verified, where either numerical solutions produced under $N^3 = 256^3$ and $\Delta t = 0.001$ or the analytical solutions (in panel) are adopted as the reference.\label{spatial_temporal_convergence}}
\end{figure} 

As the first step, we need to verify the convergence of the exponential Strang splitting and the Fourier spectral method. The model parameters are set as follows: the group velocity $c_P =1$km/s, the quality factor $Q_P = 10$, the mass density $\rho = 1\textup{g}/\textup{cm}^3$, and the reference frequency $\omega_0 = 2\pi \times 100$Hz. The computational domain is $[-15, 15]^3$, and several groups of simulations under different grid meshes $N^3 = 8^3,16^3, 32^3, 64^3, 128^3$ and time stepsizes $\Delta t= 0.04, 0.02, 0.01, 0.005, 0.0025$ are performed up to the final time $T=8$s, where numerical solutions produced under $N^3 = 256^3$ and $\Delta t= 0.001$ are used as the reference.  As presented \cref{spatial_temporal_convergence}, it exhibits second-order convergence in time and spectral convergence in space. We also use analytical solutions as the reference and plot the error curve in the panel of \cref{spatial_temporal_convergence} as a comparison, and find that the convergence trend is almost the same.

More importantly, it is necessary to investigate the convergence with respect to the memory length $\nexp$ for different quality factors $Q_P$. To this end, we adopt the same parameters: the group velocity $c_P = 1$km/s, the mass density $\rho =1\textup{g}/\textup{cm}^3$, the reference frequency $\omega = 2\pi \times 100$Hz, and the computational domain $[-15, 15]^3$ under the grid mesh $N^3 = 256^3$ and time stepsize $\Delta t = 0.005$s. Three groups of simulations with $Q_P = 10, 32, 50$ are performed and the numerical results are collected in \cref{SOE_error_Q}, where the reference is chosen to be the numerical solutions with $\nexp = 19, \varepsilon=10^{-9}$ for $Q_P = 10$, $\nexp = 18, \varepsilon=10^{-9}$ for $Q_P = 32$, $\nexp = 17, \varepsilon=10^{-9}$ for $Q_P = 50$, respectively. To make numerical results more convincing, we also calculate the analytical solutions as the reference and plot the error curves in the panels of \cref{SOE_error_Q}. In both situations, rapid convergence with respect to $\nexp$ is verified. Moreover, for a larger $Q_P$, fewer memory wavefields are needed to attain the same precision, as clearly observed in \cref{SOE_error_Q}. Therefore, the stagnation in the SOE approximation has indeed been avoided.

A comparison between the numerical wavefield $v(0, 0, x_3)$ at $t=8$s and the analytical
solution \eqref{analytical} with different $Q_P$,  as well as the visualization of the relative errors,  is given in \cref{acoustic_error_QP10}--\cref{acoustic_error_QP50}. The coincidence between the two different approaches is observed with only a few memory variables. The accuracy can be systematically improved with a larger $\nexp$. Moreover, even for the strong attenuation case ($Q_P = 10$),
our SOE approximation can still achieve a relative error less than $10^{-3}$ with $N_{\exp} < 10$.  Therefore, the efficient compression of memory wavefields in new SOE approximation may significantly reduce
 the memory requirement and computational complexity when solving the time-fractional wave equation.

\subsection{Viscoelastic wave propagation in 3D  homogeneous media}
\label{sec.3d.homo}
We now study the vector wave equation \cref{conservation_momentum}--\cref{stress_strain_relation}.
To study the propagation of P- and S-wave in viscoelastic media, we set the initial velocity and stress tensors to an equilibrium state and activate the velocities components with source functions $f_1(\bx, t) = f_2(\bx, t) = f_3(\bx, t) = A(\bx) f_r(t)$, which uses a Ricker-type wavelet history and a Gaussian profile $A(\bx)$,
\begin{equation}
A(\bx) = e^{-{|\bx-\bx^c|^2}}, \quad f_r(t) = (1 - 2(\pi f_P (t- d_r))^2 )e^{-(\pi f_P(t- d_r))^2},
\end{equation}
where $f_P$ is the peak frequency and $d_r$ is the temporal delay. Here we choose $f_P = 100$Hz, $d_r = 0$ and the center position $\bx^c = (0, 0, 10)$. The model parameters in homogenous media are set as follows: the group velocities for P-wave and S-wave  $c_P = 2.614$km/s and $c_S = 0.802$km/s, respectively, the mass density $\rho = 2.2 \textup{g}/\textup{cm}^3$, and the frequency frequency $\omega_0 = 2\pi \times 100$Hz \cite{Carcione2009}. The computational domain is $[-40, 40]^3$ ($80\textup{km}\times 80 \textup{km}\times 80 \textup{km}$), and the Fourier spectral method is adopted for spatial discretization with the grid size $N^3 = 256^3$. The reference solutions for $Q_P = 32$, $Q_S = 10$ are produced using the grid mesh $N^3 = 256^3$,  $\Delta t = 0.005$s,  $\varepsilon = 10^{-9}$ and $M_P = 18$, $M_S = 19$ for P- and S-waves, respectively. 
\begin{figure}[!ht]
\centering
\subfigure[$v_3(x_1, 0, x_3)$ at $4 \to 6 \to 10$s in viscoelastic (left) or elastic (right) medium. Here $M_P=4$, $M_S=5$.]{
{\includegraphics[width=0.32\textwidth,height=0.21\textwidth]{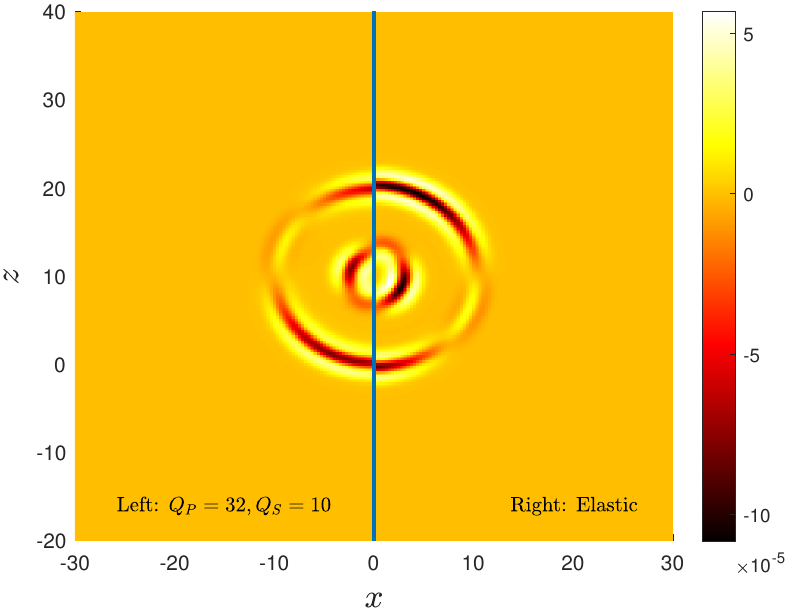}}
{\includegraphics[width=0.32\textwidth,height=0.21\textwidth]{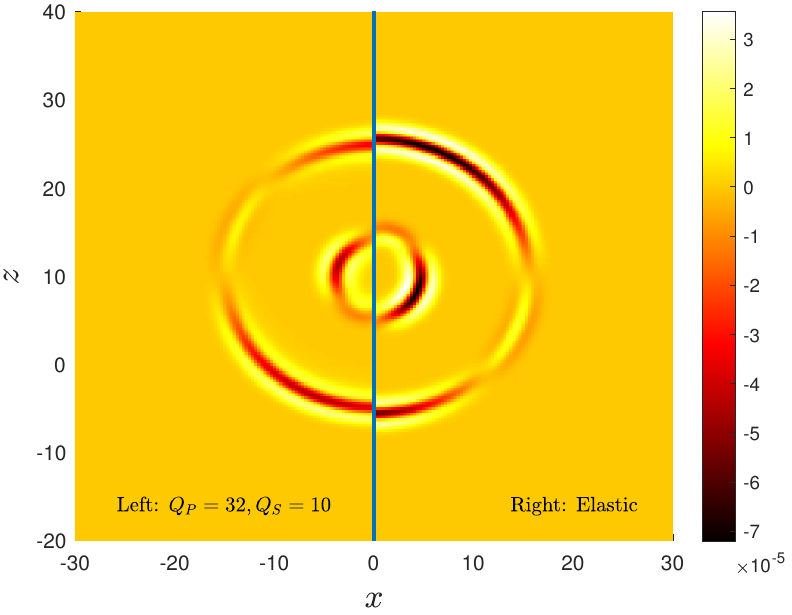}}
{\includegraphics[width=0.32\textwidth,height=0.21\textwidth]{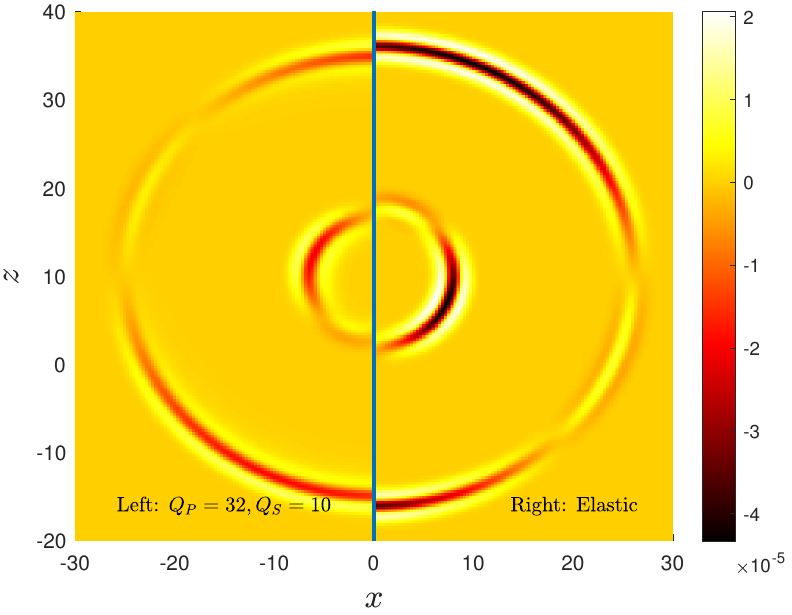}}}
\\
\centering
\subfigure[$v_3(x_1, 0, x_3)$ at $4 \to 6 \to 10$s in viscoelastic (left) or elastic (right) medium. Here $M_P=18$, $M_S=19$.]{
{\includegraphics[width=0.32\textwidth,height=0.21\textwidth]{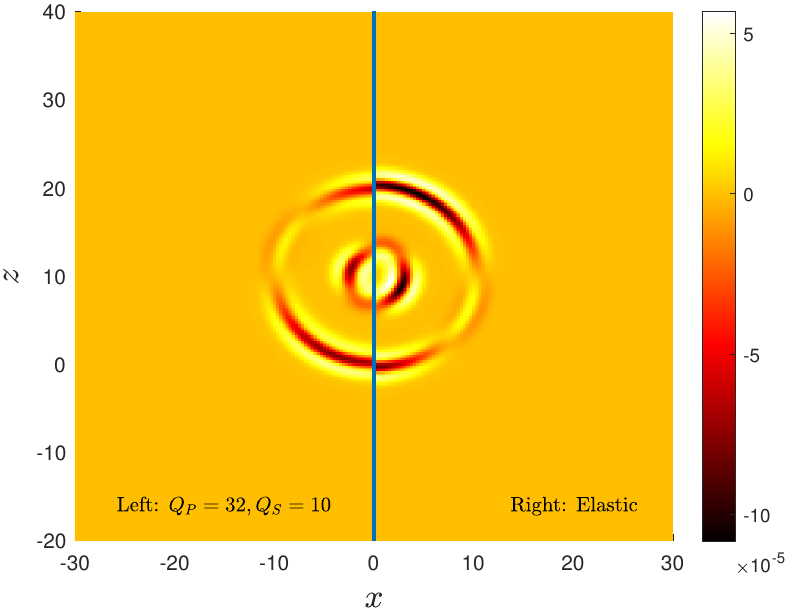}}
{\includegraphics[width=0.32\textwidth,height=0.21\textwidth]{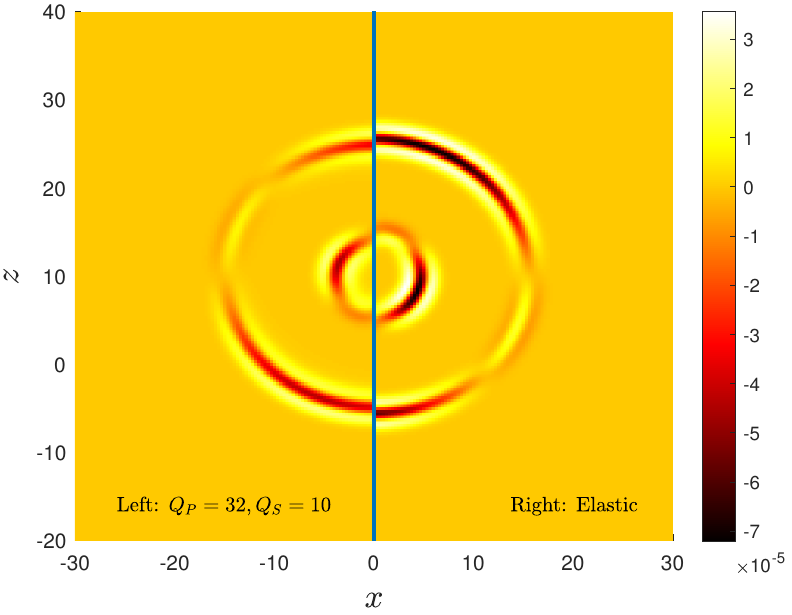}}
{\includegraphics[width=0.32\textwidth,height=0.21\textwidth]{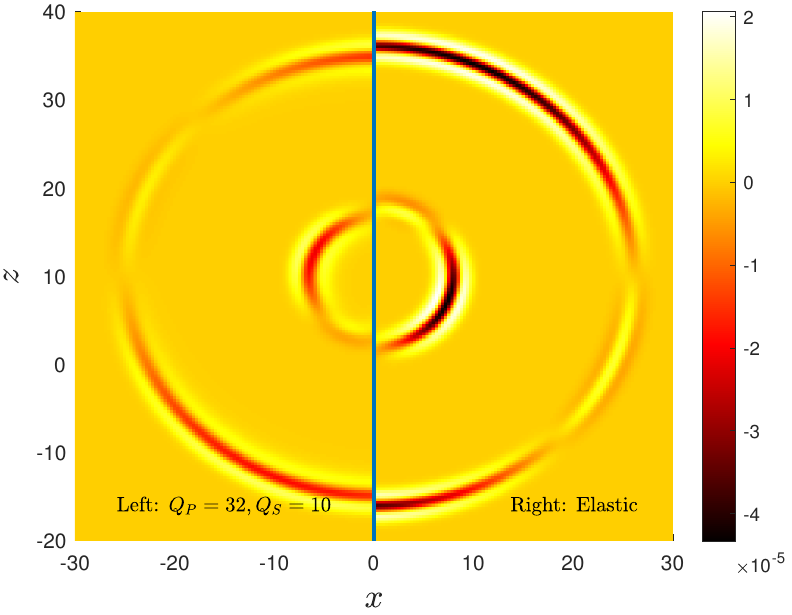}}}
\\
\centering
\subfigure[$v_3(x_1, 0, x_3)$ at $4 \to 6 \to 10$s in viscoelastic (left) or elastic (right) medium. Here $M_P=14$, $M_S=15$.]{
{\includegraphics[width=0.32\textwidth,height=0.21\textwidth]{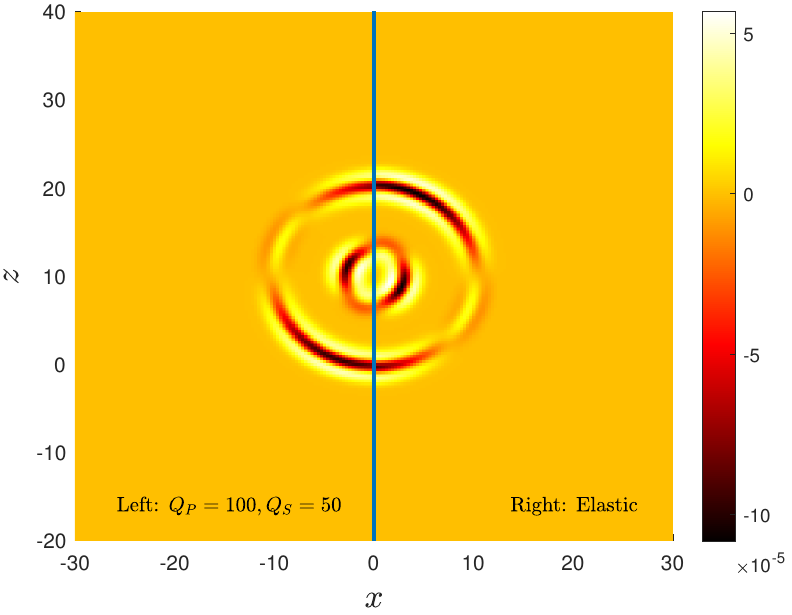}}
{\includegraphics[width=0.32\textwidth,height=0.21\textwidth]{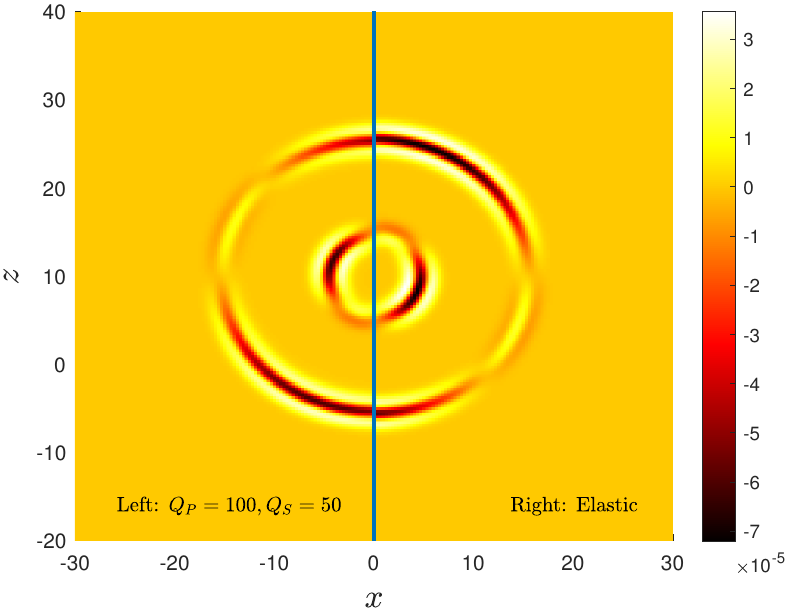}}
{\includegraphics[width=0.32\textwidth,height=0.21\textwidth]{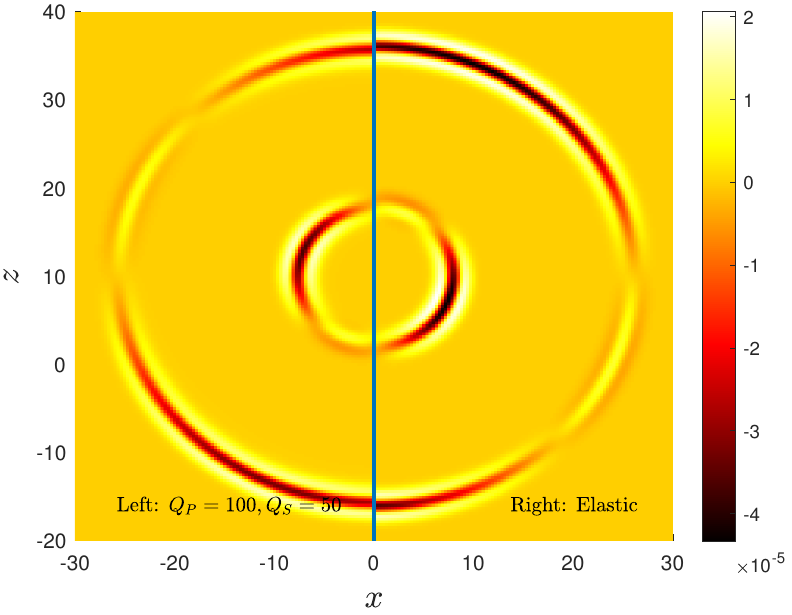}}}
\\
\centering
\subfigure[$v_3$ at $10$s. (left: $M_P=4$, $M_S=5$, middle: $M_P=18$, $M_S=19$, right: $M_P=14$, $M_S=15$).]{
{\includegraphics[width=0.32\textwidth,height=0.21\textwidth]{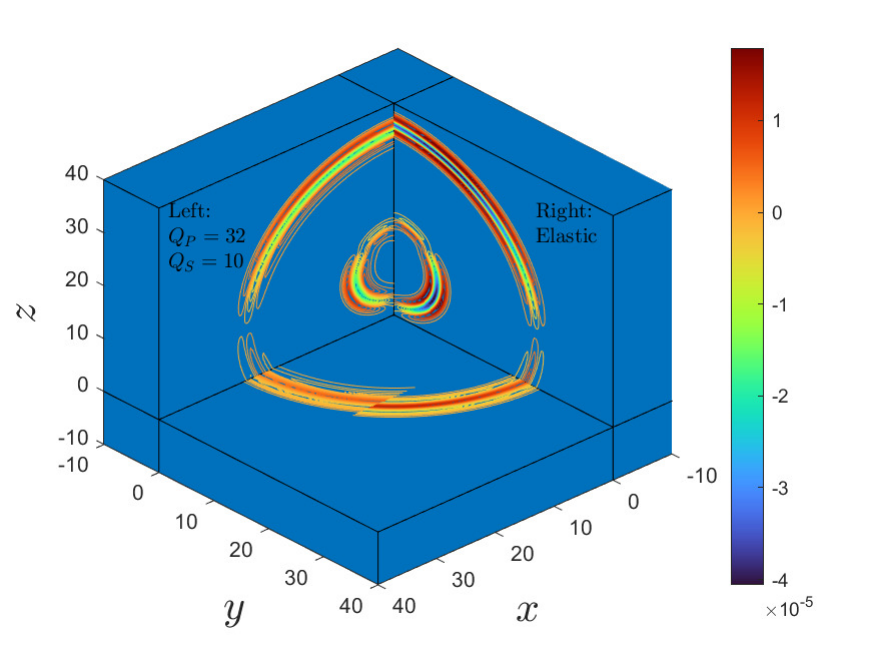}}
{\includegraphics[width=0.32\textwidth,height=0.21\textwidth]{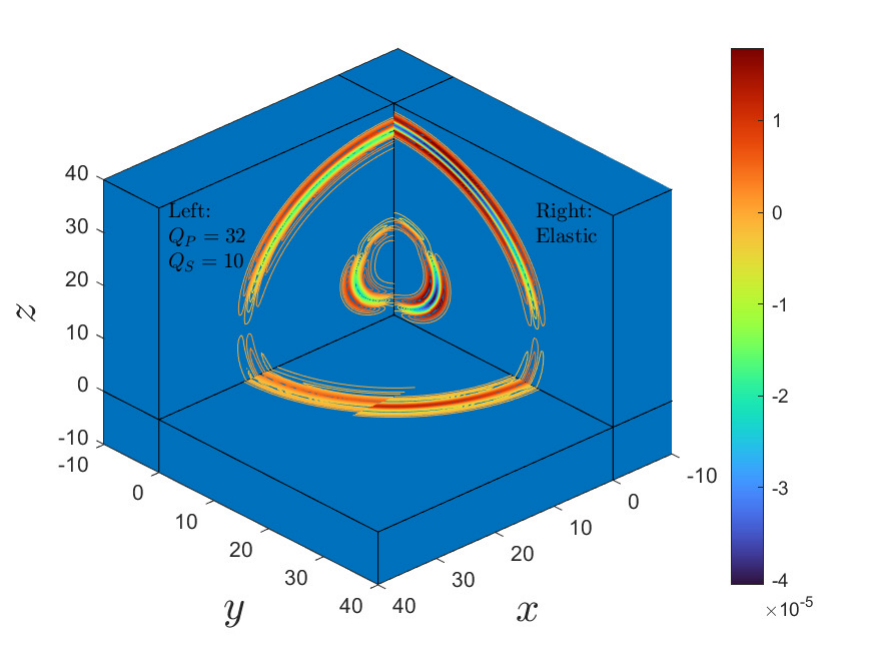}}
{\includegraphics[width=0.32\textwidth,height=0.21\textwidth]{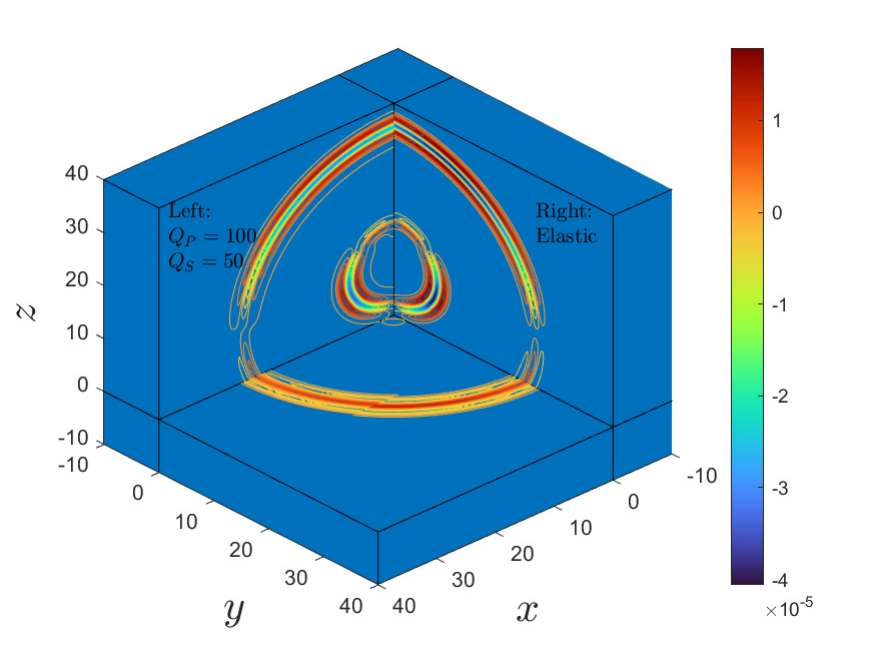}}}
\\
 \centering
 \subfigure[Comparison of the wavefield $v_3(0, 0, x_3)$ at $t=5$s (left) and $t=10$s (right). \label{PV_wavefied}]{
 {\includegraphics[width=0.49\textwidth,height=0.27\textwidth]{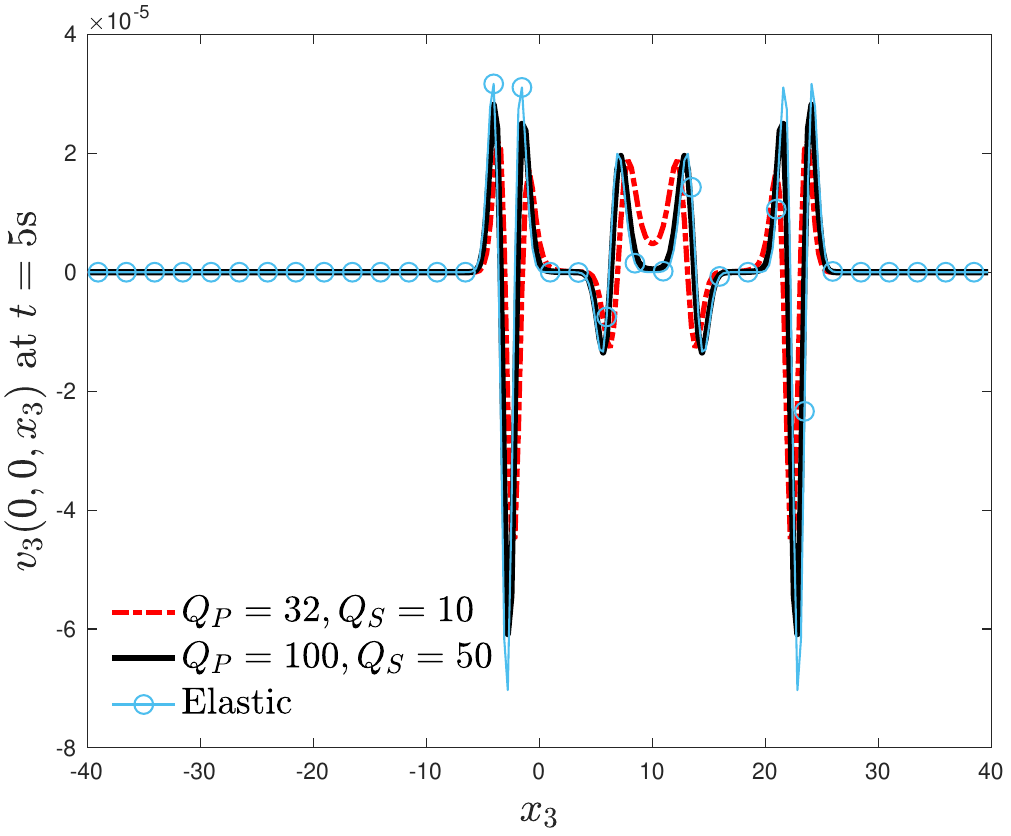}}
 {\includegraphics[width=0.49\textwidth,height=0.27\textwidth]{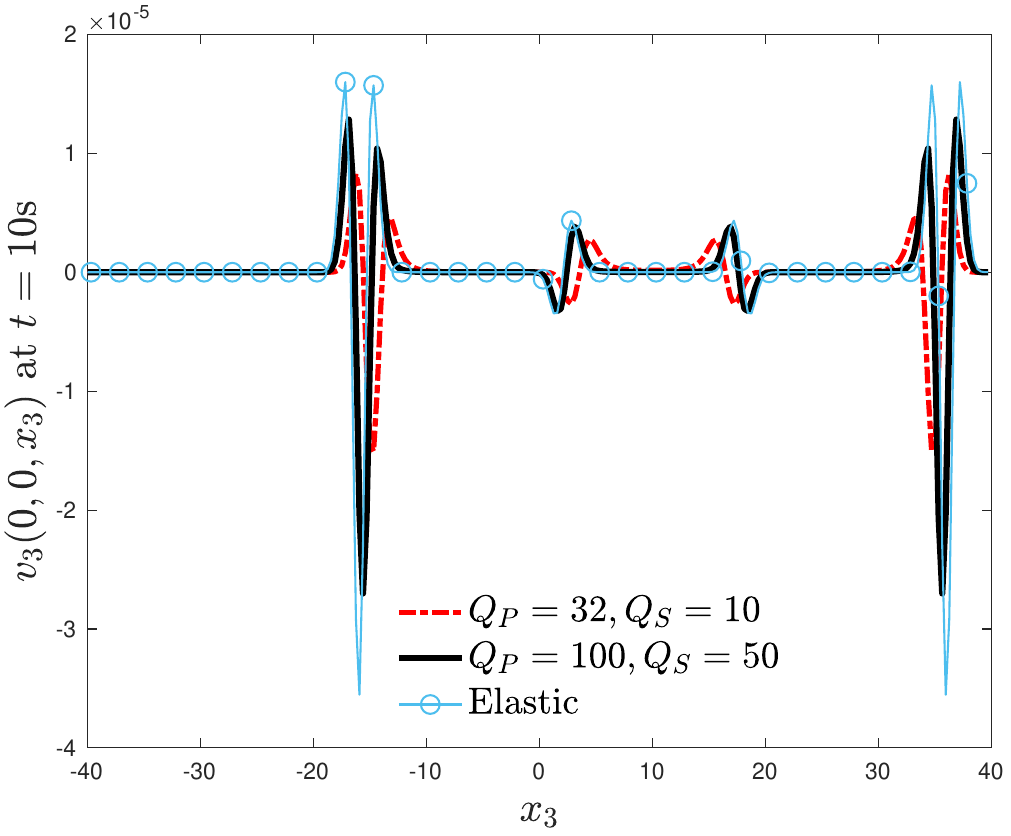}}}
\caption{\small 3-D viscoelastic wave equation in homogenous media:  Comparison of snapshots of the propagation of wavefield $v_3$.  The viscoelasticity not only influences the amplitude but also causes the lag in the first arrival time of the seismic signals. \label{viscoelastic_3d_snapshots}}
\end{figure}

\begin{figure}[!ht]
     \centering
    \subfigure[The relative errors in the wavefield $v_3(0,0,z)$ at $t = 5$s (left) and $t=10$s (right).]{{\includegraphics[width=0.49\textwidth,height=0.27\textwidth]{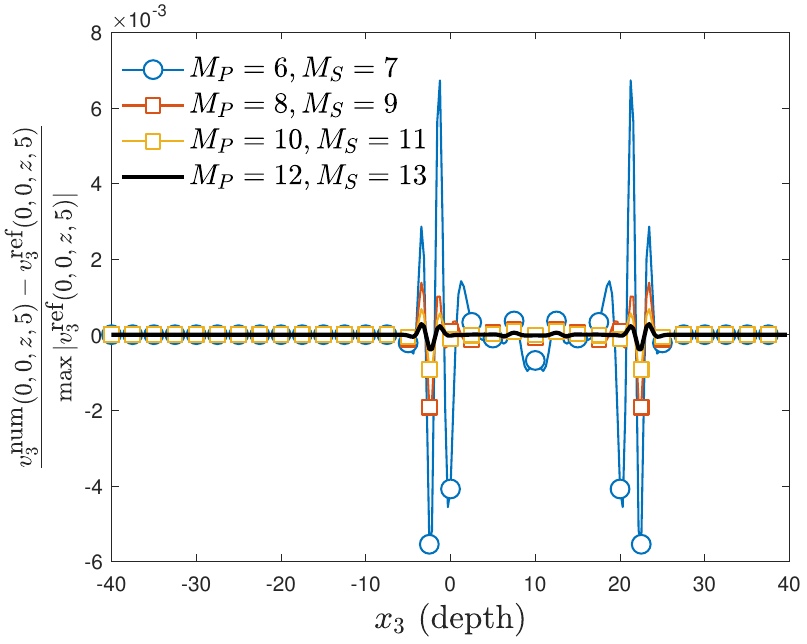}}
    {\includegraphics[width=0.49\textwidth,height=0.27\textwidth]{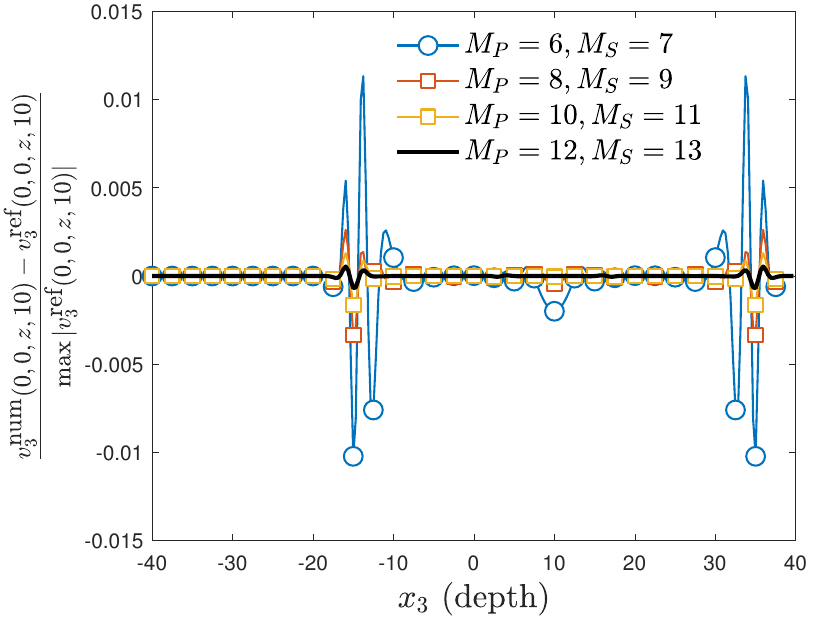}}}
    \\
    \centering
    \subfigure[Convergence with respect to the memory length $M_P + 6M_S$ at $t = 5$s (left) and $t = 10$s (right).]{
    {\includegraphics[width=0.49\textwidth,height=0.27\textwidth]{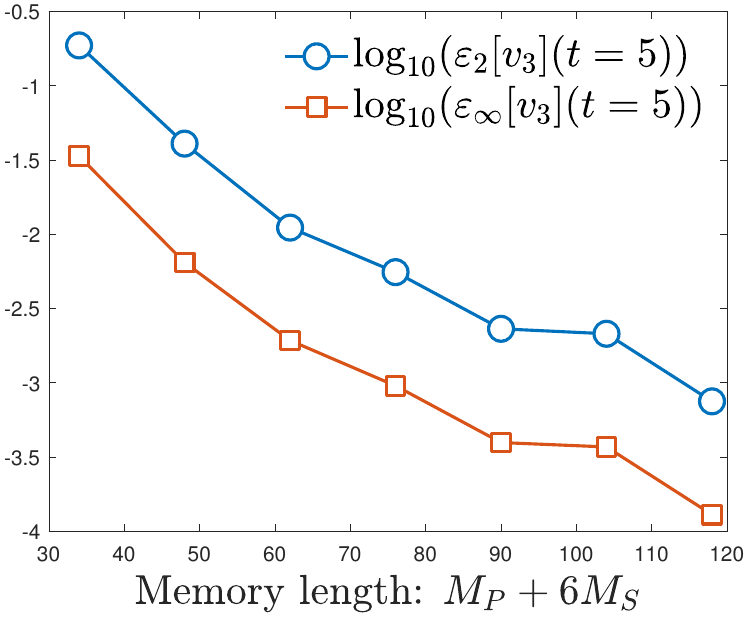}}
    {\includegraphics[width=0.49\textwidth,height=0.27\textwidth]{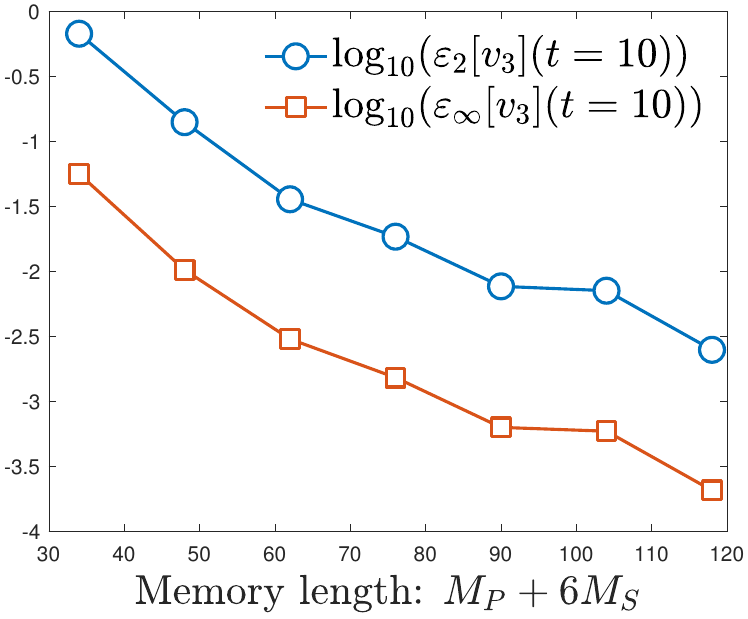}}}
    \caption{\small 3-D viscoelastic wave equation in homogenous media: The rapid convergence of nearly-optimal SOE approximation is verified. Our scheme can achieve relative error about $10^{-3}$ with $M_P = 8, M_S = 9$, say, by storing $M_P+6M_S = 62$ extra memory variables.  \label{homo_convergence}}
\end{figure} 

\begin{table}[!ht]
  \centering
  \caption{\small Homogenous case: A comparison of memory requirement and computational time up to $T=10$s (with $\Delta t = 0.005$s, 2000 steps and spatial resolution $N^3 = 256^3$) under various memory lengths $M_p$ and $M_s$. The relative maximum and $L^2$-errors at $t = 10$s exhibits the rapid convergence, which verifies the accuracy of the nearly-optimal SOE approximation. The reference solution is produced with $M_p =18$, $M_s = 19$.  \label{cpu_time}}
 \begin{lrbox}{\tablebox}
  \begin{tabular}{c|cc|cc|c|c|c}
\hline\hline
Type			&$M_P$ 	&	$M_S$	&	Wavefields	& Memory & Time(h)	& $\varepsilon_{2}[v_3](t=10)$	&	$\varepsilon_{\infty}[v_{3}](t=10)$\\
\hline
Elastic					&	0	&	0	&	9	&	2.0GB	& 	4.32	&	-		&	-		\\
\hline	
\multirow{8}{*}{$Q_P=32$, $Q_S=10$}	&	4	&	5	&	43	&	6.3GB	&	4.77 	&	6.741$\times10^{-1}$	&	5.626$\times10^{-2}$	\\
	&	6	&	7	&	57	&	7.8GB	&	4.92 	&	1.406$\times10^{-1}$	&	1.025$\times10^{-2}$	\\
	&	8	&	9	&	71	&	9.5GB	&	5.24 	&	3.587$\times10^{-2}$	&	3.011$\times10^{-3}$	\\
	&	10	&	11	&	85	&	11.3GB	&	5.42 	&	1.849$\times10^{-2}$	&	1.522$\times10^{-3}$	\\
	&	12	&	13	&	99	&	13.0GB	&	5.87 	&	7.669$\times10^{-3}$	&	6.299$\times10^{-4}$	\\
	&	14	&	15	&	113	&	14.8GB	&	6.00 	&	7.098$\times10^{-3}$	&	5.886$\times10^{-4}$	\\
	&	16	&	17	&	127	&	16.5GB	&	6.21 	&	2.489$\times10^{-3}$	&	2.061$\times10^{-4}$	\\
	&	18	&	19	&	141	&	18.3GB	&	6.43 	&	Reference	&	Reference	\\
\hline\hline
 \end{tabular}
\end{lrbox}
\scalebox{0.92}{\usebox{\tablebox}}
\end{table} 

\Cref{viscoelastic_3d_snapshots}  compares the propagation of the vibrational wavefield $v_3$ in both the elastic media and the viscoelastic media under different $Q_P$ and $Q_S$. The sections of $v_3$ at $y = 0$ and two sectional drawings are provided to present the propagation patterns in the homogenous media. Clearly, the reduction in amplitude caused by the fractional stress-strain relation becomes more evident in strongly attenuated media ($Q_P = 32$, $Q_S = 10$).  \Cref{PV_wavefied} shows that the viscoelasticity not only affects the amplitude but also alters its phase information. Specifically,  the first arrival time of the seismic signals might be delayed by strong viscoelasticity \cite{ZhuCarcioneHarris2013}, a finding that aligns with theoretical predictions \cite{HanyaSerdynska2003}. 

\Cref{homo_convergence} demonstrates the rapid convergence with respect to the memory length $M_P + 6M_S$. Once again, the new SOE approximation can achieve relative maximum error of $10^{-3}$ when $M_P = 8$, $M_S = 9$, thus only requiring $M_P+6M_S = 62$ extra memory wavefields to be stored.

Table \ref{cpu_time} records the total memory requirement and computational time for different $M_P$ and $M_S$. The short-memory operator splitting scheme takes slightly more time on solving the auxiliary dynamics \eqref{dynamics_phi_trace} and \eqref{dynamics_phi} for the memory wavefields. Nonetheless, the increase in the computational time is relatively modest compared to the Strang splitting scheme for the elastic problems, as the primary complexity remains in the calculation of the first-order spatial derivatives of velocity and stress components. Moreover, when $M_P = 8$ and $M_S = 9$, it requires less than five times the memory of the elastic case, thereby significatly reducing
the need for long-term storage.

\subsection{Viscoelastic wave propagation in 3D layered media}
\label{sec.3d.layeredmedia}

In a more challenging example, we study wave propagations in a double-layer structure.
Here, $\gamma_P$ is a function of spatial variable and different SOE approximations must be
used for each layer (see the discussions in \cref{sec.differential_form} for details). We 
compare the wavefields in both elastic and viscoelastic media, using the same source
impulse from \cref{sec.3d.homo}. The model parameters of the heterogenous media, depicted
in \cref{double_layer}, are as follows: For the upper layer spanning $50^2 \times [0, 50]$,
the group velocities are $c_P = 2.614$km/s and $c_S = 0.802$km/s, quality factors are $Q_P=32$, $Q_S=10$
and the mass density is $\rho = 2.2 \textup{g}/\textup{cm}^3$. For the lower layer spanning
$50^2 \times [-50, 0]$, the group velocities are $c_P = 3.2$km/s and $c_S = 1.85$km/s, with quality factors $Q_P=100$, $Q_S=50$ and
a mass density of $\rho = 2.5\textup{g}/\textup{cm}^3$. The domain is $[-50, 50]^3$ ($100\textup{km}\times 100 \textup{km}\times 100 \textup{km}$), and the Fourier spectral method is adopted with $N^3 = 256^3$.  
\begin{figure}[!ht]
\centering
\subfigure[A double-layer elastic structure. \label{elastic_double_layer}]{\includegraphics[width=0.49\textwidth,height=0.27\textwidth]{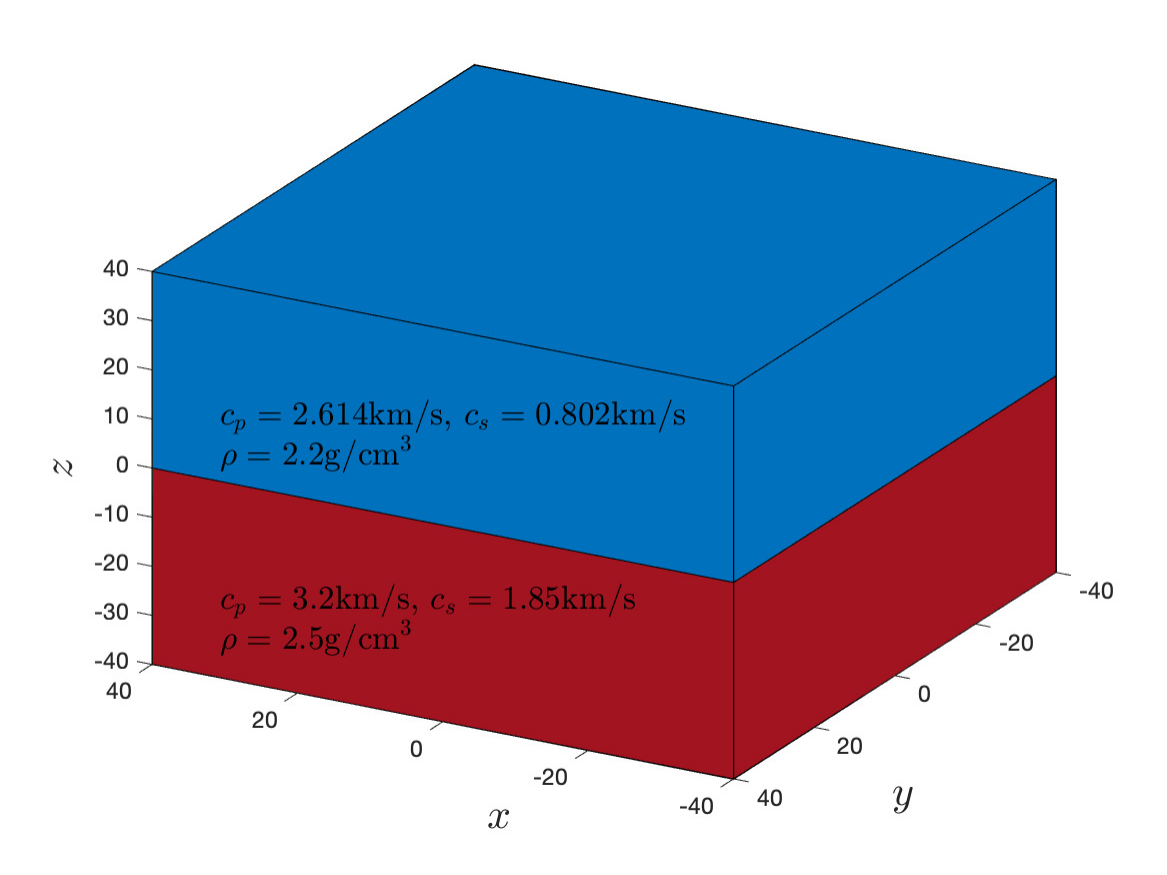}}
\subfigure[A double-layer viscoelastic structure. \label{visco_double_layer}]
{\includegraphics[width=0.49\textwidth,height=0.27\textwidth]{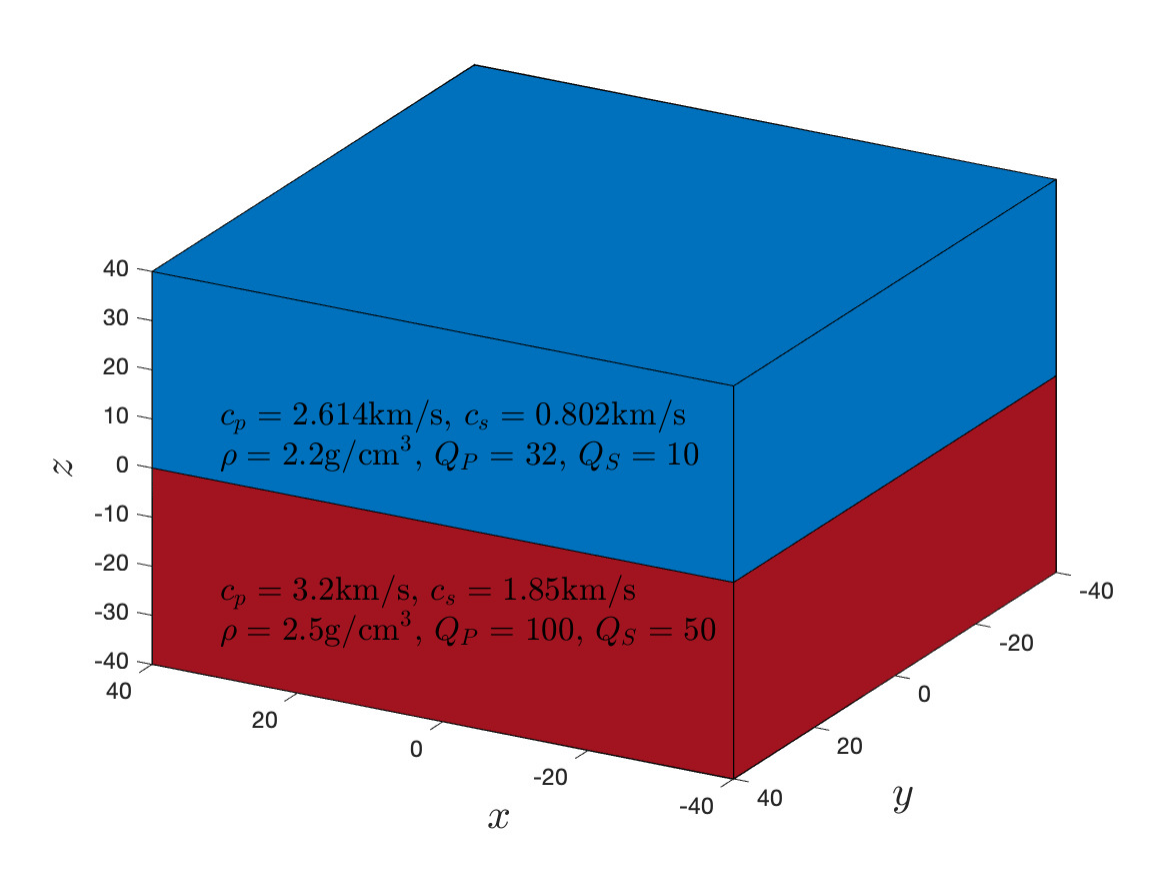}}
\caption{An illustration of a double-layer elastic structure and a viscoelastic structure. \label{double_layer}}
\end{figure}

\Cref{viscoelastic_3d_hetero_snapshots} presents the propagation of vibrational wavefield $v_3$ in both the elastic and viscoelastic double-layer media. A reflected field from the interface at $x_3 = 0$ is evident, with notable
attenuation above, attibuted to the low-quality factor of the upper medium. The lower medium
experiences much less attenuation; thus, the refracted field exhibits a high amplitude. From \cref{hetero_PV_field}, it is clear that  the material anelasticity not only delays the passage of
initial wavefront but also results in
significant phase dispersion and amplitude loss for the refracted waves,
potentially impacting the quality of inversion. 

As the qualify factors in different layers vary, it is necessary to store $M_P^{(u)} + 6 M_S^{(u)}$ memory variables for the upper layer and an additional $M_P^{(l)} + 6 M_S^{(l)}$ for the lower layer,
given the same error tolerance $\varepsilon$. The reference solutions are generated with $M_P^{(u)} = 14$, $M_S^{(u)} = 15$, $M_P^{(u)} = 12$, $M_S^{(u)} = 13$, $\varepsilon=10^{-8}$.
\Cref{hetero_convergence} illustrates a rapid convergence with respect to the memory length $M_P + 6M_S$, where $M_P = \max(M_P^{(u)}, M_P^{(l)})$ and $M_S = \max(M_S^{(u)}, M_S^{(l)})$.
This further confirms the uniform accuracy of the new SOE approximation. From
\cref{hetero_accuracy}, the relative error is about $3\%$ even with much fewer memory variables,
such as $M_P^{(u)} = 6$, $M_S^{(u)} = 7$, $M_P^{(u)} = 4$, $M_S^{(u)} = 5$.
\Cref{hetero_cpu_time} lists the total memory requirement and computational time for various memory variables. Clearly, the storage requirement and arithmetic operations have been reduced.

\begin{figure}[!ht]
\centering
\subfigure[$v_3(x_1, 0, x_3)$ at $4 \to 6 \to 10$s in a double-layer medium. Here $M_P^{(u)}=4$, $M_S^{(u)}=5$, $M_P^{(l)}=2$, $M_S^{(l)}=3$. ]{
{\includegraphics[width=0.32\textwidth,height=0.21\textwidth]{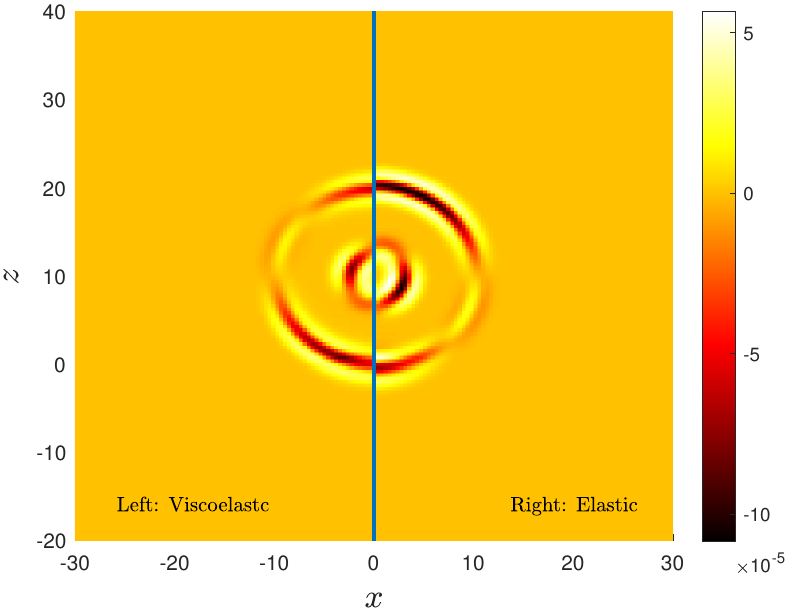}}
{\includegraphics[width=0.32\textwidth,height=0.21\textwidth]{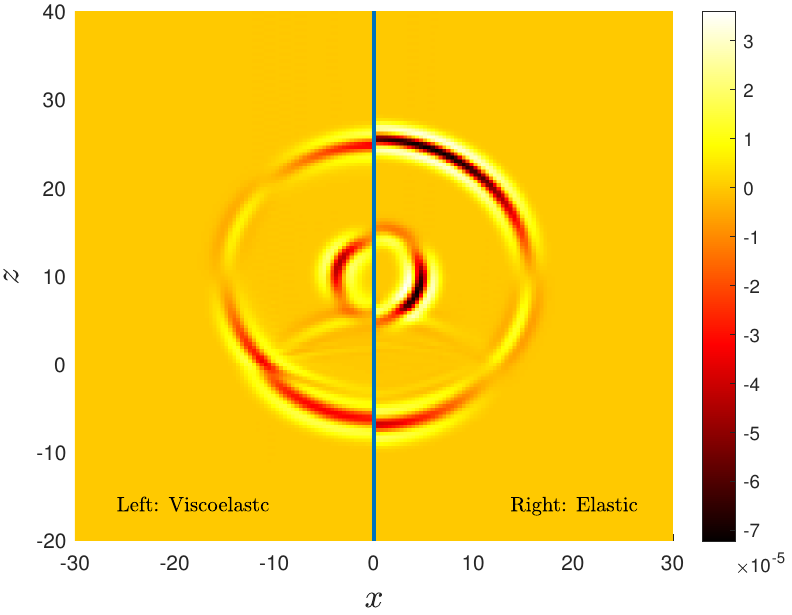}}
{\includegraphics[width=0.32\textwidth,height=0.21\textwidth]{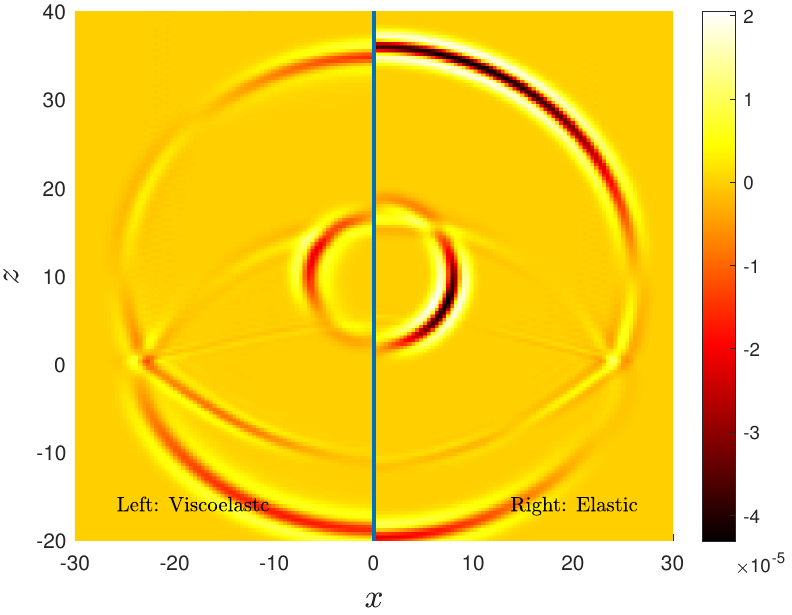}}}
\\
\centering
\subfigure[$v_3(x_1, 0, x_3)$ at $4 \to 6 \to 10$s in a double-layer medium. Here $M_P^{(u)}=14$, $M_S^{(u)}=15$, $M_P^{(l)}=12$, $M_S^{l}=13$. ]{
{\includegraphics[width=0.32\textwidth,height=0.21\textwidth]{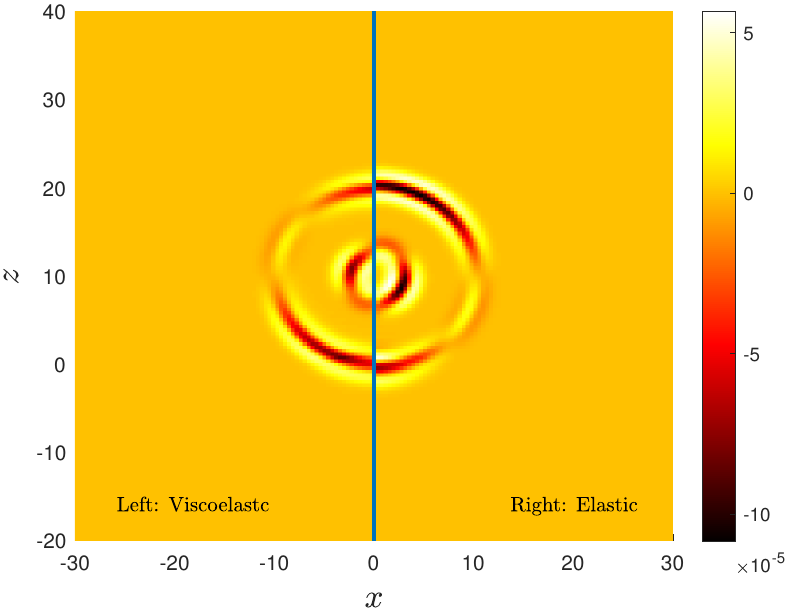}}
{\includegraphics[width=0.32\textwidth,height=0.21\textwidth]{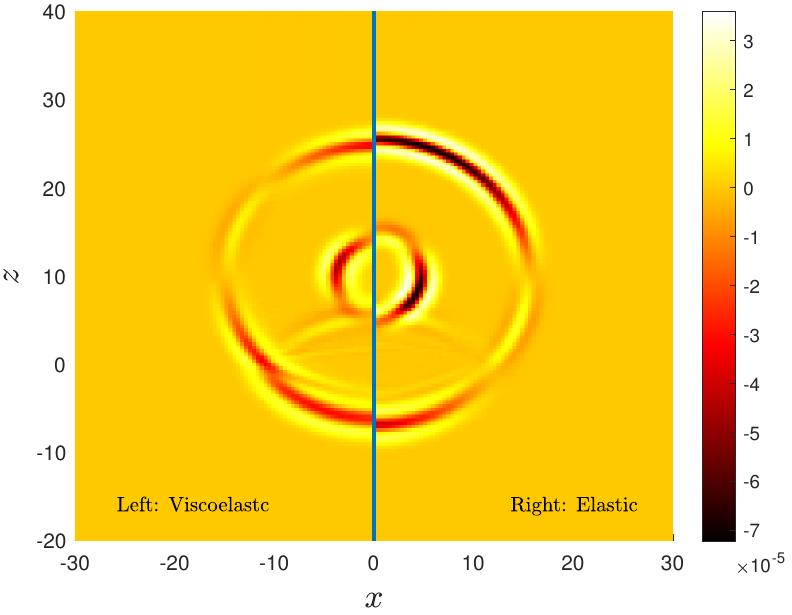}}
{\includegraphics[width=0.32\textwidth,height=0.21\textwidth]{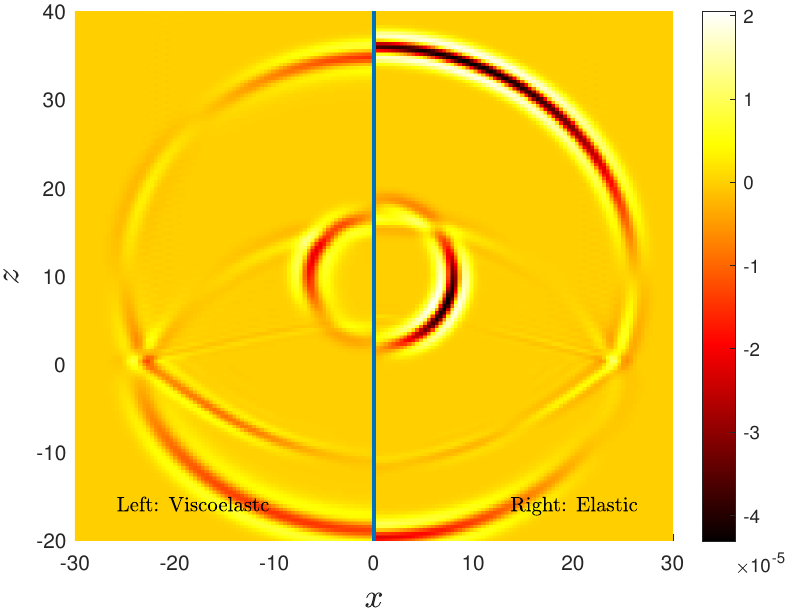}}}
\\
\centering
\subfigure[$v_3$ at 10s (left: $M_P^{(u)}=4$, $M_S^{(u)}=5$, $M_P^{(l)}=2$, $M_S^{(l)}=3$, right: $M_P^{(l)}=14$, $M_S^{(l)}=15$, $M_P^{(u)}=12$, $M_S^{(u)}=13$).]{
{\includegraphics[width=0.49\textwidth,height=0.27\textwidth]{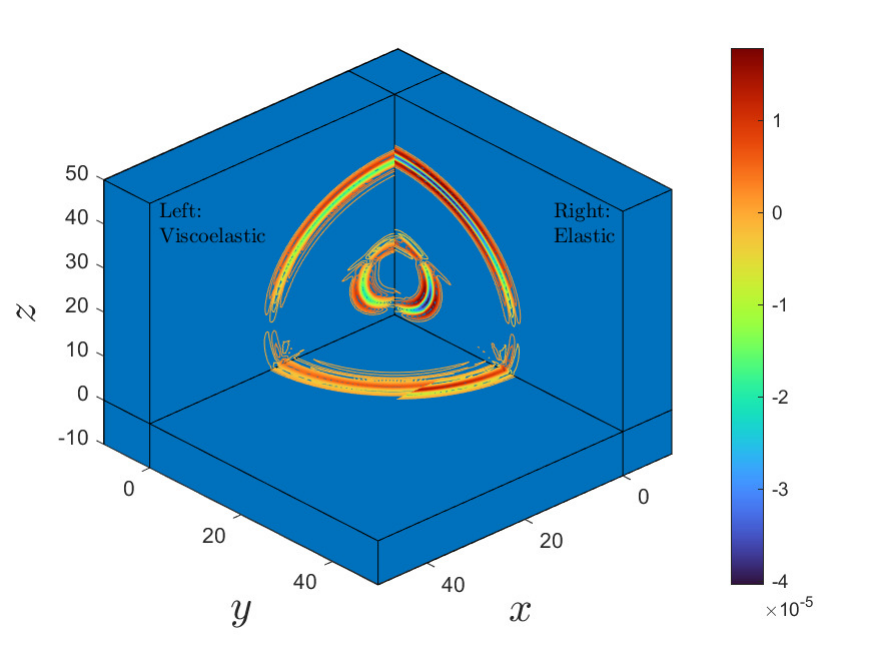}}
{\includegraphics[width=0.49\textwidth,height=0.27\textwidth]{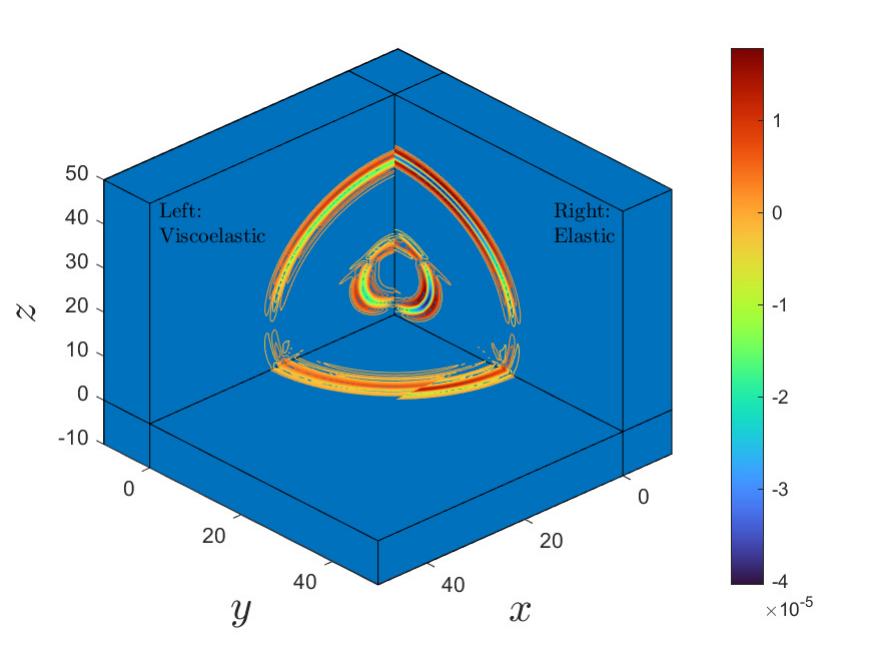}}}
\\
\centering
\subfigure[Comparison of the wavefield $v_3(0, 0, x_3)$ at $t=5$s (left) and $t=10$s (right).\label{hetero_PV_field}]{
 {\includegraphics[width=0.49\textwidth,height=0.27\textwidth]{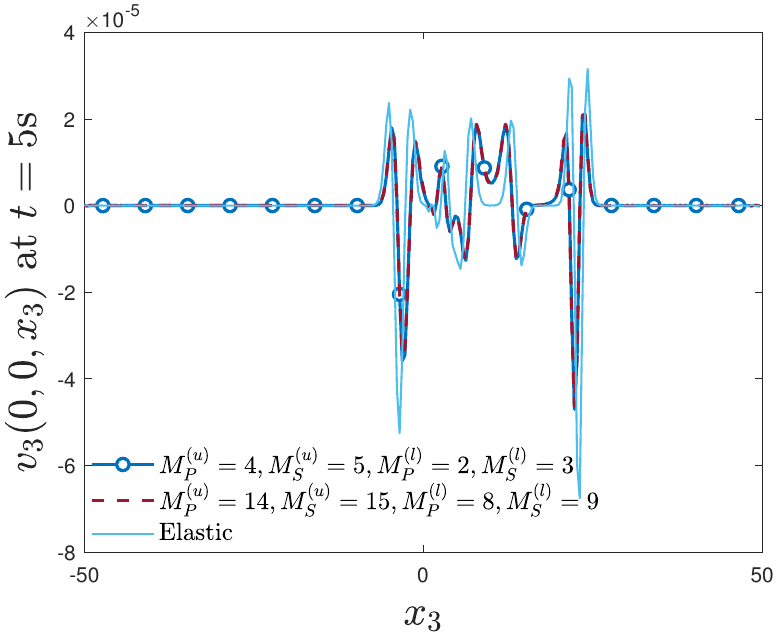}}
 {\includegraphics[width=0.49\textwidth,height=0.27\textwidth]{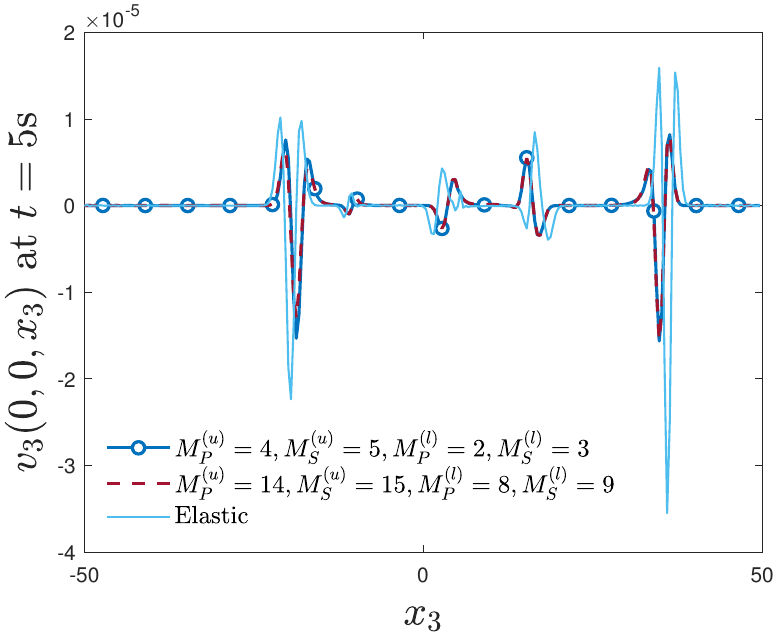}}}
\caption{\small 3-D viscoelastic wave equation in double-layer media:  Snapshots of the propagation of wavefield $v_3$.  The viscoelasticity not only influences the amplitude but also causes the lag in the first arrival time of the reflected fields. \label{viscoelastic_3d_hetero_snapshots}}
\end{figure}

\begin{figure}[!ht]
     \centering
    \subfigure[The relative errors in the wavefield $v_3(0,0,z)$ at $t = 5$s (left) and $t=10$s (right). \label{hetero_accuracy}]{
    {\includegraphics[width=0.49\textwidth,height=0.27\textwidth]{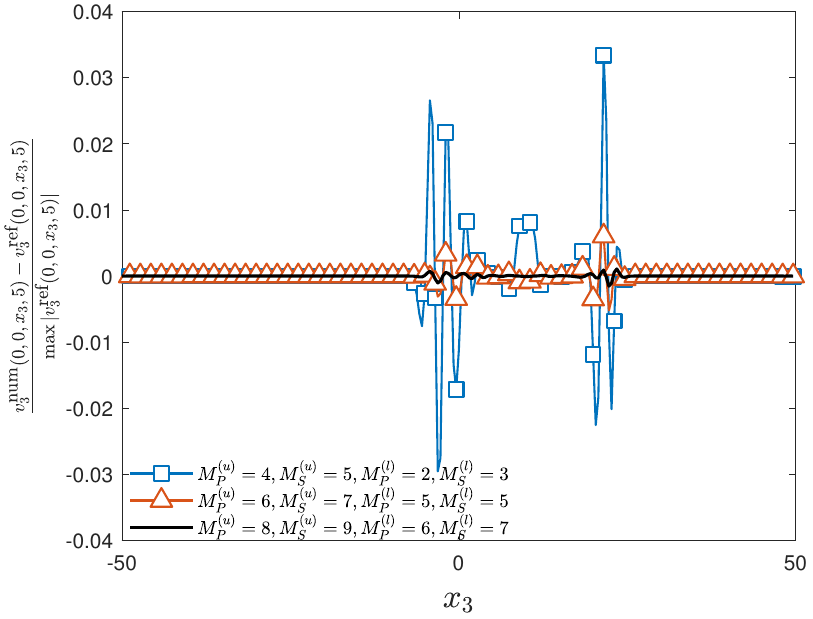}}
    {\includegraphics[width=0.49\textwidth,height=0.27\textwidth]{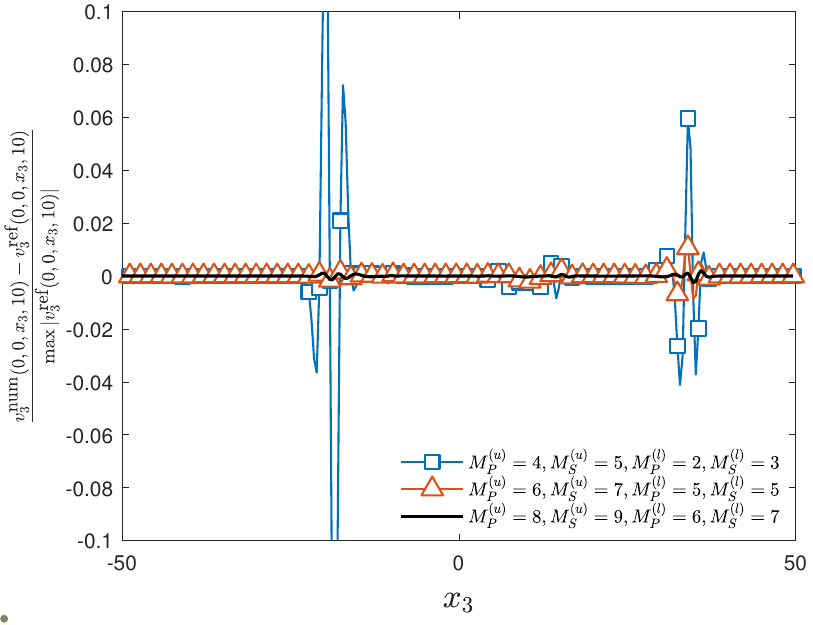}}}
    \\
    \centering
    \subfigure[Convergence with respect to the memory length $M_P + 6M_S$ at $t = 5$s (left) and $t = 10$s (right).\label{hetero_convergence}]{
    {\includegraphics[width=0.49\textwidth,height=0.27\textwidth]{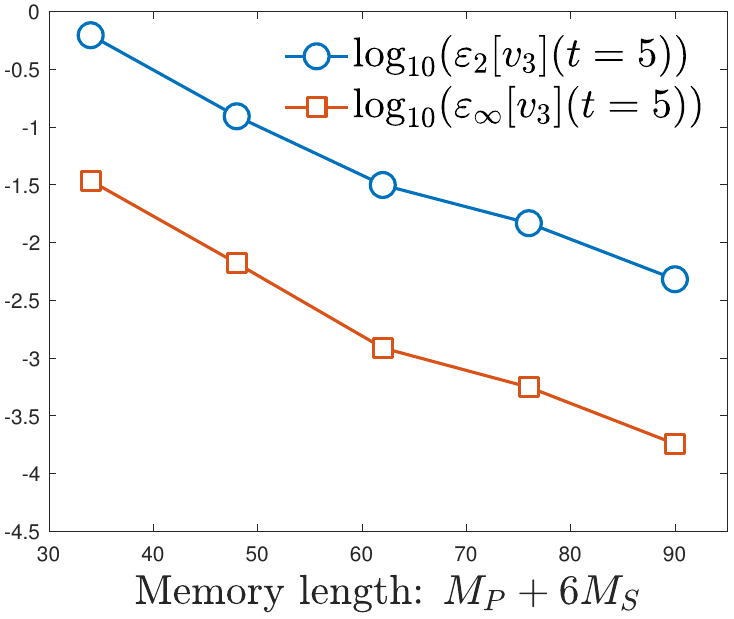}}
    {\includegraphics[width=0.49\textwidth,height=0.27\textwidth]{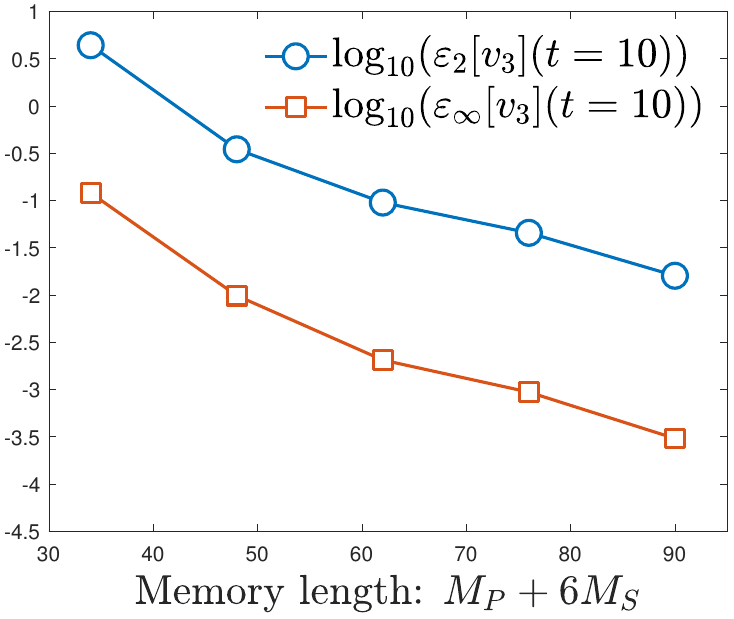}}}
    \caption{\small 3-D viscoelastic wave equation in double-layer media: The rapid convergence of nearly-optimal SOE approximation is verified. Our scheme can achieve relative error about $10^{-3}$ with $M_P^{(u)}+6M_S^{(u)} = 62$ extra memory variables for the upper layer and $M_P^{(l)}+6M_S^{(l)} = 48$ ones for the lower layer.  \label{hetero_PS_SOE_error_compare}}
\end{figure} 

\begin{table}[!ht]
  \centering
  \caption{\small heterogenous case: A comparison of memory requirement and computational time up to $T=10$s (with $\Delta t = 0.005$s, 2000 steps and spatial resolution $N^3 = 256^3$) under different memory lengths $M_p^{(u)}$, $M_p^{(l)}$, $M_s^{(u)}$ and $M_s^{(l)}$. The relative maximum and $L^2$-errors at $t = 10$s exhibit rapid convergence, which verifies the accuracy of the nearly optimal SOE approximation. The reference solution is produced with $M_P^{(u)} = 14$, $M_S^{(u)} = 15$, $M_P^{(u)} = 12$, $M_S^{(u)} = 13$. \label{hetero_cpu_time}}
 \begin{lrbox}{\tablebox}
  \begin{tabular}{cc|cc|cc|c|c|c}
\hline\hline
$M_P^{(u)}$ 	&	$M_S^{(u)}$	&$M_P^{(l)}$ 	&	$M_S^{(l)}$	& Wavefields	&	Memory & Time(h)	& $\varepsilon_{2}[v_3](t=10)$	&	$\varepsilon_{\infty}[v_{3}](t=10)$\\
\hline
4	&	5	&	2	&	3	&	43	&	6.3GB	&	7.92 	&	4.391	&	1.200$\times10^{-1}$	\\
6	&	7	&	5	&	5	&	57	&	8.0GB	&	8.94 	&	3.480$\times10^{-1}$	&	9.853$\times10^{-3}$	\\
8	&	9	&	6	&	7	&	71	&	9.6GB	&	10.76 	&	9.503$\times10^{-2}$	&	2.048$\times10^{-3}$	\\
10	&	11	&	8	&	9	&	85	&	11.3GB	&	11.86 	&	4.545$\times10^{-2}$	&	9.381$\times10^{-4}$	\\
12	&	13	&	10	&	11	&	99	&	12.9GB	&	13.42 	&	1.598$\times10^{-2}$	&	3.033$\times10^{-4}$	\\
14	&	15	&	12	&	13	&	113	&	14.6GB	&	14.93 	&	Reference				&	Reference	\\
\hline\hline
 \end{tabular}
\end{lrbox}
\scalebox{0.95}{\usebox{\tablebox}}
\end{table} 

\section{Conclusions and further discussions}
\label{sec.conclusion}

The time-fractional wave equation accounts for both wave attenuation and velocity dispersion in real viscoelastic media. However, existing numerical methods usually require storing the entire
history of wavefields due to the very small fractional exponent, which hinders its practical usage.
This paper proposes a new SOE approximation to resolve this long-memory limitation. The new SOE approximation seeks an optimal curve fitting for the power function $t^{-2\gamma}$ with the generalized Gaussian quadrature applied to a new integral representation, and requires
only half the nodes compared to existing SOE approximations. The equivalence between the SOE approximation of the continuous fractional stress-strain relation and the rheological model is established. An improved bound of the new SOE approximation provides a quantitative characterization of their differences. Numerical simulations on 3D constant-Q viscoacoustic equation and P- and S- viscoelastic wave equations have been conducted to demonstrate that the proposed method can effectively
capture changes in both amplitude and phase induced by material anelasticity. Furthermore,
the memory wavefields can be significantly compressed.

In real seismic applications, complex geological structures might exhibit strong anisotropy in each direction. Using a domain decomposition strategy, one only needs to fit the fractional stress-strain relation within each subdomain. The application of the time-fractional wave equation in seismic inversion may be a topic of our future work. 

\section*{Acknowledgement}
This research was supported by the Project funded by the Natural Science Foundation of Shandong Province for Excellent Youth Scholars (Nos.~ZR2020YQ02), the Taishan Scholars Program of Shandong Province of China (Nos.~tsqn201909044), the National Natural Science Foundation of China (No.~1210010642) and the Fundamental Research Funds for the Central Universities (Nos.~ 310421125).

\newpage 
\appendix
 \setcounter{equation}{0} 
 \renewcommand{\theequation}{A.\arabic{equation}}
\section{\label{Green}Green's function for 3-D viscoacoustic wave equation}
The Green's function $G^{(3)}(\bx, t; 1-\gamma_P)$ for 3-D viscoacoustic wave equation is given by 
\begin{equation}
G^{(3)}(\bx, t; 1-\gamma_P) = -\frac{1}{4\pi C_P |\bx| t^{2-2\gamma_P}} \frac{\partial  M_{1 - \gamma_P}(z)}{\partial z}\Big |_{z = \frac{|\bx|}{\sqrt{C_P} t^{1-\gamma_P}}}.
\end{equation}
Here the Mainardi function $M_{l}(z)$ is a special Wright function of the second kind \cite{bk:Mainardi2010}, say, $M_{l}(z) \coloneqq W_{-l, 1 - l}(-z)$, $0 < l < 1$, $z\in \mathbb{C}$. Moreover, for $-1 < \lambda < 0$, $l < 1$ and $x > 0$ \cite{Luchko2008}, it has 
\begin{equation*}
W_{\lambda, l}(- x) = \frac{1}{\pi} \int_0^{+\infty} K_{\lambda, l}(-x, r ) \D r, ~~ K_{\lambda, l}(x, r ) = r^{-l} e^{-r} \left[e^{x r^{-\lambda} \cos(\lambda \pi)} \sin(x r^{-\lambda} \sin(\pi \lambda) + \pi l)\right].
\end{equation*}

Noting that 
\begin{equation}
\begin{split}
\frac{\partial K_{\lambda, l}(x, r)}{\partial x}  =  &r^{-(l+\lambda)} e^{-r} \left[e^{x r^{-\lambda} \cos(\lambda \pi)} \sin(x r^{-\lambda} \sin(\pi \lambda) + \pi l) \cos(\lambda\pi) \right] \\
&+ r^{-(l+\lambda)} e^{-r} \left[e^{x r^{-\lambda} \cos(\lambda \pi)} \cos(x r^{-\lambda} \sin(\pi \lambda) + \pi l) \sin(\lambda\pi) \right] \\
 = & r^{-(l+\lambda)} e^{-r} \left[e^{x r^{-\lambda} \cos(\lambda \pi)} \sin(x r^{-\lambda} \sin(\pi \lambda) + \pi (l+\lambda)) \right] = K_{\lambda, \lambda+l}(x, r),
\end{split}
\end{equation}
 the derivative of the Mainardi function can be represented as
\begin{equation}
\frac{\partial M_l(x)}{\partial x} = - \frac{1}{\pi} \int_{0}^{+\infty} K_{-l, 1-2l}(x, r)\D r, \quad \frac{\partial M_l(0)}{\partial x} = - \frac{1}{\Gamma(1 - 2l)}.
\end{equation}
For sufficiently large $x$, the Mainardi function has a saddle-point approximation \cite{bk:Mainardi2010}, 
 \begin{equation}\label{saddle_point}
\begin{split}
&\frac{\partial M_l(\frac{z}{l})}{\partial z} \sim \frac{1}{l \sqrt{2\pi(1-l)}} \left(\frac{l-1/2}{1-l}z^{\frac{2l-3/2}{1-l}} - \frac{1}{l}z^{\frac{2l-1/2}{1-l}}\right)\exp\left(-\frac{1-l}{l} z^{\frac{1}{1-l}}\right), \quad |z| \to \infty.
\end{split}
\end{equation}

The derivatives of Mainardi's functions can be calculated via the Gauss-Laguerre quadrature and the saddle point approximation, and details can be found in the Appendix of \cite{XiongGuo2022}.

\clearpage
\section{Tables of nodes and weights in new SOE for typical quality factors}

\begin{table}[!ht]
  \centering
  \caption{\small The nodes and weights in new SOE for $Q = 10$, $2\gamma =2\pi^{-1}\arctan(Q^{-1})  \approx  0.0635$ to attain $|t^{-\beta} - \sum_{j=1}^{\nexp} w_j e^{-s_j t}| \le \varepsilon t^{-\beta}$ for $\delta \le t \le T$, with the final instant $T \le 10$ and temporal gap $\delta \ge 5\times 10^{-3}$. \label{SOE_table_Q10}}
 \begin{lrbox}{\tablebox}
  \begin{tabular}{cc|cc}
  \hline\hline
$\omega_j$ 	&	$s_j$	&	$\omega_j$ 	&	$s_j$	\\
\hline
 \multicolumn{2}{c|}{$\varepsilon = 10^{-2}$, $\nexp = 5$} 	&  \multicolumn{2}{c}{$\varepsilon =  10^{-3}$, $\nexp = 7$} \\
\hline
      0.1440797465700240D-10  &0.8415261929142592D+00	&	      0.1440797155149404D-10 & 0.8240830013371835D+00	\\
      0.1954901542163180D+00 & 0.1636845055652096D+00	&	      0.1322237166027692D+00 & 0.1501811717860068D+00	\\
      0.1723999921378503D+01 & 0.1323100227004369D+00	&	      0.8517903732042016D+00  &0.9648611144577410D-01	\\
      0.1229908868952335D+02 & 0.1547537593365945D+00	&	      0.3441114686073438D+01  &0.9644807153788990D-01	\\
      0.9777983682050727D+02 & 0.1869902132381615D+00	&	      0.1350088871384277D+02  &0.1069689948347042D+00	\\
&	&	      0.5492464653183761D+02  &0.1205169710604457D+00	\\
&	&	      0.2338504327514402D+03  &0.1374498357868730D+00	\\
\hline
 \multicolumn{2}{c|}{$\varepsilon = 10^{-4}$, $\nexp = 9$} 	&  \multicolumn{2}{c}{$\varepsilon =  10^{-5}$, $\nexp = 11$} \\
\hline
      0.1440797164254682D-10  &0.8112378597026695D+00	&	      0.1440797619283847D-10  &0.8010997231278196D+00	\\
      0.1009890841402805D+00  &0.1443806338391152D+00	&	      0.8197828436191787D-01  &0.1409941501440947D+00	\\
      0.5823334989875405D+00  &0.8332378558784953D-01	&	      0.4498296707524729D+00  &0.7719857356447507D-01	\\
      0.1869388736178351D+01  &0.7366620152204306D-01	&	      0.1296039075552287D+01  &0.6269524338587557D-01	\\
      0.5367910686013082D+01  &0.7625660830946936D-01	&	      0.3160106693096034D+01  &0.6078138363618767D-01	\\
      0.1531655543336788D+02  &0.8230551789077256D-01	&	      0.7409771919428429D+01  &0.6321002136990252D-01	\\
      0.4439653407641731D+02  &0.8962282114855959D-01	&	      0.1734398686592705D+02  &0.6712827947872171D-01	\\
      0.1312617293848889D+03  &0.9788699963267940D-01	&	      0.4092281319055039D+02  &0.7169125085778819D-01	\\
      0.4001238289113953D+03  &0.1103304595213679D+00	&	      0.9756294726492744D+02  &0.7671133135458129D-01	\\
	&	      &0.2354800125540291D+03  &0.8253927154623260D-01	\\
	&	      &0.5852581776377881D+03  &0.9324471363963684D-01	\\
\hline	
	 \multicolumn{2}{c|}{$\varepsilon = 10^{-6}$, $\nexp = 13$} 	&  \multicolumn{2}{c}{$\varepsilon =  10^{-7}$, $\nexp = 15$} \\
\hline
      0.1440797344443702D-10  &0.7927416768056691D+00	&	      0.1440797301373925D-10  &0.7856419964970299D+00	\\
      0.6909132948947617D-01  &0.1386674750682972D+00	&	      0.5974823699768806D-01  &0.1369104454153603D+00	\\
      0.3693004676261487D+00  &0.7379286075670194D-01	&	      0.3144306059252867D+00  &0.7163361769803406D-01	\\
      0.1005372969224108D+01  &0.5678739803068599D-01	&	      0.8284526998828101D+00  &0.5329152980471319D-01	\\
      0.2235671903599100D+01  &0.5201787459951138D-01	&	      0.1744202695136431D+01  &0.4672404154667911D-01	\\
      0.4653227235497959D+01  &0.5208079978662095D-01	&	      0.3366361346154247D+01  &0.4508321149250512D-01	\\
      0.9520871769286824D+01  &0.5406183576472692D-01	&	      0.6296947578816033D+01  &0.4569204223027506D-01	\\
      0.1946905282824594D+02  &0.5680493830595190D-01	&	      0.1167667863594848D+02  &0.4728179409751884D-01	\\
      0.3999161843226904D+02  &0.5991962774574524D-01	&	      0.2164708182376960D+02  &0.4930960413677760D-01	\\
      0.8264328348675848D+02  &0.6329721168957117D-01	&	      0.4024043478382247D+02  &0.5157105998736930D-01	\\
      0.1719522811815946D+03  &0.6699623446010711D-01	&	      0.7508563548322061D+02  &0.5399937693294307D-01	\\
      0.3613875662253792D+03  &0.7162043024628401D-01	&	      0.1406994214974785D+03  &0.5660145559683093D-01	\\
      0.7832726493548571D+03  &0.8143769999391330D-01	&	      0.2650393315555953D+03  &0.5952492491283821D-01	\\
	&	&	      0.5041800183069405D+03  &0.6350067238989261D-01	\\
	&	&	      0.9906709743596185D+03  &0.7274601644296926D-01	\\
\hline\hline
 \end{tabular}
\end{lrbox}
\scalebox{0.8}{\usebox{\tablebox}}
\end{table}

\begin{table}[!ht]
  \centering
  \caption{\small  The nodes and weights in new SOE for $Q = 32$, $2\gamma =2\pi^{-1}\arctan(Q^{-1}) \approx  0.0199$ to attain $|t^{-\beta} - \sum_{j=1}^{\nexp} w_j e^{-s_j t}| \le \varepsilon t^{-\beta}$ for $\delta \le t \le T$, with the final instant $T \le 10$ and temporal gap $\delta \ge 5\times 10^{-3}$.\label{SOE_table_Q32}}
 \begin{lrbox}{\tablebox}
  \begin{tabular}{cc|cc}
  \hline\hline
$s_j$ 	&	$\omega_j$	&	$s_j$ 	&	$\omega_j$	\\
\hline
 \multicolumn{2}{c|}{$\varepsilon = 10^{-2}$, $\nexp = 4$} 	&  \multicolumn{2}{c}{$\varepsilon =  10^{-3}$, $\nexp = 6$} \\
\hline
      0.1233271244807241D-27  &0.9507315622743330D+00	&	      0.1233271244807303D-27  &0.9307117960462683D+00	\\
      0.2446613360651417D+00  &0.5909871860442109D-01	&	      0.7955490392898384D-01  &0.5410148687511617D-01	\\
      0.3188494705145181D+01  &0.5108257353637974D-01	&	      0.6400754573055761D+00  &0.3427590000250326D-01	\\
      0.4374046349798036D+02  &0.6006469461273273D-01	&	      0.3406885956980713D+01  &0.3469155299642254D-01	\\
	&	&	      0.1938932219421636D+02  &0.3838452138989368D-01	\\
	&	&	      0.1244724672270921D+03  &0.4246441195877351D-01	\\
\hline
 \multicolumn{2}{c|}{$\varepsilon = 10^{-4}$, $\nexp = 8$} 	&  \multicolumn{2}{c}{$\varepsilon =  10^{-5}$, $\nexp = 10$} \\
\hline
      0.1233271244806500D-27  &0.9257639935683603D+00	&	      0.1233271244808180D-27  &0.9342219590785330D+00	\\
      0.5960153475701349D-01  &0.5267049620909144D-01	&	      0.8687488238252516D-01  &0.4970195686328412D-01	\\
      0.4181740222757921D+00  &0.2931551656158208D-01	&	      0.5015374242392519D+00  &0.2548030796317139D-01	\\
      0.1574911535622993D+01  &0.2518001231324935D-01	&	      0.1529011432214429D+01  &0.2044968316336439D-01	\\
      0.5391000223314470D+01  &0.2570793408412855D-01	&	      0.4035666163812632D+01  &0.1963888134955552D-01	\\
      0.1891729095316842D+02  &0.2722073088872728D-01	&	      0.1039920617855594D+02  &0.1995668303240832D-01	\\
      0.6935590107043559D+02  &0.2891757255269721D-01	&	      0.2692365365300807D+02  &0.2052617564942962D-01	\\
      0.2670138625582335D+03  &0.3114529866907899D-01	&	      0.7037020812725662D+02 & 0.2113029151314107D-01	\\
	&	&	      0.1857132902212041D+03  &0.2178984669856504D-01	\\
	&	&	      0.5017183552663466D+03  &0.2340202613949829D-01	\\
\hline	
	 \multicolumn{2}{c|}{$\varepsilon = 10^{-6}$, $\nexp = 12$} 	&  \multicolumn{2}{c}{$\varepsilon =  10^{-7}$, $\nexp = 14$} \\
\hline
      0.1233271244834537D-27  &0.9311450359309102D+00	&	      0.1233271246154151D-27  &0.9224243946813435D+00	\\
      0.7309291099348825D-01  &0.4918924204721100D-01	&	      0.4565633910836096D-01  &0.4882502116086868D-01	\\
      0.4081678128820871D+00  &0.2436636893296938D-01	&	      0.2549389400613811D+00  &0.2394565431114788D-01	\\
      0.1156721990103773D+01  &0.1841162414492538D-01	&	      0.7027385624420864D+00  &0.1727215242352003D-01	\\
      0.2722207854479077D+01  &0.1679826283180114D-01	&	      0.1543612188551032D+01  &0.1481758253849639D-01	\\
      0.6089902279269454D+01  &0.1663193579207381D-01	&	      0.3117126857811117D+01  &0.1407730402715572D-01	\\
      0.1351634416771198D+02  &0.1690890326587400D-01	&	      0.6138445183215538D+01  &0.1406238726348420D-01	\\
      0.3010967506196511D+02  &0.1729694788251238D-01	&	      0.1206000042888941D+02  &0.1431217725498386D-01	\\
      0.6743945752842677D+02  &0.1768423697908001D-01	&	      0.2381472369894094D+02  &0.1464745537413458D-01	\\
      0.1516100483822089D+03  &0.1802327708903234D-01	&	      0.4736179967023599D+02  &0.1501032214584595D-01	\\
      0.3414653124039590D+03  &0.1838524006642803D-01	&	      0.9490353496367300D+02  &0.1538740114341876D-01	\\
      0.7811858483832548D+03  &0.1969170630104930D-01	&	      0.1916971393801842D+03  &0.1579940949347425D-01	\\
	&	&	      0.3916535995905957D+03  &0.1639177497211432D-01	\\
	&	&	      0.8266960970664088D+03  &0.1810583027037206D-01	\\
\hline\hline
 \end{tabular}
\end{lrbox}
\scalebox{0.8}{\usebox{\tablebox}}
\end{table} 

\begin{table}[!ht]
  \centering
  \caption{\small  The nodes and weights in new SOE for $Q = 50$, $2\gamma =2\pi^{-1}\arctan(Q^{-1}) \approx  0.0127$ to attain $|t^{-\beta} - \sum_{j=1}^{\nexp} w_j e^{-s_j t}| \le \varepsilon t^{-\beta}$ for $\delta \le t \le T$, with the final instant $T \le 10$ and temporal gap $\delta \ge 5\times 10^{-3}$.\label{SOE_table_Q50}}
 \begin{lrbox}{\tablebox}
  \begin{tabular}{cc|cc}
  \hline\hline
$s_j$ 	&	$\omega_j$	&	$s_j$ 	&	$\omega_j$	\\
\hline
 \multicolumn{2}{c|}{$\varepsilon = 10^{-2}$, $\nexp = 3$} 	&  \multicolumn{2}{c}{$\varepsilon =  10^{-3}$, $\nexp = 5$} \\
\hline
      0.3291459309486449D-43  &0.9715159334883926D+00	&	      0.3291459309486449D-43  &0.9654032847335073D+00	\\
      0.3844977954137703D+00  &0.4494377137687228D-01	&	      0.1805141287964131D+00  &0.3541837895493161D-01	\\
      0.1269217440126866D+02  &0.5029910225330850D-01	&	      0.1603711027271578D+01  &0.2495583781684037D-01	\\
	&	&	      0.1131524795360649D+02  &0.2656277187664437D-01	\\
	&	&	      0.9103455146880404D+02  &0.2933916818950616D-01	\\
\hline
 \multicolumn{2}{c|}{$\varepsilon = 10^{-4}$, $\nexp = 7$} 	&  \multicolumn{2}{c}{$\varepsilon =  10^{-5}$, $\nexp = 9$} \\
\hline
      0.3291459309486449D-43  &0.9613696925562936D+00	&	      0.3291459309486449D-43  &0.9583545824772974D+00	\\
      0.1230055319059663D+00  &0.3335077521224538D-01	&	      0.9421313698694372D-01  &0.3257050484288297D-01	\\
      0.8113385187918014D+00  &0.1901845369491019D-01	&	      0.5594525408055533D+00  &0.1684018645821342D-01	\\
      0.3262465799193611D+01  &0.1760194323773591D-01	&	      0.1799033081583265D+01  &0.1390725893797957D-01	\\
      0.1273359255878506D+02  &0.1827779127280318D-01	&	      0.5146221618958586D+01  &0.1362728285045673D-01	\\
      0.5190919485559557D+02  &0.1930292644345800D-01	&	      0.1464101300598260D+02  &0.1397871517834396D-01	\\
      0.2237616215581631D+03  &0.2060469071916570D-01	&	      0.4245035087292208D+02  &0.1447337311323433D-01	\\
	&	&	      0.1260744262177191D+03  &0.1502467492901445D-01	\\
	&	&	      0.3878389209183667D+03  &0.1606353834964511D-01	\\
\hline	
	 \multicolumn{2}{c|}{$\varepsilon = 10^{-6}$, $\nexp = 11$} 	&  \multicolumn{2}{c}{$\varepsilon =  10^{-7}$, $\nexp = 13$} \\
\hline
      0.3291459309486449D-43  &0.9559470828794648D+00	&	      0.3291459309486449D-43  &0.9539432545841504D+00	\\
      0.7659112881304787D-01  &0.3217361279926604D-01	&	      0.6461018484650244D-01  &0.3193336413074308D-01	\\
      0.4338032622992327D+00  &0.1585545388214821D-01	&	      0.3568618864911536D+00  &0.1533257985633187D-01	\\
      0.1256503454895687D+01  &0.1209976708574026D-01	&	      0.9785073400536500D+00  &0.1113044331387444D-01	\\
      0.3060007372246838D+01  &0.1116945869588176D-01	&	      0.2177669719981179D+01  &0.9746907371580906D-02	\\
      0.7156010727142673D+01  &0.1112346488285604D-01	&	      0.4525655673256203D+01  &0.9387145829502928D-02	\\
      0.1671889839142481D+02  &0.1133261712262682D-01	&	      0.9242537672647016D+01  &0.9399519115875818D-02	\\
      0.3944184916633604D+02  &0.1161367233920682D-01	&	      0.1887578469291693D+02  &0.9535911538885434D-02	\\
      0.9421749253717621D+02  &0.1192214677313131D-01	&	      0.3876276251779908D+02  &0.9713434014349787D-02	\\
      0.2283913597457030D+03  &0.1230024163842561D-01	&	      0.8018021213022408D+02  &0.9907610910937507D-02	\\
      0.5714449581431992D+03  &0.1329624863545951D-01	&	      0.1672104632942592D+03  &0.1012302317471267D-01	\\
	&	&	      0.3527186673439197D+03  &0.1044035591633067D-01	\\
	&	&	      0.7683347210259111D+03  &0.1142908470080125D-01	\\
\hline\hline
 \end{tabular}
\end{lrbox}
\scalebox{0.8}{\usebox{\tablebox}}
\end{table}

\begin{table}[!ht]
  \centering
  \caption{\small  The nodes and weights in new SOE for $Q = 100$, $2\gamma =2\pi^{-1}\arctan(Q^{-1}) \approx  0.0064$ to attain $|t^{-\beta} - \sum_{j=1}^{\nexp} w_j e^{-s_j t}| \le \varepsilon t^{-\beta}$ for $\delta \le t \le T$, with the final instant $T \le 10$ and temporal gap $\delta \ge 5\times 10^{-3}$.\label{SOE_table_Q100}}
 \begin{lrbox}{\tablebox}
  \begin{tabular}{cc|cc}
  \hline\hline
$s_j$ 	&	$\omega_j$	&	$s_j$ 	&	$\omega_j$	\\
\hline
 \multicolumn{2}{c|}{$\varepsilon = 10^{-2}$, $\nexp = 2$} 	&  \multicolumn{2}{c}{$\varepsilon =  10^{-3}$, $\nexp = 5$} \\
\hline
      0.1717351273768353D-90  &0.9847183908720862D+00	&	      0.1717351273768353D-90  &0.9816124627012480D+00	\\
      0.3016163347415830D+00  &0.1946540375040938D-01	&	      0.1513258304760864D+00  &0.1749586258243257D-01	\\
	&	&	      0.1243840094780066D+01  &0.1190462163297126D-01	\\
	&	&	      0.8568483928967764D+01  &0.1322284392208735D-01	\\
	&	&	      0.7517155220295190D+02  &0.1506498116754216D-01	\\
\hline
 \multicolumn{2}{c|}{$\varepsilon = 10^{-4}$, $\nexp = 6$} 	&  \multicolumn{2}{c}{$\varepsilon =  10^{-5}$, $\nexp = 8$} \\
\hline
      0.1717351273768353D-90  &0.9813979844412281D+00	&	      0.1717351273768353D-90  &0.9796394795856220D+00	\\
      0.1444688547833037D+00  &0.1725070612123545D-01	&	      0.1056560707574286D+00  &0.1667886723036181D-01	\\
      0.1062521612149110D+01  &0.1054578706877307D-01	&	      0.6560288935303707D+00  &0.8844634987341105D-02	\\
      0.5267690362852202D+01  &0.1032992162487232D-01	&	      0.2306291959214068D+01  &0.7630982018328110D-02	\\
      0.2703967558153169D+02  &0.1094329897487313D-01	&	      0.7494069882130154D+01  &0.7643894387121303D-02	\\
      0.1513162921362690D+03  &0.1168001387206569D-01	&	      0.2475380224244508D+02  &0.7895263738883522D-02	\\
	&	&	      0.8455768375945300D+02  &0.8193476657837220D-02	\\
	&	&	      0.3014062855622324D+03  &0.8669821820757132D-02	\\
\hline	
	 \multicolumn{2}{c|}{$\varepsilon = 10^{-6}$, $\nexp = 10$} 	&  \multicolumn{2}{c}{$\varepsilon =  10^{-7}$, $\nexp = 12$} \\
\hline
      0.1717351273768353D-90  &0.9782726238548010D+00	&	     0.1717351273768353D-90  &0.9771544730076503D+00	\\
      0.8372863302094305D-01  &0.1643118013623608D-01	&	      0.6947921924972357D-01  &0.1629754539058215D-01	\\
      0.4855505122592347D+00  &0.8148771417133997D-02	&	      0.3895637403085457D+00  &0.7808468229375633D-02	\\
      0.1469355573407596D+01  &0.6388908896360456D-02	&	      0.1094368348489869D+01  &0.5750937211655103D-02	\\
      0.3830248867072182D+01  &0.6035304461848845D-02	&	      0.2532138118439746D+01  &0.5133367812236056D-02	\\
      0.9747221022108484D+01  &0.6067666015209520D-02	&	      0.5541693247062111D+01  &0.5003442742181595D-02	\\
      0.2501534578981606D+02  &0.6195686630251158D-02	&	      0.1201464446288306D+02  &0.5031093636773381D-02	\\
      0.6524253082974616D+02  &0.6348046228383885D-02	&	      0.2617235093833228D+02  &0.5105593812435561D-02	\\
      0.1733585387699787D+03  &0.6525047468028670D-02	&	      0.5751774621815093D+02  &0.5194450859496128D-02	\\
      0.4760435514260730D+03  &0.6968897297155060D-02	&	      0.1276960094034739D+03  &0.5292138293131431D-02	\\
	&	&	      0.2871575858061491D+03 & 0.5429770443927676D-02	\\
	&	&	      0.6667008790524359D+03  &0.5876242926897542D-02	\\
\hline\hline
 \end{tabular}
\end{lrbox}
\scalebox{0.8}{\usebox{\tablebox}}
\end{table}

\end{document}